\documentclass[reqno,12pt]{amsart}
\usepackage[utf8]{inputenc}
\usepackage[left=1.2in, right=1.2in, top=1in]{geometry}

\usepackage{esint}
\usepackage{amsmath,amssymb,amsthm,mathrsfs,color,times,textcomp,verbatim,mathtools}
\allowdisplaybreaks
\usepackage{xcolor}
\usepackage[colorlinks=true]{hyperref}
\hypersetup{urlcolor=blue, citecolor=red, linkcolor=blue}
\usepackage[square,numbers]{natbib}

\numberwithin{equation}{section}
\theoremstyle{plain}

\newcommand{\pp}{ {\partial} }

\newcommand{\h}{\hat  }

\newcommand{\cuad}{{\sqcap\kern-.68em\sqcup}}

\newcommand{\be}{\begin{equation}}
\newcommand{\ee}{\end{equation}}

\theoremstyle{plain}
\newtheorem{theorem}{Theorem}[section]
\newtheorem{lemma}[theorem]{Lemma}
\newtheorem{prop}[theorem]{Proposition}

\newtheorem{remark}{Remark}[theorem]
\newcommand{\bremark}{\begin{remark} \em}
\newcommand{\eremark}{\end{remark} }

\title[heat equations with Neumann boundary conditions]{Finite time blow up solutions for heat equations with Neumann boundary conditions on $\mathbb{R}_{+}^{4}$}
\author[X. Fang]{Xiang Fang}
\address{\noindent School of Mathematics, Tianjin University, Tianjin 300072, P. R. China}
\email{fangx@tju.edu.cn}
\author[J. Wei]{Juncheng Wei}
\address{\noindent Department of Mathematics, Chinese University of Hong Kong, Shatin, NT, Hong Kong}
\email{wei@math.cuhk.edu.hk}
\author[Y. Zheng]{Youquan Zheng}
\address{\noindent School of Mathematics, Tianjin University, Tianjin 300072, P. R. China}
\email{zhengyq@tju.edu.cn}

\date{\today \,(Last Typeset)}
\begin{document}
\begin{abstract}
We consider the nonlinear heat equations with Neumann boundary conditions
$$
\begin{cases}
u_{t}=\Delta u & \text{in}\ \mathbb{R}_{+}^{4} \times(0, T) ,\\
-\frac{d u}{d x_{4}}(\tilde{x}, 0, t) \ =u^2(\tilde{x}, 0, t)&  \text{in}\  \mathbb{R}^{3} \times(0, T).
\end{cases}
$$
We establish the existence of a finite-time blow-up solution. Specifically, for any sufficiently small $T>0$ and any $k$ distinct points $q_{1},\dots,q_{k}\in \mathbb{R}^{3}$, there exists an initial datum $u_{0}$ such that the corresponding solution $u(x,t)$ blows up exactly at $q_{1},\dots,q_{k}$ as $t\nearrow T$.
Furthermore, when $t\nearrow T$, the solution admits the asymptotic profile
$$u(x,t)=\sum_{j=1}^{k}U_{\mu_{j}(t),\xi_{j}(t)}(x)+Z_0^*(x)+o(1)\quad \text{as}~ t\nearrow T,$$
where
$$U_{\mu_{j}(t),\xi_{j}(t)}(x):=\mu_{j}^{-1}(t) U\left(\frac{x-\xi_{j}(t)}{\mu_{j}(t)}\right),~ x\in \mathbb{R}_{+}^{4},$$
and $Z_{0}^{*}\in C_{0}^{\infty}(\mathbb{R}_{+}^{4})$ satisfying
$$Z_{0}^{*}(q_{j},0)<0\quad \text{for all}\ j=1,\dots,k.$$
Here, $U(y)$ denotes the harmonic extension to $\mathbb{R}_{+}^{4}$ of the positive radially symmetric solution $\widetilde{U}$ to the fractional Yamabe problem $(-\Delta)^{\frac{1}{2}} \widetilde{U} = \widetilde{U}^{2}$ in $\mathbb{R}^{3}$. For some constants $\beta_{j}>0$, the scaling parameters $\mu{j}(t)$ and the translation parameters $\xi_{j}(t)$ satisfy
$$\mu_{j}(t)=\beta_{j}\frac{|\log 2T|(T-t)}{|\log(T-t)|^{2}}(1 + o(1)) \to 0,~\xi_{j}(t)\to (q_{j},0)\quad \text{as} ~t\nearrow T.$$
\end{abstract}

\maketitle

{
  \hypersetup{linkcolor=}
  \tableofcontents
}

 \section{Introduction}
In this paper, we study the space-time nonlocal equation
\begin{equation}\label{1.2}
(\partial_{t}-\Delta)^{s}u=u^{\frac{n+2s}{n-2s}} \quad\text{in}\  \mathbb{R}^{n}\times(0,T).
\end{equation}
For $0<s<1$, the fractional heat operator $H^{s}u(x,t)=(\partial_{t}-\Delta)^{s}u(x,t)$ is defined for functions $u=u(x,t)$ in the parabolic H\"{o}lder space $C_{x,t}^{2s+\epsilon,s+\epsilon}(\mathbb{R}^{n+1})$ with $\epsilon>0$ by the pointwise integro-differential expression
\begin{equation}\label{1.5:2}
H^{s}u(x,t)=(\partial_{t}-\Delta)^{s}u(x,t)=\int_{0}^{\infty}\int_{\mathbb{R}^{n}}(u(x,t)-u(x-z,t-\tau))K_{s}(\tau,z)dzd\tau,
\end{equation}
where the kernel is given by
$$K_{s}(\tau,z)=\frac{1}{(4\pi)^{\frac{n}{2}}\left|\Gamma(-s)\right|}\frac{e^{-\frac{|z|^{2}}{4\tau}}}{\tau^{\frac{n}{2}+1+s}},\quad \tau>0,~z\in\mathbb{R}^{n},$$
and $\Gamma$ denotes the Gamma function. We focus specifically on the case $s=\frac{1}{2}$, $n=3$ in \eqref{1.2}.

\medskip

Problem \eqref{1.2} is a special case of the general nonlocal fractional space-time equation.
\begin{equation}\label{1.8:2}
(\partial_{t}-\Delta)^{s}u(x,t)=f(x,t),\quad 0<s<1\quad \text{in}\ \mathbb{R}^{n}\times (0,\infty).
\end{equation}
Equation \eqref{1.8:2} serves as a fundamental example of a space-time nonlocal equation with order $s$ in time and $2s$ in space, analogous to the role of the fractional Laplacian in purely spatial nonlocal problems.

\medskip

The nonlocality of the operator $(\partial_t - \Delta)^s$ rules out the direct application of classical localization techniques, such as multiplication by compactly supported test functions and integration by parts. This challenge is effectively resolved by the extension technique introduced by Stinga and Torrea \cite{extension-problem2017}. They established a regularity theory for solutions to the fractional nonlocal equation \eqref{1.8:2}. Using the parabolic semigroups method, they rigorously justified the pointwise intefro-differential formula \eqref{1.5:2} for $(\partial_{t}-\Delta)^{s}$ and derived an associated maximum principle, and characterized the nonlocal operator via a parabolic extension equation (see \cite[Theorem 1.7]{extension-problem2017}). This extension transforms the $n$-dimensional nonlocal problem into a local boundary value problem in the $(n+1)$-dimensional upper half-space.
The key advantage lies in the applicability of numerous numerical and analytical methods for partial differential equations to this extension equation, thereby enabling us to derive estimates and properties of the solution $u(x,t)$ to the original problem. Additionally, they established parabolic interior and boundary Harnack inequalities, an Almgren-type monotonicity formula, and fractional parabolic H\"{o}lder ad Schauder estimates. We refer readers to \cite{AryaBanerjeeDanielliGarofaloJFA}, \cite{AudritoTerracini}, \cite{BanerjeeDanielliGarofaloCVPDE}, \cite{BanerjeeGonzaloSire}, \cite{BanerjeeGarofalo}, \cite{BanerjeeGarofaloMuniveCCM}, \cite{HyderSegattiSire}, \cite{LaiLinRuland} and the references therein for regularity, singularity as well as unique continuity property for the fractional heat operator $(\partial_{t}-\Delta)^{s}$.

\medskip

It is worth noting that in equation \eqref{1.2}, the random jumps are coupled with the random waiting times. To address this coupling, we reformulate the fractional nonlocal problem \eqref{1.2} as a local extension problem \eqref{1.3} in one higher spatial dimension.
Indeed, $s=\frac{1}{2}$ and $n=3$, the parabolic extension technique established in \cite[Theorem 1.7]{extension-problem2017} shows that \eqref{1.2} is equivalent to
\begin{equation}\label{1.3}
\begin{cases}
u_{t}=\Delta u &  \text{ in }\  \mathbb{R}_{+}^{4} \times(0, T) ,\\
-\frac{d u}{d x_{4}}(\tilde{x}, 0, t) \ ={u}^2(\tilde{x}, 0, t)&   \text { in }\  \mathbb{R}^{3} \times(0, T) ,
\end{cases}
\end{equation}
where $x_{4}>0$ is an additional spatial variable. This corresponds to a critical heat equation with Neumann boundary conditions.

\medskip

In this paper, we establish the existence and explicit form of a finite-time blow-up solution to \eqref{1.3}. A smooth solution of \eqref{1.3} blows up at time $T$ if
$$\lim_{t\to T}\|u(\cdot,t)\|_{L^{\infty}(\mathbb{R}_{+}^{4})}=+\infty.$$
Recall that all positive solutions of the fractional Yamabe problem
$$(-\Delta)^{s} u=u^{\frac{n+2s}{n-2s}}(y)\quad \text{in}~ \mathbb{R}^{n},$$
are given by
\begin{equation}\label{Yamabe}
\widetilde{U}_{\mu_{j}(t),\tilde{\xi}_{j}(t)}=\mu_{j}^{-\frac{n-2s}{2}}(t)\widetilde{U}\left(\frac{x-\tilde{\xi}_{j}(t)}{\mu_{j}(t)}\right)
=\alpha_{n,s}\left(\frac{\mu_{j}(t)}{\mu_{j}^{2}(t)+|\tilde x-\tilde{\xi}_{j}(t)|^{2}}\right)^{\frac{n-2s}{2}},
\end{equation}
where $\alpha_{n,s}$ is a constant depending only on $n$ and $s$, $\mu_{j}(t)$ and $\tilde{\xi}_{j}(t)$ are the corresponding scaling parameters. Here, the fractional Laplace operator $(-\Delta)^{s}u(x)$ for $x\in \mathbb{R}^{n}$ is defined as
$$
(-\Delta)^{s}u(x):=C(n,s)P.V.\int_{\mathbb{R}^{n}}\frac{u(x)-u(y)}{|x-y|^{n+2s}}dy,
$$
where $C(n,s)$ is a positive normalizing constant, $P.V.$ denotes convergence in the sense of the Cauchy principal value. For further details on the fractional Laplace operator, we refer to \cite{Davila-del-Pino-Dipierro-Valdinoci2015} and \cite{DiNezza-Palatucci-Valdinoci2012}.

Substituting $s=\frac{1}{2}$ and $n=3$ into \eqref{Yamabe}, we obtain
$$\widetilde{U}_{\mu(t), \tilde{\xi}(t)}(\tilde{x}) =\mu^{-1}(t) \widetilde{U}\left(\frac{\tilde{x}-\tilde{\xi}(t)}{\mu(t)}\right)=\alpha_0 \frac{\mu(t)}{\mu^{2}(t)+|\tilde{x}-\tilde{\xi}(t)|^{2}},$$
which solves the $\frac{1}{2}$-order Yamabe problem
$$
(-\Delta)^{\frac{1}{2}} \widetilde{U}=\widetilde{U}^2  \quad\text {in} \ \mathbb{R}^{3},
$$
with $\alpha_0$ a constant. It is well known that the harmonic extension of $\widetilde{U}$ to $\mathbb{R}_{+}^{4}$ is given by
$$
U\left(\tilde{y}, y_{4}\right):=U(y)=(P_{y_{4}} * \widetilde{U})(\tilde{y}),
$$
where $P_{y_{4}}\left(\tilde{y}\right)$ is the Poisson kernel for $(-\Delta)^{\frac{1}{2}}$ in $\mathbb{R}^{3}$, defined as
\begin{equation}\label{1.4:2}
P_{y_{4}}\left(\tilde{y}\right)=\frac{\Gamma(2)}{\pi^{2}} \frac{y_{4}}{\left(|\tilde{y}|^{2}+\left|y_{4}\right|^{2}\right)^{2}}, \quad\tilde{y} \in \mathbb{R}^{3},~ y_{4}>0,
\end{equation}
and $*$ denotes the convolution operator. Moreover, $U\left(\tilde{y}, y_{4}\right)$ solves
\begin{equation*}
\begin{cases}
\Delta U=0 & \text   { in }\ \mathbb{R}_{+}^{4}  ,\\
-\frac{d U}{d y_{4}}(\tilde{y}, 0) \ =U^2(\tilde{y}, 0)& \text     { in }\ \mathbb{R}^{3}.
\end{cases}
\end{equation*}
As shown in \cite[Section 2.2]{bubble2021}, the bubble profile takes the form
\begin{align}\label{Jul26-18}
U(y):=\frac{\alpha_0}{\left|\tilde{y}\right|^2+\left(1+y_4\right)^2},\quad y\in \mathbb{R}_{+}^{4},
\end{align}
where $\alpha_0=2$. Writing $\xi(t)=(\tilde{\xi}(t),0)$ with $\tilde{\xi}(t)\in \mathbb{R}^{3}$, we define the rescaled bubble as
\begin{equation}\label{1.4}
U_{\mu(t), \xi(t)} :=\mu^{-1}(t) U\left(\frac{x-\xi(t)}{\mu(t)}\right)=\mu^{-1}(t)\frac{\alpha_0}{\left|\frac{\tilde{x}-\tilde{\xi}(t)}{\mu(t)}\right|^2+\left(1+\frac{x_4}{\mu(t)}\right)^2},\quad x=(\tilde{x},x_4)\in \mathbb{R}_{+}^{4}.
\end{equation}

\medskip

Throughout the paper, for $k$ distinct points $q_{j}\in\mathbb{R}^{3},\ j=1,\dots,k$, where $k\in \mathbb{Z}_{+}$, we assume that a function $Z_{0}^{*}\in C_{0}^{\infty}(\mathbb{R}_{+}^{4})$ is chosen such that
$$Z_{0}^{*}(q_{j},0)<0\quad \text{for all}\ j=1,\dots,k.$$
The main result is stated as follows.
\begin{theorem}\label{theorem1}
For any sufficiently small $T>0$, there exists an initial datum $u_{0}$ such that the solution $u(x,t)$ to problem \eqref{1.3} blows up at the $k$ points $q_{1},\dots,q_{k}$ as $t\nearrow T$. Specifically, the solution admits the form
\begin{equation}\label{1.5}
u(x,t)=\sum_{j=1}^{k}U_{\mu_{j}(t),\xi_{j}(t)}(x)+Z_{0}^{*}(x)+\Theta(x,t),
\end{equation}
where $U_{\mu_{j}(t),\xi_{j}(t)}(x)$ is the function defined in \eqref{1.4}, and the parameters satisfy
$$ \mu_{j}(t)\to 0,\quad \xi_{j}(t)=(\tilde{\xi}_{j}(t),0)\to \hat{q}_{j}=(q_j,0)\quad \text{as}\quad t\nearrow T.$$
Moreover, there exist constants $c>0$ and $\beta_{j}>0, j=1,\dots,k$ such that
$$\|\Theta\|_{L^{\infty}}\leq T^{c},\quad \mu_{j}(t)=\beta_{j}\frac{|\log 2T|(T-t)}{|\log(T-t)|^{2}}(1 + o(1)).$$
\end{theorem}

\medskip

In recent years, the fractional heat equation with the Sobolev critical exponent
\begin{equation}\label{1.9}
\begin{cases}
u_{t}=-(-\Delta)^{s}u+u^{\frac{n+2s}{n-2s}}&\text{in}\ \mathbb{R}^{n}\times(0,T),\\
u(\cdot,0)=u_{0}& \text{in}\ \mathbb{R}^{n}
\end{cases}
\end{equation}
for $0<s<1$ and $0<T\leq \infty$ has attracted much attention.
Extensive research has focused on equation \eqref{1.9} and related problems, see \cite{heat-kernel-estimates2010, regularity-theory2011, Caffarelli-Figalli2013,  Caffarelli-Soria-Vazquez2013,Caffarelli-Vazquez2011,Caffarelli-Vasseur2010,Chen-Wei-Zhou2020,Chen-Kim-Song2010,Felsinger-Kassmann2013, Fang-Tan2018,Fernandez-Real-Ros-Oton2016,Musso-Sire-Wei-Zheng-Zhou2019,Silvestre2012,Silvestre-2-2012} and references therein.

\medskip

When $s=1$, equation \eqref{1.9} reduces to the classical critical parabolic equation
\begin{equation}\label{1.3:2}
\begin{cases}
u_{t}=\Delta u+u^{\frac{n+2}{n-2}} & \text{in}\ \mathbb{R}^{n}\times(0,T),\\
u(\cdot,0)=u_{0}& \text{in}\ \mathbb{R}^{n}.
\end{cases}
\end{equation}
For a smooth positive initial datum $u_{0}$, problem \eqref{1.3:2} admits a unique classical solution $u(x,t)$. If $(0,T)$ is the maximum existence interval of $u(x,t)$ with $T<\infty$, then $u(x,t)$ blows up at time $T$. Finite-time blow-up is commonly classified into two types. Letting $p=\frac{n+2}{n-2}$, we define
$$\text{Type}~I:\lim\sup_{t\to T}(T-t)^{\frac{1}{p-1}}\|u(\cdot,t)\|_{L^{\infty}(\mathbb{R}^{n})}<\infty,$$
$$\text{Type}~II:\lim\sup_{t\to T}(T-t)^{\frac{1}{p-1}}\|u(\cdot,t)\|_{L^{\infty}(\mathbb{R}^{n})}=\infty.$$
Type I behaves similarly to the ODE $u_{t}=u^{p}$, where nonlinearity dominates.
Type II blow up is more intricate and challenging to analyze mathematically, owing to the delicate interplay of diffusion, nonlinearity, and domain geometry.

\medskip

Singularity formation, especially the blow-up phenomenon for $n\geq3$ in problem \eqref{1.3:2} and related problems, for instance in smooth bounded domain $\Omega\in\mathbb{R}^{n}$, has been extensively studied in \cite{Cortazar-del-Pino-Musso2020,del-Pino-Musso-Wei2019,del-Pino-Musso-Wei2020,del-Pino-Musso-Wei2021,delpino-Musso-Wei-Zhang-Zhou2020,delPino-Musso-Wei-Zhou2020,Fila-King2012,Li-Sun-Wang2022,Schweyer2012} and references therein.
It is usually expected that under suitable conditions, a blow-up solution admits the asymptotic profile
\begin{equation}\label{form}
u(x,t)\sim \sum_{j=1}^{k}\mu_{j}^{-\frac{n-2}{2}}(t)W\left(\frac{x-\xi_{j}(t)}{\mu_{j}(t)}\right),
\end{equation}
where $\mu_{j}(t)$ and $\xi_{j}(t)$ are scaling and translation parameters with $\mu_{j}(t)\to 0$ as $t\to T$, and $W$ is the positive solution of the Yamabe problem
$$
\Delta W+W^{\frac{n+2}{n-2}}=0\quad \text{in}~ \mathbb{R}^{n}.
$$
For $n\geq3$, Filippas, Herrero and Vel\'{a}zquez \cite{Filippas-Herrero-Velazquez2000} showed that Type II blow up can occur in the class of positive, radially symmetric and monotonically decreasing solutions.
Matano and Merle \cite{Matano-Merle2004} later extended this result without monotonicity assumptions.
Furthermore, in \cite{Filippas-Herrero-Velazquez2000}, the authors formally obtain a Type II blow-up to \eqref{1.3:2} for $n=3,4,5,6$ by using the matched
asymptotic expansion approach, which were rigorously confirmed in \cite{ del-Pino-Musso-Wei2019, delpino-Musso-Wei-Zhang-Zhou2020, delPino-Musso-Wei-Zhou2020, Harada2020, Harada-6D2020, Li-Sun-Wang2022,Schweyer2012}.

\medskip

On the other hand, Galaktionov and V\'{a}zquez \cite{Galaktionov-Vazquez1997} showed that any radial solution $u(x,t)$ of \eqref{1.3:2} in the unit ball $B(0,1)$ must blow up in infinite time, i.e.,
$$\lim_{t\to \infty}\|u(\cdot,t)\|_{L^{\infty}(B(0,1))}=\infty.
$$
Galaktionov and King \cite{Galaktionov-King2003} constructed positive, radially symmetric infinite-time blow up solutions in $B(0,1)$ that satisfy the asymptotic form \eqref{form} as $t\to \infty$ with $k=1$ and $\mu_{1}(t)\sim t^{-\frac{1}{n-4}}$ for $n\geq5$.
In the non-radial setting, Cort\'{a}zar del Pino and Musso \cite{Cortazar-del-Pino-Musso2020} constructed positive infinite-time blow up solutions to problem \eqref{1.3:2} for $n\geq5$.
Fila and King \cite{Fila-King2012} studied problem \eqref{1.3:2} for positive, radially symmetric $u_{0}$ with exact power decay rate, formally deriving growth and decay rates via matched asymptotics.
They further conjectured that infinite-time blow up occurs only in dimensions $3$ and $4$ (see \cite[Conjecture 1.1]{Fila-King2012})—a claim later proved rigorously for $n=3$ by del Pino, Musso and Wei \cite{del-Pino-Musso-Wei2020}.

\medskip

del Pino, Musso, Wei, and Zhou \cite{delPino-Musso-Wei-Zhou2020} establish the existence of Type II finite-time blow-up solutions for \eqref{1.3:2} in $\mathbb{R}^4$. Their result guarantees that for any prescribed finite set of points and sufficiently small blow-up time $T>0$, there exists an initial datum leading to a solution that blows up exactly at those points with the sharp rate $\|u(\cdot,t)\|_{L^{\infty}}\sim (T-t)^{-1}|\log(T-t)|^2$, the same blow-up rate as that observed for the 2-dimensional harmonic map flow \cite{Davila-del-Pino-Wei2020, del-Pino-Musso-Wei2019}. The solution exhibits a multi-bubbling profile, asymptotically described by a sum of rescaled Aubin–Talenti bubbles with scaling parameters $\mu_j(t)\sim (T-t)|\log(T-t)|^{-2}$. The construction relies on a rigorous inner–outer gluing method and deals with a nonlocal effective dynamics for the scaling parameters. Theorem \ref{theorem1} can be viewed as counterpart of \cite{delPino-Musso-Wei-Zhou2020} to problem \eqref{1.3}.

\medskip

The proof of Theorem \ref{theorem1} is based on the parabolic inner-outer gluing method, as developed in \cite{Cortazar-del-Pino-Musso2020} and \cite{Davila-del-Pino-Wei2020}. This method has proven to be a powerful tool for constructing solutions to various parabolic problems, such as singularity formation in harmonic map flow \cite{Davila-del-Pino-Wei2020,Sire-Wei-Zheng2023,Wei-Ye-Zhang-Zeng} and infinite-time blow-up for energy-critical heat equations \cite{Cortazar-del-Pino-Musso2020,del-Pino-Musso-Wei2020}, among others.
The main difficulty in studying singularity formation for energy-critical heat equations with Neumann boundary conditions is the lack of symmetry, which renders the ODE techniques inapplicable. Inspired by \cite{Sire-Wei-Zheng2023,Wei-Ye-Zhang-Zeng}, we develop a linear theory that is independent of symmetry, thus completing the proof of Theorem \ref{theorem1}.

\medskip

This paper is organized as follows. In Section \ref{2}, we construct an approximate solution by combining Aubin–Talenti bubbles with a nonlocal correction term $\Psi_0$, and carefully estimate its error. Section \ref{25-inner-outer} introduces the inner–outer gluing scheme, decomposing the problem into an inner parabolic linearized equation and an outer heat equation with Neumann conditions. In Section \ref{3}, we determine the leading-order behavior of the parameters $\mu(t)$ and $\xi(t)$ via a reduction procedure, deriving the blow-up rate $\mu(t)\sim (T-t)|\log(T-t)|^{-2}$. Section \ref{5} establishes a linear theory for the outer problem, providing weighted estimates for the heat equation. In Section \ref{linear-inner}, we develop a linear theory for the inner problem, constructing fast-decay solutions to the linearized equation around $U_{\mu,\xi}$ by combining methods from \cite{Cortazar-del-Pino-Musso2020} and \cite{Davila-del-Pino-Wei2020}, with special attention to the nonlocal structure of the operator $\sqrt{\partial_t-\Delta}$. Finally, Section \ref{6} solves the full inner–outer gluing system via a fixed-point argument, proving the existence of Type II blow-up solutions.

\medskip

{\bf Notations:}
\begin{enumerate}
\item[(1)] Throughout the paper, we use the symbol "$\lesssim$" to denote "$\leq C$" for a positive constant $C$ independent of $t$ and $T$, and $C$ may change from line to line.

\item[(2)] $f\asymp g$ means that there exist constants $c$ and $C$, such that $cg\leq f\leq Cg$.

\item[(3)] In general, for any $x=(x_{1},x_{2},x_{3},x_{4})\in\mathbb{R}_{+}^{4}$, we denote its projection onto $\mathbb{R}^{3}$ by $\tilde{x}=(x_{1},x_{2},x_{3})$.

\item[(4)] $\frac{d u}{d x_{4}}(\tilde{x}, 0, t)$ means $\frac{d u(\tilde{x},x_{4}, t)}{d x_{4}}\Big|_{(\tilde{x}, 0, t)}$.
\end{enumerate}

\section{Approximate solutions and error estimates}\label{2}
In this section, we construct an approximate solution to \eqref{1.3} and carefully estimate its residual error in both the interior and the boundary sense. For simplicity, we focus primarily on the one bubble case, modifications for multi-bubble case will be addressed at the end of this paper.

Define the error operators as
\begin{align*}
&S^{i n}(u):=-u_{t}+\Delta u \qquad\qquad\qquad\ \text { in } \ \mathbb{R}_{+}^{4} \times(0, T),\\
&S^{b}(u):=\frac{d u}{d x_{4}}(\tilde{x}, 0, t)+u^2(\tilde{x}, 0, t)\quad\text { in } \ \mathbb{R}^{3} \times(0, T).
\end{align*}
Solving system \eqref{1.3} is equivalent to finding $u$ satisfying
$$S^{i n}(u)=0\quad \text{and}\quad S^{b}(u)=0.$$

It is a well-established result that the linearized operator around the bubble
$$L_{0}[\phi]:=-(-\Delta)^{\frac{1}{2}} \phi+2 \widetilde{U}(\tilde{y}) \phi$$
is non-degenerate, in the sense that all bounded solutions to $L_{0}[\phi]=0$
in $\mathbb{R}^{3}$ are linear combinations of the functions $\widetilde{Z}_{1},\dots,\widetilde{Z}_{4}$, given by
$$\widetilde{Z}_{i}(\tilde y)=\partial_{\tilde{y}_{i}}\widetilde{U}(\tilde{y}), \quad i=1, 2, 3, \quad \widetilde{Z}_{4}(y)=\widetilde{U}(\tilde{y})+\tilde{y} \cdot \nabla \widetilde{U}(\tilde{y}).$$

We now define the extension operator $L_U$ in the half-space as follows
\begin{align}\label{Jul26-10}
L_{U}[\phi]:=
\begin{cases}
\Delta_{y}\phi  & \text{in}\ \mathbb{R}_{+}^{4} \times(0, T)  ,\\
-\frac{d\phi}{d y_{4}}(\tilde{y},0,t) -2U(\tilde{y}, 0)\phi(\tilde{y},0,t)   &\text{in}\ \mathbb{R}^{3}\times(0, T).
\end{cases}
\end{align}
By the Caffarelli-Silvestre extension technique \cite{extension2007}, the Poisson convolution
$$
Z_{i}(y)=(P_{y_{4}} * \widetilde{Z}_{i})(\tilde{y}),\quad i=1,\dots,4,
$$
with $P_{y_{4}}$ as the Poisson kernel for $(-\Delta)^{\frac{1}{2}}$, as defined in \eqref{1.4:2}, yields a basis of the bounded solution space to $L_U[\phi] = 0$.
These functions can be alternatively expressed as
 \begin{align}\label{Jul26-19}
    Z_{i}(y)=\partial_{y_{i}} U(y), \ i=1, 2, 3, \quad Z_{4}(y)= U(y)+y \cdot \nabla U(y),
 \end{align}
and they satisfy $L_U[Z_i] = 0$ in the sense of the system \eqref{Jul26-10}.

Our first approximation is the rescaled bubble
\begin{equation*}
U_{\mu(t), \xi(t)}=\mu^{-1}(t) U\left(\frac{x-\xi(t)}{\mu(t)}\right)=\mu^{-1}(t)\frac{\alpha_0}{\left|\frac{\tilde{x}-\tilde{\xi}(t)}{\mu(t)}\right|^2+\left(1+\frac{x_4}{\mu(t)}\right)^2},\quad x\in \mathbb{R}_{+}^{4},
\end{equation*}
where $\mu(t)>0$ and $\xi(t)=(\tilde{\xi}(t),0)\in \mathbb{R}_{+}^{4}$ denote scaling and translation parameters to be determined later.
Direct computations yield
\begin{align}\label{Jul26-1}
S^{in}(U_{\mu(t),\xi(t)})
&= -\partial_{t}U_{\mu(t),\xi(t)} \notag
\\&= \mu^{-2}(t)\dot{\mu}(t)\left[U(y)+y\cdot\nabla U(y)\right]+ \mu^{-2}(t)\dot{\xi}(t)\cdot\nabla U(y)  \notag
\\&= \mu^{-2}(t)\dot{\mu}(t)\left(-\frac{\alpha_{0}}{|\tilde y|^2+(1+y_4)^2} + \frac{\alpha_{0}(2+2y_4)}{\left(|\tilde y|^2+(1+y_4)^2\right)^2}\right) \notag
\\&\quad + \mu^{-2}(t)\nabla_{y}U(y)\cdot\dot{\xi}(t),
\end{align}
and
\begin{align*}
S^{b}(U_{\mu(t),\xi(t)})
=\frac{d U_{\mu(t),\xi(t)}}{d x_4}(\tilde x,0)+ U^2_{\mu(t),\xi(t)}(\tilde x,0)=0,
\end{align*}
where $\tilde{y}=\frac{\tilde{x}-\tilde{\xi}(t)}{\mu(t)}$, $y_4=\frac{x_4}{\mu(t)}$. The slow-decaying error component in \eqref{Jul26-1} is
$$
\mathcal{E}_{0} = -\frac{\alpha_{0}\dot{\mu}(t)}{\left|\tilde{x}-\tilde{\xi}(t)\right|^2+\left(\mu(t) + x_4\right)^2} \approx  -\frac{\alpha_{0}\dot{\mu}(t)}{\rho^2},
$$
where $\rho:= |x -\xi(t)|$. To refine the approximation, we introduce the correction equation
\begin{equation}\label{Jul26-2}
\partial_{t}u_{1} = \Delta u_{1} + \mathcal{E}_{0} \quad \text{in} \quad \mathbb{R}^{4} \times (0,T).
\end{equation}
Following the analysis in \cite{Davila-del-Pino-Wei2020}, an explicit solution to \eqref{Jul26-2} is given by
$$
u_{1} = -\alpha_{0}\int_{-T}^{t} \dot{\mu}(s)k(\rho,t-s)  ds,
$$
where
\begin{equation}\label{Jul26-3}
k(\rho,t) := \frac{1 - e^{-\frac{\rho^{2}}{4t}}}{\rho^{2}}.
\end{equation}
Note that $u_{1}$  is radially symmetric about $\xi(t)=(\tilde{\xi}(t),0)$, we have $u_1(\tilde{x},x_4)=u_1(\tilde{x},-x_4)$ for $x_4\geq 0$, it follows that $u_1$ solves
\begin{equation}\label{Jul26-4}
\begin{cases}
\partial_{t}u_{1} = \Delta u_{1} + \mathcal{E}_{0} &~ \text{in} \ \mathbb{R}^{4}_+ \times (0,T),
\\ -\frac{d u_1}{d x_4}=0 & \text { in }\ \mathbb{R}^{3}.
\end{cases}
\end{equation}
To further regularize $u_1$, we choose a refined correction term $\Psi_0$ to be
\begin{equation}\label{Jul26-5}
\Psi_{0}(x,t) = -\alpha_{0}\int_{-T}^{t} \dot{\mu}(s)k(\zeta(\rho,t),t-s)  ds,
\end{equation}
where
$$
\zeta(\rho,t) = \sqrt{\rho^{2} + \mu^{2}(t)}.
$$
We now analyze the error introduced by $\Psi_0$. A direct computation gives
\begin{align*}
&\partial_{t}\Psi_{0} - \Delta\Psi_{0} - \mathcal{E}_{0}
\\&= \alpha_{0}\left[\frac{\mu y \cdot \dot{\xi} - \mu(t)\dot{\mu}(t)}{\zeta}\right]
\int_{-T}^{t} \dot{\mu}(s)k_{\zeta}(\zeta,t-s)  ds
\\&\quad + \alpha_{0}\int_{-T}^{t} \dot{\mu}(s)\bigg[{-k_{t}(\zeta,t-s)} + \frac{\rho^{2}}{\zeta^{2}}k_{\zeta\zeta}(\zeta,t-s)+ \frac{3}{\zeta}k_{\zeta}(\zeta,t-s) + \frac{\mu^{2}(t)}{\zeta^{3}}k_{\zeta}(\zeta,t-s)\bigg]  ds.
\end{align*}
Using the identity $-k_t + k_{\zeta\zeta} + \frac{3}{\zeta} k_\zeta = 0$, which follows from \eqref{Jul26-3}, the expression simplifies to
\begin{align} \label{Jul26-7}
\partial_{t}\Psi_{0} - \Delta\Psi_{0} - \mathcal{E}_{0}
&= \alpha_{0}\left[\frac{y \cdot \dot{\xi} - \dot{\mu}(t)}{(1+|y|^{2})^{\frac{1}{2}}}\right]
\int_{-T}^{t} \dot{\mu}(s)k_{\zeta}(\zeta,t-s)  ds \notag
\\&\quad + \frac{\alpha_{0}}{\mu(t)(1+|y|^{2})^{3/2}}
\int_{-T}^{t} \dot{\mu}(s) \left[-\zeta k_{\zeta\zeta}(\zeta,t-s) + k_{\zeta}(\zeta,t-s)\right]  ds \notag
\\&=: \mathcal{R}[\mu].
\end{align}
The corrected approximate solution is therefore
$$
u^{*}(x,t) = U_{\mu(t),\xi(t)}(x) + \Psi_{0}(x,t)\quad \text{in} \ \mathbb{R}_{+}^{4} \times(0, T),
$$
with residual error
\begin{align*}
S^{in}(u^{*}) &= S^{in}(U_{\mu(t),\xi(t)}) - \mathcal{E}_{0}- \mathcal{R}[\mu]
=: \mathcal{K}[\mu,\xi],
\\ S^{b}(u^{*}) &=\left(U_{\mu(t),\xi(t)} + \Psi_{0}\right)^2 - U_{\mu(t),\xi(t)}^2
\end{align*}
where
\begin{equation}\label{Jul26-8}
\mathcal{K}[\mu,\xi] := \frac{\mu^{-2}(t)\dot{\mu}(t)\alpha_{0}(2+2y_4)}{\left(|\tilde y|^2+(1+y_4)^2\right)^2}+ \mu^{-2}(t)\nabla U(y) \cdot \dot{\xi}(t) - \mathcal{R}[\mu],
\end{equation}
and $\mathcal{R}[\mu]$ is as defined in \eqref{Jul26-7}.

\section{The inner-outer gluing scheme}\label{25-inner-outer}

This section introduces the inner-outer gluing scheme, which decomposes the full problem into a system of coupled equations for an inner and an outer component. We then derive the precise form of this coupled system.

We seek a solution of the form
$$
u(x,t) = u^{*}(x,t) + w \quad \text{in} \ \mathbb{R}_{+}^{4} \times(0, T),
$$
where $w$ is a small perturbation comprising inner and outer components
$$w = \varphi_{in} + \varphi_{out}, \quad
\varphi_{in} = \mu^{-1}(t)\eta_{R}\phi(y,t), \quad
\varphi_{out} = \psi(x,t) + Z^{*}(x,t).
$$
Here, the cut-off function
$$
\eta_{R(t)}(y)=\eta_{0}\left(\frac{|y|}{R(t)}\right)\quad \text{for}~R(t)>0,
$$
is defined via a smooth function $\eta_{0}:\mathbb{R}\to [0,1]$ satisfying
\begin{equation}\label{estimate2.6}
\eta_{0}(s)=
\begin{cases}1, & s<1, \\
0, & s>2 .
\end{cases}\end{equation}
And $Z^{*}(x,t)$ is the unique solution to the initial-boundary value problem
\begin{equation*}
\begin{cases}
Z_{t}^{*}=\Delta Z^{*} &  \text { in } \  \mathbb{R}_{+}^{4} \times(0, T), \\
-\frac{d Z^{*}}{d x_{4}}(\tilde{x}, 0,t)=0 &  \text { in }\  \mathbb{R}^{3} \times(0, T), \\
Z^{*}(x, 0)=Z_{0}^{*}(x) &  \text { in } \ \mathbb{R}_{+}^{4}.
\end{cases}
\end{equation*}
Direct computation yields
\begin{equation}\label{error1}
\begin{aligned}
S^{in}\left(u^*+w\right)
&=\mathcal{K}[\mu,\xi]-\mu^{-2}\dot{\mu}\eta_R\phi-\mu^{-1}\phi\partial_t\eta_R-\mu^{-1}\eta_R\phi_t+\mu^{-2}\eta_R\nabla_y\phi\left(\dot{\xi}+\dot{\mu}y\right)
\\&\quad+\mu^{-3}\phi\Delta_y\eta_R+\mu^{-3}\eta_R\Delta\phi+2\mu^{-3}\nabla_y\phi\nabla_y\eta_R-\psi_t+\Delta\psi,
\end{aligned}
\end{equation}
and
\begin{equation}\label{error2}
\begin{aligned}
S^{b}\left(u^*+w\right)
&=\mu^{-2}\eta_{R}\frac{d \phi}{d y_{4}}(\tilde{y},0, t)+\mu^{-2}\phi(\tilde{y},0,t)\frac{d\eta_{R}}{d y_{4}}(\tilde{y},0,t)+\frac{d \psi}{d x_{4}}(\tilde{x},0,t)
\\&\quad+2\mu^{-2}U(\tilde y,0)\eta_R\phi(\tilde y,0,t)+2\mu^{-1}U(\tilde y,0)\left(\Psi_0+\psi+Z^*\right)(\tilde x,0,t).
\end{aligned}
\end{equation}
Thus, $u= u^{*}+ w$ solves the original equation  \eqref{1.3} if $(\phi(y, t),\psi(x,t))$ satisfies the following inner-outer gluing system
\begin{equation}\label{Jul-inner-problem}
\begin{cases}
\mu^{2} \phi_{t}=\Delta_{y} \phi+\mathcal{H}_1[\phi,\psi,\mu,\xi] &\text { in } \ B_{2R}^{+} \times(0, T),  \\
-\frac{d \phi}{d y_{4}}(\tilde{y}, 0, t)=2U(\tilde{y}, 0) \phi(\tilde{y}, 0, t)+\mathcal{H}_2[\phi,\psi, \mu, \xi]&\text { in }\ (\partial B_{2R}^{+}\cap\mathbb{R}^{3}) \times(0, T),
\end{cases}
\end{equation}
and
\begin{equation}\label{Jul-outer-problem}
\begin{cases}
\psi_{t}=\Delta_{x} \psi+ \mathcal{G}_1\left[\phi, \psi, \mu, \xi\right]& \text { in }\ \mathbb{R}_{+}^{4} \times(0, T), \\
 -\frac{d \psi}{d x_{4}}\left(\tilde{x}, 0, t\right)=\mathcal{G}_2\left[\phi, \psi, \mu, \xi\right]
& \text { in }\ \mathbb{R}^{3} \times(0, T),
\end{cases}
\end{equation}
where
\begin{align*}
B_{2 R}^{+}&:=\left\{y=(\tilde{y},y_{4})\mid \tilde{y}\in \mathbb{R}^{3},\ y_{4}\geq 0,\ |y|\leq 2R(t)\right\} ,
\\\mathcal{H}_1[\phi,\psi,\mu,\xi](y, t)&:=\mu\left[\dot{\mu}\left(\nabla_y\phi\cdot y+\phi\right)+\nabla_y\phi\cdot \dot{\xi}\right]+\mu^3\mathcal{K}[\mu,\xi],
\\\mathcal{H}_2[\phi,\psi,\mu,\xi](\tilde y,0,t)&:=2\mu U(\tilde y,0)\left(\Psi_0+\psi+Z^*\right)(\tilde x,0,t),
\end{align*}
and
\begin{align*}
\mathcal{G}_1\left[\phi,\psi,\mu,\xi\right](y,t):=&~\mu^{-3}\left[\phi\Delta_y\eta_R+2\nabla_y\phi\nabla_y\eta_R-\mu^{2}\phi\partial_t\eta_R\right]+(1-\eta_R)\mathcal{K}[\mu,\xi],
\\\mathcal{G}_2\left[\phi,\psi,\mu,\xi\right](\tilde{y},0,t):=&~2\mu^{-1}(1-\eta_R)U(\tilde y,0)\left(\Psi_0+\psi+Z^*\right)(\tilde x,0,t)
\\&~+\mu^{-2}\phi(\tilde{y},0,t)\frac{d\eta_{R}}{d y_{4}}(\tilde{y},0,t).
\end{align*}

\section{Formal derivation of $\mu_{0}(t)$ and $\xi_{0} (t)$ }\label{3}

In this section, we formally derive the leading-order behavior of the scaling parameter $\mu(t)$ and the translation parameter $\xi(t)$ as $t\nearrow T$. This is achieved by imposing orthogonality conditions against the kernel of the linearized operator.

At this point, we consider the linear system
\begin{equation}\label{Jul26-9}
\begin{cases}
\Delta_{y}\phi=-\mathcal{H}_1[\phi,\psi,\mu,\xi](y, t) &\text{in}\ \mathbb{R}_{+}^{4} \times(0, T)  ,\\
-\frac{d\phi}{d y_{4}}(\tilde{y},0,t) -2U(\tilde{y}, 0)\phi(\tilde{y},0,t) =\mathcal{H}_2[\phi,\psi,\mu,\xi](\tilde{y},0,t) &  \text{in} \ \mathbb{R}^{3}\times(0, T).
\end{cases}
 \end{equation}
Testing equation \eqref{Jul26-9} against the kernel functions $Z_{i}(y)$ and integrating over $\mathbb{R}{+}^{4}$ gives
\begin{align*}
0=&\int_{\mathbb{R}_{+}^{4}}\left(\Delta \phi+\mathcal{H}_1[\phi,\psi,\mu,\xi](y, t)\right) Z_{i}(y) d y
\\=&\int_{\mathbb{R}_{+}^{4}} \phi \Delta Z_{i}(y) d y+\int_{\mathbb{R}^{3}} Z_{i} \frac{\partial \phi}{\partial \nu} d \tilde{y}-\int_{\mathbb{R}^{3}} \phi \frac{\partial Z_{i}}{\partial \nu} d \tilde{y}+\int_{\mathbb{R}_{+}^{4}} \mathcal{H}_1[\phi,\psi,\mu,\xi](y, t) Z_{i}(y)dy
\\=&\int_{\mathbb{R}_{+}^{4}} \mathcal{H}_1[\phi,\psi,\mu,\xi](y, t) Z_{i}(y) d y+\int_{\mathbb{R}^{3}} \mathcal{H}_2[\phi,\psi,\mu,\xi](\tilde{y},0,t)Z_{i}(\tilde{y}, 0) d \tilde{y}.
\end{align*}
Therefore, the inner problem must satisfy the orthogonality conditions
\begin{align}\label{Oct9-1}
\int_{\mathbb{R}_{+}^{4}} \mathcal{H}_1[\phi,\psi,\mu,\xi](y, t) Z_{i}(y) d y+\int_{\mathbb{R}^{3}} \mathcal{H}_2[\phi,\psi,\mu,\xi](\tilde{y},0,t)Z_{i}(\tilde{y}, 0) d \tilde{y}=0,\quad i=1,\dots,4,
\end{align}
where $Z_i$ are the kernel functions of the linearized operator $L_U$ defined in \eqref{Jul26-10}.

We consider functions $\xi(t)=(\tilde{\xi}(t),0) \to \hat{q}=(q,0)$ and parameters $\mu(t) \to 0$ as $t \nearrow T$, decomposed as
\begin{equation}\label{error3:2}
\mu(t)=\mu_{0}(t)+\mu_{1}(t),\quad \xi(t)=\xi_{0}(t)+\xi^{1}(t),
\end{equation}
where $\mu_{0}(t)$ and $\xi_{0}(t)$ are the leading-order terms of $\mu(t)$ and $\xi(t)$, respectively, while $\mu_{1}(t)$ and $\xi^{1}(t)$ are lower-order remainders satisfying $\mu_{1}(t)\ll\mu_{0}(t)$ and $\xi^{1}(t)\ll\xi_{0}(t)$.
To determine the leading-order parameters $\mu_0(t)$ and $\xi_0(t)$ from \eqref{Oct9-1}, we extract the dominant terms in \eqref{Jul-inner-problem},
\begin{equation}\label{Jul-derivation1}
\begin{cases}
\mu_{0}^{2} \phi_{t}=\Delta_{y} \phi+\mathcal{H}_1^*[\mu,\xi]  &\text{ in } \  B_{2R}^{+} \times(0, T),  \\
-\frac{d \phi}{d y_4}(\tilde{y}, 0, t)=pU^{p-1}(\tilde{y}, 0) \phi(\tilde{y}, 0, t)+\mathcal{H}_2^*[\mu,\xi,\psi,Z^*]  & \text{ in }\ (\partial B_{2R}^{+}\cap\mathbb{R}^{3}) \times(0, T)
\end{cases}
\end{equation}
with
\begin{align*}
\mathcal{H}_1^*[\mu,\xi](y, t):=&~
\frac{\mu(t)\dot{\mu}(t)\alpha_{0}(2+2y_4)}{\left(1+|y|^2\right)^2}+ \mu(t)\nabla U(y) \cdot \dot{\xi}(t) -\mu^3(t)\mathcal{R}[\mu],
 \\
\mathcal{H}_2^*[\mu,\xi,\psi,Z^*](\tilde y,0,t):=&~2\mu U(\tilde y,0)\left(\Psi_0+\psi+Z^*\right)(\tilde x,0,t).
\end{align*}
The contributions from the residual terms $\mathcal{H}_i-\mathcal{H}_i^*$, ($i=1,2$) to the orthogonality conditions are negligible compared to those of the leading-order terms $\mathcal{H}_i$.

We thus approximate solutions to \eqref{Jul-derivation1} by requiring that $\mu_0(t)$ and $\xi_0(t)$ to satisfy
\begin{equation}\label{Jul-orthogonality}
\int_{\mathbb{R}_{+}^{4}} \mathcal{H}_1^*[\mu,\xi](y, t) Z_{i}(y) d y+\int_{\mathbb{R}^{3}}\mathcal{H}_2^*[\mu,\xi,\psi,Z^*](\tilde y,0,t) Z_{i}(\tilde{y}, 0)  d \tilde{y}=0, \quad i=1, \dots, 4.
\end{equation}

For $i=1,2, 3$, \eqref{Jul-orthogonality} imply
$$\dot{\xi}_0(t)=o(1),$$
where $o(1)\to0$ as $t\nearrow T$. Hence,
\begin{equation}\label{solution2}
\xi_{0}(t)=\hat{q}+o(1)\vec{v}
\end{equation}
for some vector $\vec{v}$, where $\hat{q}=(q,0)$ is a prescribed point in $\mathbb{R}_{+}^{4}$.

For $i=4$, we expand $Z^*(\mu y+\xi,t)$ and $\psi(\mu y+\xi,t)$ at the point $\hat{q}$,
$$Z^*(\mu y+\xi,t)=Z_0^*(\hat{q})+o(1),\quad \psi(\mu y+\xi,t)=\psi(\hat{q},0)+o(1).$$
Substituting into the orthogonality condition
\begin{equation*}
\int_{\mathbb{R}_{+}^{4}} \mathcal{H}_1^*[\mu,\xi](y, t) Z_{4}(y) d y+\int_{\mathbb{R}^{3}}\mathcal{H}_2^*[\mu,\xi,\psi,Z^*](\tilde y,0,t) Z_{4}(\tilde{y}, 0)  d \tilde{y}=0,
\end{equation*}
we obtain
\begin{align}\label{Jul26-14}
&\int_{\mathbb{R}_{+}^{4}} \left[\frac{\dot{\mu}(t)\alpha_{0}(2+2y_4)}{\left(1+|y|^2\right)^2}+ \nabla U(y) \cdot \dot{\xi}(t) -\mu^2(t)\mathcal{R}[\mu]\right] Z_{4}(y) d y \notag
\\&+\int_{\mathbb{R}^{3}}2U(\tilde y,0)\left[\Psi_0(x,t)+\psi(\hat{q},0)+Z_0^*(\hat{q})\right]Z_{4}(\tilde{y},0)  d\tilde{y}+o(1)=0.
\end{align}
We first calculate
\begin{align}\label{Jul26-15}
\int_{\mathbb{R}_{+}^{4}} \left[\frac{\dot{\mu}(t)\alpha_{0}(2+2y_4)}{\left(1+|y|^2\right)^2}+ \nabla U(y) \cdot \dot{\xi}(t)\right] Z_4(y) d y
=\alpha_0\dot{\mu}(t)\int_{\mathbb{R}_{+}^{4}}\frac{2+2y_4}{\left(1+|y|^2\right)^2} Z_{4}(y) d y.
\end{align}
Next, we examine the $\mathcal{R}[\mu]$ term
\begin{align*}
&\int_{\mathbb{R}_{+}^{4}}\mathcal{R}[\mu]Z_{4}(y) dy
\\&= -\alpha_{0}\dot{\mu}(t)\int_{\mathbb{R}^{4}_+}\frac{Z_{4}(y)}{(1+|y|^{2})^{\frac{1}{2}}}
\int_{-T}^{t}\dot{\mu}(s)k_{\zeta}(\zeta,t-s)ds dy
\\&\quad+\frac{\alpha_{0}}{\mu(t)}\int_{\mathbb{R}^{4}_+}\frac{Z_{4}(y)}{(1+|y|^{2})^{3/2}}
\int_{-T}^{t}\dot{\mu}(s)[k_{\zeta}(\zeta,t-s)-\zeta k_{\zeta\zeta}(\zeta,t-s)]dsdy
\end{align*}
Let
\begin{equation}\label{Jul26-11}
\Upsilon = \frac{\zeta^{2}}{t-s} = \frac{\mu^{2}(t)(1+|y|^{2})}{t-s}, \quad
\tau = \frac{\mu^{2}(t)}{t-s}
\end{equation}
and define $K(\Upsilon) = \frac{1-e^{-\frac{\Upsilon}{4}}}{\Upsilon}$. From \eqref{Jul26-3}, we derive
\begin{align*}
k_{\zeta}(\zeta,t-s) - \zeta k_{\zeta\zeta}(\zeta,t-s)
&= -4\left(\frac{\Upsilon}{t-s}\right)^{3/2}K_{\Upsilon\Upsilon}(\Upsilon)
\end{align*}
and also
\begin{align*}
k_{\zeta}(\zeta,t-s)
= -\frac{2}{\zeta^{3}} + \frac{e^{-\frac{\zeta^{2}}{4(t-s)}}}{2\zeta(t-s)} + \frac{2e^{-\frac{\zeta^{2}}{4(t-s)}}}{\zeta^{3}}
= \frac{2\sqrt{\Upsilon}}{(t-s)^{3/2}}K_{\Upsilon}(\Upsilon).
\end{align*}
Substituting these yields
\begin{align} \label{Jul26-12}
\int_{\mathbb{R}^{4}}\mathcal{R}[\mu]Z_{4}(y)dy
&=-\frac{2\alpha_{0}\dot{\mu}(t)}{\mu(t)}\int_{\mathbb{R}^{4}_+}\frac{Z_{4}(y)}{1+|y|^{2}}
\left(\int_{-T}^{t}\frac{\dot{\mu}(s)}{t-s}\Upsilon K_{\Upsilon}(\Upsilon)ds\right)dy \notag \\
&\quad  -\frac{4\alpha_{0}}{\mu^{2}(t)}\int_{\mathbb{R}^{4}_+}\frac{Z_{4}(y)}{(1+|y|^{2})^{2}}
\left(\int_{-T}^{t}\frac{\dot{\mu}(s)}{t-s}\Upsilon^{2}K_{\Upsilon\Upsilon}(\Upsilon)ds\right)dy.
\end{align}

On the other hand, from \eqref{Jul26-5} and \eqref{Jul26-11}, we find
\begin{align}\label{Jul26-13}
\int_{\mathbb{R}^{3}}2U(\tilde y,0)\Psi_0(\tilde{x},0,t)Z_{4}(\tilde{y},0)  d\tilde{y}
= -2\alpha_{0}\int_{\mathbb{R}^{3}}U^{2}(\tilde y,0)Z_{4}(\tilde y,0)
\int_{-T}^{t}\frac{\dot{\mu}(s)}{t-s}K(\Upsilon)dsd\tilde{y}.
\end{align}
Combining \eqref{Jul26-14}, \eqref{Jul26-15} \eqref{Jul26-12} and \eqref{Jul26-13} yields
\begin{align}\label{Jul26-16}
&\alpha_0\dot{\mu}(t)\int_{\mathbb{R}_{+}^{4}}\frac{2+2y_4}{\left(1+|y|^2\right)^2} Z_{4}(y) d y
+2\alpha_{0}\dot{\mu}(t)\mu(t)\int_{\mathbb{R}^{4}_+}\frac{Z_{4}(y)}{1+|y|^{2}}
\left(\int_{-T}^{t}\frac{\dot{\mu}(s)}{t-s}\Upsilon K_{\Upsilon}(\Upsilon)ds\right)dy \notag
\\& +4\alpha_{0}\int_{\mathbb{R}^{4}_+}\frac{Z_{4}(y)}{(1+|y|^{2})^{2}}
\left(\int_{-T}^{t}\frac{\dot{\mu}(s)}{t-s}\Upsilon^{2}K_{\Upsilon\Upsilon}(\Upsilon)ds\right)dy \notag
\\&-2\alpha_{0}\int_{\mathbb{R}^{3}}U^{2}(\tilde y,0)Z_{4}(\tilde y,0)
\int_{-T}^{t}\frac{\dot{\mu}(s)}{t-s}K(\Upsilon)dsd\tilde{y}\notag
\\&+2\left[\psi(\hat{q},0)+Z_0^*(\hat{q})\right]\int_{\mathbb{R}^{3}}U(\tilde y,0)Z_{4}(\tilde{y},0)  d\tilde{y}+o(1)=0.
\end{align}

For the construction of a blow-up solution, the scaling parameter $\mu(t)$ must tend to $0$ as $t \nearrow T$. We therefore impose the condition
$$
\dot{\mu}(t) = o(1) \quad \text{as} \quad t\nearrow T.
$$
Thus, \eqref{Jul26-16} simplifies to
\begin{align}\label{Jul26-17}
&4\alpha_{0}\int_{\mathbb{R}^{4}_+}\frac{Z_{4}(y)}{(1+|y|^{2})^{2}}
\left(\int_{-T}^{t}\frac{\dot{\mu}(s)}{t-s}\Upsilon^{2}K_{\Upsilon\Upsilon}(\Upsilon)ds\right)dy \notag
\\&-2\alpha_{0}\int_{\mathbb{R}^{3}}U^{2}(\tilde y,0)Z_{4}(\tilde y,0)
\int_{-T}^{t}\frac{\dot{\mu}(s)}{t-s}K(\Upsilon)dsd\tilde{y} \notag
\\&= -2\left[\psi(\hat{q},0)+Z_0^*(\hat{q})\right]\int_{\mathbb{R}^{3}}U(\tilde y,0)Z_{4}(\tilde{y},0)  d\tilde{y}+o(1).
\end{align}
Now define the function
\begin{align*}
\Gamma(\tau) :=4\alpha_{0}\int_{\mathbb{R}^{4}_+}\frac{Z_{4}(y)}{(1+|y|^{2})^{2}}
\Upsilon^{2}K_{\Upsilon\Upsilon}(\Upsilon)dy
-2\alpha_{0}\int_{\mathbb{R}^{3}}U^{2}(\tilde y,0)Z_{4}(\tilde y,0)
K(\Upsilon)d\tilde{y},
\end{align*}
where $\Upsilon = \tau(1+|y|^{2})$. Then \eqref{Jul26-17} becomes
\begin{align}\label{Oct8-1}
\int_{-T}^{t}\frac{\dot{\mu}(s)}{t-s}\Gamma\left(\frac{\mu^{2}(t)}{t-s}\right)ds= -2\left[\psi(\hat{q},0)+Z_0^*(\hat{q})\right]\int_{\mathbb{R}^{3}}U(\tilde y,0)Z_{4}(\tilde{y},0)  d\tilde{y}+o(1).
\end{align}

Using the explicit forms of $U(\tilde{y},0)$ and $Z_{4}(\tilde{y},0)$ from \eqref{Jul26-18} and \eqref{Jul26-19}, we compute the asymptotic behavior
$$
\Gamma(\tau) = \begin{cases}
c_{*} + O\left(\tau\right) & \text{for } \tau < 1, \\
O\left(\frac{1}{\tau}\right)  & \text{for } \tau > 1,
\end{cases}
$$
with $c_{*}>0$ is a constant. Thus, \eqref{Oct8-1} reduces to
\begin{equation}\label{Jul26-20}
c_{*}\int_{-T}^{t-\mu^{2}(t)}\frac{\dot{\mu}(s)}{t-s}ds
= -2c_{0}[Z^{*}_{0}(\hat{q}) + \psi(\hat{q},0)] + o(1),
\end{equation}
where
$$
c_{0} := \int_{\mathbb{R}^{3}}U(\tilde{y},0)Z_{4}(\tilde{y},0)d\tilde{y} < 0.
$$
Since $\mu(t)$ decreases to $0$ as $t\nearrow T$, we require
$$
a_{*} := Z^{*}_{0}(\hat{q}) + \psi(\hat{q},0) < 0.
$$

A natural choice for the leading-order behavior of $\mu(t)$ is given by
\begin{equation}\label{Jul26-21}
\dot{\mu}(t) = -\frac{c}{|\log(T-t)|^{2}},
\end{equation}
where $c>0$ is a constant  that will be determined below. Indeed, substituting into the integral yields
\begin{align*}
\int_{-T}^{t-\mu^{2}(t)}\frac{\dot{\mu}(s)}{t-s}ds
&=\int_{-T}^{t-(T-t)}\frac{\dot{\mu}(s)}{t-s}ds
+\int_{t-(T-t)}^{t-\mu^{2}(t)}\frac{\dot{\mu}(t)}{t-s}ds
-\int_{t-(T-t)}^{t-\mu^{2}(t)}\frac{\dot{\mu}(t) - \dot{\mu}(s)}{t-s}ds
\\&= \int_{-T}^{t-(T-t)}\frac{\dot{\mu}(s)}{t-s}ds
   +\dot{\mu}(t)(\log(T-t) - 2\log\mu(t))
\\&\quad-\int_{t-(T-t)}^{t-\mu^{2}(t)}\frac{\dot{\mu}(t) - \dot{\mu}(s)}{t-s}ds
\\&= \int_{-T}^{t}\frac{\dot{\mu}(s)}{T-s}ds - \dot{\mu}(t)\log(T-t)+o(1).
\end{align*}
Define $\beta(t):=\int_{-T}^{t}\frac{\dot{\mu}(s)}{T-s}ds - \dot{\mu}(t)\log(T-t)$. From \eqref{Jul26-21}, we compute
$$
\log(T-t)\frac{d\beta}{dt}(t) = \frac{d}{dt}\left(-\log^{2}(T-t)\dot{\mu}(t)\right) = 0,
$$
which implies $\beta(t)$ is a constant. Thus, \eqref{Jul26-20} can be approximately solved by
\eqref{Jul26-21} with $c$ chosen as
$$-c\int_{-T}^{T}\frac{ds}{(T-s)|\log(T-s)|^{2}} = -\frac{3c_{0}\,a_{*}}{c_{*}}:=\kappa_{*}.$$
At main order, we thus obtain
$$\dot{\mu}(t) = \kappa_{*}\dot{\mu}_{0}(t)\quad \text{with}\quad \dot{\mu}_{0}(t) = -\frac{|\log 2T|}{|\log(T-t)|^{2}}.$$
Imposing the condition $\mu_{0}(T) = 0$, we obtain
\begin{align}\label{Aug1-1}
\mu_{0}(t) = \frac{|\log 2T|(T-t)}{|\log(T-t)|^{2}}(1 + o(1)) \quad \text{as} \quad t\nearrow T.
\end{align}

\section{The linear theory of the outer problem}\label{5}
In this section, we establish a priori estimates of the outer problem. We first consider the corresponding linear problem
\begin{equation}\label{outer1}
\begin{cases}
\psi_{t}=\Delta \psi+g_{1}(x,t)& \text{in}~ \mathbb{R}_{+}^4\times (0,T),\\
-\frac{d\psi}{dx_4}(\tilde{x},0,t)=g_{2}(\tilde{x},t) &\text{in}~  \mathbb{R}^3\times (0,T).
\end{cases}
\end{equation}
We decompose problem (\ref{outer1}) into
\begin{equation}\label{outer19}
\begin{cases}
\psi_{t}=\Delta\psi+g_{1}(x,t)& \text{in}~ \mathbb{R}_{+}^4\times (0,T),\\
-\frac{d\psi}{dx_4}(\tilde{x},0,t)=0 & \text{in}~ \mathbb{R}^3\times (0,T)
\end{cases}
\end{equation}
and
\begin{equation}\label{outer2}
\begin{cases}
\psi_{t}=\Delta\psi&\text{in}~ \mathbb{R}_{+}^4\times (0,T),\\
-\frac{d\psi}{dx_4}(\tilde{x},0,t)=g_{2}(\tilde{x},t)& \text{in}~ \mathbb{R}^3\times (0,T).\\
\end{cases}
\end{equation}

\subsection{Estimates for problem \eqref{outer19}}
We begin by estimating solutions to \eqref{outer19}. From \cite{heat-kernels2011}, the heat kernel of problem \eqref{outer19} is given by
\begin{equation}\label{outer20}
K_{t}(x,y):= \frac{1}{(4\pi t)^{2}}\left(e^{\frac{-|x-y|^{2}}{4t}}+e^{\frac{-|x-\bar{y}|^{2}}{4t}}\right),
\end{equation}
where $\bar {y}=(\tilde{y},-y_4)$ denotes the reflection of $y$ across the hyperplane $\{y_4=0\}$.
A solution $\psi$ to \eqref{outer19} can then be represented as
\begin{equation}\label{outer22:3}
\begin{aligned}\psi(x,t):&=\int_{0}^{t}\int_{\mathbb{R}^{4}_{+}}K_{t-s}(x,z)g_{1}(z,s)dzds
\\&=\int_{0}^{t}\int_{\mathbb{R}^{4}_{+}}\frac{1}{(t-s)^{2}}\left(e^{\frac{-|x-z|^{2}}{4(t-s)}}+e^{\frac{-|x-\bar{z}|^{2}}{4(t-s)}}\right)g_{1}(z,s)dzds.
\end{aligned}\end{equation}

We introduce the following weight functions
\begin{align}\label{Aug5-1}
\begin{cases}
\varrho_{1} := \mu_0^{\nu-3}(t) R^{-2-\alpha}(t) \chi_{\{|x - \xi(t)| \leq 2\mu_0R\}} \\
\varrho_{2} := \frac{\mu_0^{\nu_{2}}(t)}{|x-\xi(t)|^{2}} \chi_{\{|x-\xi(t)| \geq \mu_0R\}} \\
\varrho_{3} := 1
\end{cases}
\end{align}
where $ R(t) = \mu_0^{-\beta}(t) $ for $ \beta \in (0,1/2) $ throughout this paper.
We further define the norm for $g_1$ as
\begin{align}\label{Aug5-2}
\|g_1\|_{**}^{(1)} := \sup_{(x,t) \in \mathbb{R}^{4}_{+} \times (0,T)} \Big( \sum_{i=1}^{3} \varrho_{i}(x,t) \Big)^{-1} |g_1(x,t)|
\end{align}
For the solution $\psi$, we define a weighted norm
\begin{align}\label{Aug5-3}
\begin{aligned}
\|\psi\|_{*}^{(1)} &:= \left(\mu_0^{\nu-1}(0) R^{-\alpha}(0)+T^{\nu_2}|\log T|^{1-\nu_2}+T\right)^{-1} \|{\psi}\|_{L^{\infty}\left(\mathbb{R}^{4}_{+} \times (0,T)\right)} \\
&\quad + \left(\mu_0^{\nu-2}(0) R^{-1-\alpha}(0)+\mu_0^{\nu_2-1}(t)R^{-1}(t)+T^{\frac{1}{2}}\right)^{-1} \|{\nabla\psi}\|_{L^{\infty}\left(\mathbb{R}^{4}_{+} \times (0,T)\right)} \\
&\quad + \sup_{(x,t) \in \mathbb{R}^{4}_{+} \times (0,T)} \left[ \left(\mu_0^{\nu-1}(t) R^{-\alpha}(t)+(T-t)^{\nu_{2}}|\log T|^{\nu_{2}}|\log(T-t)|^{1-2\nu_{2}}\right.\right. \\
&\left.\left.\qquad\qquad\qquad\qquad+(T-t)|\log(T-t)|\right)^{-1} \left| \psi(x,t) - \psi(x,T) \right| \right].
\end{aligned}
\end{align}

We now establish the main estimate for problem \eqref{outer19}.
\begin{prop}\label{25Aug-prop5.1}
Let $\psi$ solve \eqref{outer19} with $\|g_1\|_{**}^{(1)} < +\infty$. Then
\begin{align*}
\|\psi\|_{*}^{(1)} \lesssim \|g_1\|_{**}^{(1)}.
\end{align*}
\end{prop}
Proposition \ref{25Aug-prop5.1} follows from the next three lemmas, each addressing different forms of the nonhomogeneous term $g_1(x,t)$.

\begin{lemma}\label{25Aug-lemma5.1}
Let $\psi$ solve problem \eqref{outer19} with
$$
|g_1(x,t)| \lesssim \mu_0^{\nu-3}(t)R^{-2-\alpha}(t)\chi_{\{|x-\xi(t)|\leq 2\mu_0R\}} , \quad \alpha,\nu \in (0,1).
$$
Then the following estimates hold
\begin{align}
|\psi(x,t)| &\lesssim \mu_0^{\nu-1}(0)R^{-\alpha}(0), \label{Aug5-4}
\\
|\psi(x,t) - \psi(x,T)| &\lesssim \mu_0^{\nu-1}(t)R^{-\alpha}(t), \label{Aug5-5}
\\
|\nabla\psi(x,t)| &\lesssim \mu_0^{\nu-2}(0)R^{-1-\alpha}(0). \label{Aug5-6}
\end{align}
\end{lemma}

\begin{lemma}\label{25-lemma5.2}
Assume that $\psi(\tilde{x},x_{4},t)$ solves \eqref{outer19} with the right hand side
\begin{align}\label{Aug2-4}
|g_1(x,t)|\lesssim \frac{\mu_0^{\nu_2}}{|x-\xi(t)|^2}\chi_{\{|x-\xi(t)|\geq \mu_0R\}}.
\end{align}
Then for sufficiently small $T>0$, the unique solution $\psi(\tilde{x},x_{4},t)$ of problem \eqref{outer19} satisfies
\begin{align*}
|\psi(x,t)|&\lesssim T^{\nu_2}|\log T|^{1-\nu_2},
\\|\psi(x,t)-\psi(x,T)|&\lesssim (T-t)^{\nu_{2}}|\log T|^{\nu_{2}}|\log(T-t)|^{1-2\nu_{2}},
\\ |\nabla\psi(x,t)|&\lesssim \frac{\mu_0^{\nu_{2}-1}(t)}{R(t)}.
\end{align*}
\end{lemma}

\begin{lemma}\label{25-lemma5.3}
Let $\psi$ solve problem \eqref{outer19} with right hand side
$$|g_1(x,t)| \lesssim 1.$$
Then it holds that
\begin{align*}
|\psi(x,t))| &\lesssim T,
\\|\psi(x,t)-\psi(x,T)| &\lesssim (T-t)|\log(T-t)|,
\\ |\nabla\psi(x,t)| &\lesssim T^{\frac{1}{2}}.
\end{align*}
\end{lemma}

Proposition \ref{25Aug-prop5.1} follows directly from Lemmas \ref{25Aug-lemma5.1}, \ref{25-lemma5.2}, and \ref{25-lemma5.3}.
The proofs of these lemmas are established via Duhamel's principle, following a similar approach to that presented in \cite[Section 15]{Davila-del-Pino-Wei2020}.
For brevity, we present complete arguments only for estimates \eqref{Aug5-4} and \eqref{Aug5-5}, omitting details for the remaining estimates.

\begin{proof}[Proof of \eqref{Aug5-4}.]
Using the heat kernel \eqref{outer20}, we bound
\begin{equation}\label{25Aug-outer23}
\begin{aligned}
\psi(x,t)
=&~\int_{0}^{t}\int_{\mathbb{R}^{4}_{+}}K_{t-s}(x,z)g_{1}(z,s)dzds
\\\lesssim&~ \int_{0}^{t}\frac{\mu_0^{\nu-3}(s)R^{-2-\alpha}(s)}{(t-s)^2}\int_{\{|z-\xi(s)|\leq 2\mu_0(s)R(s)\}\cap \mathbb{R}^{4}_{+}}e^{\frac{-|x-z|^{2}}{4(t-s)}}+e^{\frac{-|x-\bar{z}|^{2}}{4(t-s)}}
dzds.
\end{aligned}
\end{equation}
By symmetry, we focus on the integral containing $e^{-\frac{|x-z|^2}{4(t-s)}}$. Apply the change of variables
\begin{align*}
\hat{z} = \frac{z - \xi(s)}{2\sqrt{t-s}}, \quad \hat{x} = \frac{x - \xi(s)}{2\sqrt{t-s}},
\end{align*}
to obtain
\begin{equation}\label{outer24}
|\psi(x,t)| \lesssim \int_{0}^{t} \mu_0^{\nu-3}(s) R^{-2-\alpha}(s) \int_{|\hat{z}| \leq \mu_0(s) R(s) (t-s)^{-1/2}} e^{-|\hat{x} - \hat{z}|^2} d\hat{z}  ds.
\end{equation}
Split the time integral into three regions:
\begin{align*}
\left[0, t-(T-t)\right], \quad \left[t-(T-t), t-\mu_0^2(t)R^2(t)\right], \quad \left[t-\mu_0^2(t)R^2(t), t\right].
\end{align*}
Firstly, for $s\in [0,t-(T-t)]$, since $T-s<2(t-s)$, we have
\begin{align}\label{Aug5-9}
&\int_{0}^{t-(T-t)}\mu_0^{\nu-3}(s)R^{-2-\alpha}(s)\int_{\left\{|\hat{z}|\leq \frac{\mu_0(s)R(s)}{\sqrt{t-s}}\right\}\cap \mathbb{R}^{4}_{+}}e^{-|\hat{x}-\hat{z}|^{2}}d\hat{z} ds \notag
\\&\lesssim\int_{0}^{t-(T-t)}\frac{\mu_0^{\nu+1}(s)R^{2-\alpha}(s)}{(t-s)^2}ds\notag
\\&\lesssim\int_{0}^{t-(T-t)}\frac{\mu_0^{\nu+1}(s)R^{2-\alpha}(s)}{(T-s)^2}ds\notag
\\&\lesssim \mu_0^{\nu+1}(0)R^{2-\alpha}(0).
\end{align}
Consider $s\in \left[t-(T-t), t-\mu_0^2(t)R^2(t)\right]$, there holds
\begin{align}\label{Aug5-10}
&\int_{t-(T-t)}^{t-\mu^{2}_0(t)R^2(t)}\mu_0^{\nu-3}(s)R^{-2-\alpha}(s)\int_{\left\{|\hat{z}|\leq \frac{\mu_0(s)R(s)}{\sqrt{t-s}}\right\}\cap \mathbb{R}^{4}_{+}}e^{-|\hat{x}-\hat{z}|^{2}}d\hat{z}ds \notag
\\&\lesssim\int_{t-(T-t)}^{t-\mu^{2}_0(s)R^2(s)}\frac{\mu_0^{\nu+1}(s)R^{2-\alpha}(s)}{(t-s)^2}ds\notag
\\&\lesssim \mu_0^{\nu-1}(t)R^{-\alpha}(t).
\end{align}
For $s\in \left[t-\mu_0^2(t)R^2(t), t\right]$, we obtain
\begin{align}\label{Aug5-11}
&\int_{t-\mu^{2}_0(t)R^2(t)}^{t}\mu_0^{\nu-3}(s)R^{-2-\alpha}(s)\int_{\left\{|\hat{z}|\leq \frac{\mu_0(s)R(s)}{\sqrt{t-s}}\right\}\cap \mathbb{R}^{4}_{+}}e^{-|\hat{x}-\hat{z}|^{2}}d\hat{z}ds \notag
\\&\lesssim\int_{t-\mu^{2}_0(t)R^2(t)}^{t}\mu_0^{\nu-3}(s)R^{-2-\alpha}(s)ds \notag
\\&\lesssim \mu_0^{\nu-1}(t)R^{-\alpha}(t).
\end{align}
Combining \eqref{Aug5-9}, \eqref{Aug5-10} and \eqref{Aug5-11} yields
$$
|\psi| \lesssim \mu_0^{\nu-1}(0)R^{-\alpha}(0),
$$
which establishes the estimate \eqref{Aug5-4}.
\end{proof}

\begin{proof}[Proof of \eqref{Aug5-5}.]
Decompose the difference as
$$|\psi(x,t)-\psi(x,T)|\leq I_{1}+I_{2}+I_{3},$$
where
\begin{align*}
I_{1} &:= \int_{0}^{t-(T-t)}\int_{\mathbb{R}^{4}_{+}}\left|K_{t-s}(x,z)-K_{T-s}(x,z)\right| |g_{1}(z,s)|dz  ds, \\
I_{2} &:= \int_{t-(T-t)}^{t}\int_{\mathbb{R}^{4}_{+}}\left|K_{t-s}(x,z)-K_{T-s}(x,z)\right| |g_{1}(z,s)|dz  ds, \\
I_{3} &:= \int_{t}^{T}\int_{\mathbb{R}^{4}_{+}}\left|K_{T-s}(x,z)\right|  |g_{1}(z,s)|dz  ds.
\end{align*}
\textbf{Estimate for $I_1$.} For $t_v = v t + (1-v)T$, $v \in [0,1]$, there holds
$$I_{1}\leq(T-t)\int_{0}^{1}\int_{0}^{t-(T-t)}\mu_0^{\nu-3}(s)R^{-2-\alpha}(s)\int_{\{|z-\xi(s)|\leq 2\mu_0(s)R(s)\}\cap \mathbb{R}^{4}_{+}}|\partial_{t}K_{t_v-s}(x,z)| dz ds dv.$$
Since
\begin{align*}
&\int_{\{|z-\xi(s)|\leq 2\mu_0(s)R(s)\}\cap \mathbb{R}^{4}_{+}}|\partial_{t}K_{t_v-s}(x,z)|  dz \\
&\lesssim \frac{1}{(t_{v}-s)^{3}}\int_{\{|z-\xi(s)|\leq 2\mu_0(s)R(s)\}\cap \mathbb{R}^{4}_{+}}e^{-\frac{|x-z|^{2}}{4(t_{v}-s)}}\left(1+\frac{|x-z|^{2}}{t_{v}-s}\right)  dz \\
&\lesssim \frac{1}{(t_{v}-s)}\int_{\left\{|\hat{z}|\leq \frac{\mu_0(s)R(s)}{\sqrt{t-s}}\right\}\cap \mathbb{R}^{4}_{+}}e^{-\left|\hat{x}-\hat{z}\right|^{2}}\left(1+4\left|\hat{x}-\hat{z}\right|^{2}\right)  d\hat{z}.
\end{align*}
We then get
\begin{align*}
&\int_{0}^{t-(T-t)}\mu_0^{\nu-3}(s)R^{-2-\alpha}(s)\int_{\{|z-\xi(s)|\leq 2\mu_0(s)R(s)\}\cap \mathbb{R}^{4}_{+}}|\partial_{t}K_{t_v-s}(x,z)|  dz  ds \\
&\lesssim \int_{0}^{t-(T-t)}\frac{\mu_0^{\nu+1}(s)R^{2-\alpha}(s)}{(t_{v}-s)^{3}}  ds \\
&\lesssim \int_{0}^{t-(T-t)}\frac{\mu_0^{\nu+1}(s)R^{2-\alpha}(s)}{(T-s)^{3}}  ds \\
&\lesssim \frac{|\log T|^{(\nu+1)-\beta(2-\alpha)}(T-t)^{(\nu+1)-\beta(2-\alpha)-2}}{|\log (T-t)|^{2(\nu+1)-2\beta(2-\alpha)}},
\end{align*}
where we have used $R(t)=\mu_0^{-\beta}(t)$. Thus,
$$I_{1}\lesssim \frac{\mu_0^{\nu+1}(t)R^{2-\alpha}(t)}{T-t}.$$
\textbf{Estimate of $I_2$.}
\begin{align*}
I_{2} &\leq \int_{t-(T-t)}^{t}\int_{\{|z-\xi(s)|\leq 2\mu_0(s)R(s)\}\cap \mathbb{R}^{4}_{+}}|K_{t-s}(x,z)|\mu_0^{\nu-3}(s)R^{-2-\alpha}(s)  dz  ds \\
&\quad + \int_{t-(T-t)}^{t}\int_{\{|z-\xi(s)|\leq 2\mu_0(s)R(s)\}\cap \mathbb{R}^{4}_{+}}|K_{T-s}(x,z)|\mu_0^{\nu-3}(s)R^{-2-\alpha}(s)  dz  ds
\\&:=J_1+J_2.
\end{align*}
Similar to the estimate of \eqref{Aug5-10} and \eqref{Aug5-11}, we get
\begin{align*}
J_1&\lesssim\left[\int_{t-(T-t)}^{t-\mu^{2}_0(t)R^2(t)}+ \int_{t-\mu^{2}_0(t)R^2(t)}^{t}\right]
\mu_0^{\nu-3}(s)R^{-2-\alpha}(s)\int_{\left\{|\hat{z}|\leq \frac{\mu_0(s)R(s)}{\sqrt{t-s}}\right\}\cap \mathbb{R}^{4}_{+}}e^{-|\hat{x}-\hat{z}|^2}  d\hat{z} ds
\\&\lesssim \mu_0^{\nu-1}(t)R^{-\alpha}(t).
\end{align*}
Similar computations show that
\begin{align*}
J_2\lesssim \frac{\mu_0^{\nu+1}(t)R^{2-\alpha}(t)}{T-t}+\mu_0^{\nu-1}(t)R^{-\alpha}(t)
\lesssim \mu_0^{\nu-1}(t)R^{-\alpha}(t).
\end{align*}
Thus,
$$I_{2}\lesssim \mu_0^{\nu-1}(t)R^{-\alpha}(t).$$
\textbf{Estimate for $I_3$.}
\begin{align*}
I_{3} &\lesssim\int_{t}^{T}\frac{{\mu_0^{\nu-3}(s)R^{-2-\alpha}(s)}}{(T-s)^2} \int_{\{|z-\xi(s)|\leq 2\mu_0(s)R(s)\}\cap \mathbb{R}^{4}_{+}}e^{-\frac{|x-z|^{2}}{4(T-s)}}dz ds \\
&\lesssim\int_{t}^{T}\mu_0^{\nu-3}(s)R^{-2-\alpha}(s)\int_{\left\{|\hat{z}|\leq \frac{\mu_0(s)R(s)}{\sqrt{t-s}}\right\}\cap \mathbb{R}^{4}_{+}}e^{-|\hat{x}-\hat{z}|^2}  d\hat{z} ds \\
&\lesssim\int_{t}^{T}\frac{\mu_0^{\nu+1}(s)R^{2-\alpha}(s)}{(T-s)^2}  ds \\
&\lesssim\frac{\mu_0^{\nu+1}(t)R^{2-\alpha}(t)}{T-t}.
\end{align*}
Combining all estimates, we obtain
$$|\psi(x,t) - \psi(x,T)| \lesssim \mu_0^{\nu-1}(t) R^{-\alpha}(t).$$
\end{proof}

\subsection{Estimates for problem \eqref{outer2}}
We now consider problem \eqref{outer2}. Recall that its solution kernel is given by
\begin{equation}\label{outer3}
K(\tilde{x},x_4,t):=\frac{1}{\Gamma\left(\frac{1}{2}\right)(4\pi)^{\frac{3}{2}}} \frac{1}{t^2}e^{\frac{-(|\tilde{x}|^{2}+x_{4}^{2})}{4t}}.
\end{equation}
Then the function
$$\psi(x,t)=\psi(\tilde{x},x_{4},t):=\int_{0}^{\infty}\int_{\mathbb{R}^3}K(\tilde{z},x_{4},s)g_{2}(\tilde{x}-\tilde{z},t-s)d\tilde{z}ds$$
is a solution of problem \eqref{outer2}.
For $\psi(\tilde{x},x_{4},t)$ and $g_{2}(\tilde{x},t)$, we introduce the following norms

\begin{align}\label{outer5}
\begin{aligned}
\|\psi\|_{*}^{(2)}&:=\left(\mu_0^{\nu-1}(0) R^{-\alpha}(0)+T^{\nu_2}|\log T|^{1-\nu_2}\right)^{-1}\|\psi\|_{L^{\infty}\left(\mathbb{R}^{4}_{+} \times (0,T)\right)}
\\&\quad + \left(\mu_0^{\nu-2}(0) R^{-1-\alpha}(0)+\mu_0^{\nu_2-1}(t)R^{-1}(t)\right)^{-1} \|{\nabla\psi}\|_{L^{\infty}\left(\mathbb{R}^{4}_{+} \times (0,T)\right)}
\\&\quad + \sup_{(x,t) \in \mathbb{R}^{4}_{+} \times (0,T)} \left[ \left(\mu_0^{\nu-1}(0) R^{-\alpha}(0)+(T-t)^{\nu_{2}}|\log T|^{\nu_{2}}|\log(T-t)|^{1-2\nu_{2}}\right)^{-1}\right.
\\&\qquad\qquad\qquad\qquad\left.\left| \psi(x,t) - \psi(x,T) \right| \right] ,
\end{aligned}
\end{align}

\begin{equation}\label{outer6}
\|g_{2}\|_{**}^{(2)}:=\sup_{\mathbb{R}^3\times(0,T)}\frac{|g_{2}(\tilde{x},t)|}{\rho_{**}^{(2)}},
\end{equation}
with weight function
\begin{align*}
\rho_{**}^{(2)}:=\mu_0^{\nu-2}(t) R^{-1-\alpha}(t) \chi_{\{|\tilde{x} -\tilde{\xi}(t)| \leq 2\mu_0R\}}+\frac{\mu_0^{\nu_{2}}(t)}{|\tilde{x}-\tilde{\xi}(t)|} \chi_{\{|\tilde{x}-\tilde{\xi}(t)| \geq \mu_0R\}}.
\end{align*}
Then we have the following key estimate for $\psi$.
\begin{prop}\label{25-prop5.5}
Assume $\|g_2\|_{**}^{(2)} <\infty$ and $g_{2}(\tilde{x },t)\equiv g_{2}(\tilde{x},0)$ for $t\in (-\infty,0)$. Then the unique solution to \eqref{outer2} satisfies
\begin{equation}\label{outer:7}
\|\psi\|_{*}^{(2)}\lesssim\|g_{2}\|_{**}^{(2)}.\end{equation}
\end{prop}
Proposition \ref{25-prop5.5} follows from the subsequent lemmas.

\begin{lemma}\label{25Aug-lemma5.6}
Let $\psi$ solve problem \eqref{outer2} with $g_{2}(\tilde{x },t)\equiv g_{2}(\tilde{x},0)$ for $t\in (-\infty,0)$  and
$$
|g_2(\tilde{x},t)| \lesssim \mu_0^{\nu-2}(t) R^{-1-\alpha}(t) \chi_{\{|\tilde{x} -\tilde{\xi}(t)| \leq 2\mu_0R\}} , \quad \alpha,\nu \in (0,1).
$$
Then the following estimates hold
\begin{align}
|\psi(x,t)| &\lesssim \mu_0^{\nu-1}(0)R^{-\alpha}(0), \label{Aug7-1}
\\
|\psi(x,t) - \psi(x,T)| &\lesssim \mu_0^{\nu-1}(0)R^{-\alpha}(0), \label{Aug7-2}
\\
|\nabla_{\tilde{x}} \psi(x,t)| &\lesssim \mu_0^{\nu-2}(t)R^{-1-\alpha}(t). \label{Sep19-1}
\end{align}
\end{lemma}

\begin{lemma}\label{25Aug-lemma5.7}
Assume that $\psi$ solves \eqref{outer2} with $g_{2}(\tilde{x },t)\equiv g_{2}(\tilde{x},0)$ for $t\in (-\infty,0)$  and
\begin{align}\label{Aug2-4}
|g_2(\tilde{x},t)|\lesssim \frac{\mu_0^{\nu_{2}}(t)}{|\tilde{x}-\tilde{\xi}(t)|} \chi_{\{|\tilde{x}-\tilde{\xi}(t)| \geq \mu_0R\}}.
\end{align}
Then for sufficiently small $T>0$, the following estimates hold
\begin{align*}
|\psi(x,t)|&\lesssim T^{\nu_2}|\log T|^{1-\nu_2},
\\|\psi(x,t)-\psi(x,T)|&\lesssim (T-t)^{\nu_{2}}|\log T|^{\nu_{2}}|\log(T-t)|^{1-2\nu_{2}},
\\ |\nabla_{\tilde{x}}\psi(x,t)|&\lesssim  \frac{\mu_0^{\nu_{2}-1}(t)}{R(t)}.
\end{align*}
\end{lemma}

\begin{proof}[Proof of Lemma \ref{25Aug-lemma5.6}]
Decompose the solution $\psi(\tilde{x},x_4,t)$ as
\begin{align*}
|\psi(\tilde{x},x_{4},t)|&
\leq\left|\int_{0}^{t}\int_{\mathbb{R}^3}K(\tilde{z},x_{4},s)g_{2}(\tilde{x}-\tilde{z},t-s)d\tilde{z}ds\right|
\\&\quad+\left|\int_{t}^{\infty}\int_{\mathbb{R}^3}K(\tilde{z},x_{4},s)g_{2}(\tilde{x}-\tilde{z},t-s)d\tilde{z}ds\right|
\\&:=I_{1}+I_{2}.
\end{align*}

For $I_1$, using the bound on $g_2(\tilde{x},s)$, we have
\begin{align*}
I_{1}&=\left|\int_{0}^{t}\int_{\mathbb{R}^3}K(\tilde{x}-\tilde{z},x_{4},t-s)g_{2}(\tilde{z},s)d\tilde{z}ds\right|
\\&\lesssim \int_{0}^{t}\mu_0^{\nu-2}(s)R^{-1-\alpha}(s)\int_{\{|\tilde{z}-\tilde{\xi}(s)|\leq 2\mu_0(s)R(s)\}}K(\tilde{x}-\tilde{z},x_4,t-s)d\tilde{z}ds
\\&\lesssim\int_{0}^{t}\frac{\mu_0^{\nu-2}(s)R^{-1-\alpha}(s)}{(t-s)^2}\int_{\{|\tilde{z}-\tilde{\xi}(s)|\leq 2\mu_0(s)R(s)\}}e^{\frac{-|\tilde{x}-\tilde{z}|^{2}}{4(t-s)}}d\tilde{z}ds.
\end{align*}
Applying the change of variables
\begin{align*}
\hat{z} = \frac{\tilde{z}-\tilde{\xi}(s)}{2\sqrt{t-s}}, \quad \hat{x} = \frac{\tilde{x}-\tilde{\xi}(s)}{2\sqrt{t-s}},
\end{align*}
yields
\begin{equation*}
I_1 \lesssim \int_{0}^{t} \frac{\mu_0^{\nu-2}(s) R^{-1-\alpha}(s)}{\sqrt{t-s}} \int_{|\hat{z}| \leq \mu_0(s) R(s) (t-s)^{-1/2}} e^{-|\hat{x} - \hat{z}|^2} d\hat{z}ds.
\end{equation*}
By analogous estimates to those used for \eqref{outer24}, we find
$$
I_1 \lesssim \mu_0^{\nu-1}(0)R^{-\alpha}(0).
$$

For $I_{2}$, using the assumption $g_{2}(\tilde{x},t) \equiv g_{2}(\tilde{x},0)$ for $t < 0$, then for $s > t$, there holds
\begin{align*}
I_{2}&=\left|\int_{t}^{\infty}\int_{\mathbb{R}^3}K(\tilde{z},x_4,s)g_{2}(\tilde{x}-\tilde{z},0)d\tilde{z}ds\right|
\\&\lesssim \mu_0^{\nu-2}(0)R^{-1-\alpha}(0)\int_{\{|\tilde{z}-\tilde{\xi}(0)|\leq 2\mu_0(0)R(0)\}}\int_{t}^{\infty}\frac{e^{\frac{-|\tilde{z}|^{2}}{4s}}}{s^2}dsd\tilde{z}
\\&\lesssim \mu_0^{\nu-2}(0)R^{-1-\alpha}(0)\int_{\{|\tilde{z}-\tilde{\xi}(0)|\leq 2\mu_0(0)R(0)\}}\frac{1}{|\tilde{z}|^2}d\tilde{z}
\\&\lesssim \mu_0^{\nu-1}(0)R^{-\alpha}(0),
\end{align*}
where we have used the integral estimate
$$\int_{t}^{\infty}\frac{e^{\frac{-|\tilde{z}|^{2}}{4s}}}{s^2}ds\lesssim|\tilde{z}|^{-2}.$$
Combining the estimates for $I_{1}$ and $I_{2}$, we obtain
$$|\psi(x,t)|\lesssim \mu_0^{\nu-1}(0)R^{-\alpha}(0).$$
The proof of \eqref{Aug7-2} follows similar arguments and is omitted here.
\end{proof}

\section{The linear theory of the inner problem}\label{linear-inner}
In this section, we develop a linear theory for the inner problem \eqref{Jul-inner-problem}. To analyze \eqref{Jul-inner-problem}, we first consider the linearized equation
\begin{equation}\label{Aug8-4}
\begin{cases}
\mu^{2} \phi_{t}=\Delta_{y} \phi+g(y,t) &\text { in } \ B_{2R}^{+} \times(0, T),  \\
-\frac{d \phi}{d y_{4}}(\tilde{y}, 0, t)=2U(\tilde{y}, 0) \phi(\tilde{y}, 0, t)+h(\tilde{y},0, t)  &\text { in }\ (\partial B_{2R}^{+}\cap\mathbb{R}^{3}) \times(0, T),
\end{cases}
\end{equation}
where
$$B_{2 R}^{+}:=\{y=(\tilde{y},y_4)\mid \tilde{y}\in \mathbb{R}^3,\ y_4\geq 0,\ |y|\leq 2R(t)\}.$$
We shall construct a solution to the following initial value problem
\begin{equation}\label{Aug8-1}
\begin{cases}
\mu^{2} \phi_{t}=\Delta_{y} \phi+g(y,t) &\text { in } \ B_{2R}^{+} \times(0, T),  \\
-\frac{d \phi}{d y_{4}}(\tilde{y}, 0, t)=2U(\tilde{y}, 0) \phi(\tilde{y}, 0, t)+h(\tilde{y},0, t)  &\text { in }\ (\partial B_{2R}^{+}\cap\mathbb{R}^{3}) \times(0, T).
\end{cases}
\end{equation}
Our goal is to construct a solution to \eqref{Aug8-1} that defines a bounded linear operator of $g$ and $h$, satisfying favorable bounds in appropriate weighted norms.
For any $0<\sigma<1$ and $a>0$, we introduce the weighted norms
\begin{align} \label{Aug17-5}
&\|g\|_{a+\sigma,\nu}:= \sup_{(y,t)\in B_{2R}^{+} \times(0, T)} \mu_0^{-\nu}(t)(1+|y|^{a+\sigma}) |g(y,t)|, \notag
\\&\|h\|_{a+\sigma,\nu}:= \sup_{(\tilde{y},t)\in (\partial B_{2R}^{+}\cap\mathbb{R}^{3})\times(0, T)} \mu_0^{-\nu}(t)(1+|\tilde{y}|^{a+\sigma}) |h(\tilde{y},t)|.
\end{align}
Moreover, orthogonality conditions are imposed on $g$ and $h$. Specifically, we seek a solution $\phi$ to \eqref{Aug8-1} that satisfies
\begin{equation*}
 \int_{B_{2R}^{+}} g(y, t)Z_{i}(y) d y+\int_{\partial B_{2R }^{+}\cap{\mathbb{R}}^3} h(\tilde{y}, 0, t) Z_{i}(\tilde{y},0) d \tilde{y}=0 \quad \text{for}~ i=1,\dots,4,~t\in(0,T).
\end{equation*}

As detailed below, our construction involves decomposing the equation into distinct spherical harmonic modes. Let $\{\Theta_{m}\}_{m=0}^{\infty}$ be an orthonormal basis of $L^{2}(\mathbb S^{2})$ consisting of spherical harmonics, which satisfy the eigenvalue problem
$$\Delta_{\mathbb S^{2}}\Theta_{m}+\tilde{\lambda}_{m}\Theta_{m}=0 \quad \text{in}\  \mathbb S^{2},$$
where the eigenvalues are ordered as
$$0=\tilde{\lambda}_0<\tilde{\lambda}_{1}=\cdots=\tilde{\lambda}_3=2<\tilde{\lambda}_4\leq\cdots.$$
Specifically,
$$\Theta_{0}(\tilde{y})=a_{0},\quad \Theta_{j}(\tilde{y})=a_{1}\tilde{y}_{j},~j=1,2,3,$$
for constants $a_{0}$ and $a_{1}$. In general, all eigenvalues eigenvalues take the form $\tilde{\lambda}_{m}=l(1+l)$ for some nonnegative integer $l$.

Given $g\in L^{2}(B_{2R}^{+})$ and $h\in L^{2}(\partial B_{2R}^{+}\cap\mathbb{R}^{3})$, we decompose them into spherical harmonic components,
\begin{align*}
g(\tilde{y},y_4,t)&=\sum_{j=0}^{\infty}g_{j}(r,y_4,t)\Theta_{j}\left(\frac{\tilde{y}}{r}\right),\quad r=|\tilde{y}|,
\\g_{j}(r,y_4,t)&=\int_{S^{2}}g(r\theta,y_4,t)\Theta_{j}(\theta)d\theta,
\\h(r,t)&=\sum_{j=0}^{\infty}h_{j}(r,\tau)\Theta_{j}\left(\frac{\tilde{y}}{r}\right),\quad h_{j}(r,t)=\int_{S^{2}}h(r\theta,t)\Theta_{j}(\theta)d\theta,
\end{align*}
and write $g=g^{0}+g^{\perp}$, $h=h^{0}+h^{\perp}$ with
\begin{align*}
&g^{0}=g_{0}(r,y_4,t),\quad  g^{\perp}=\sum_{j=1}^{\infty}g_{j}(r,y_4,t)\Theta_{j},
\\&
h^{0}=h_{0}(r,t),\qquad\quad  h^{\perp}=\sum_{j=1}^{\infty}h_{j}(r,t)\Theta_{j}.
\end{align*}
Analogously, we decompose the unknown solution $\phi=\phi^{0}+\phi^{\perp}$. Solving \eqref{Aug8-1} then reduces to finding pairs $(\phi^{0},g^{0},h^{0})$, $(\phi^{\perp},g^{\perp},h^{\perp})$ in each
mode. We introduce the time change
\begin{align}\label{Aug8-3}
t=t(\tau),\quad \frac{dt}{d\tau}=\mu_{0}^{2}(t),
\end{align}
which transforms \eqref{Aug8-1} into
\begin{equation}\label{Aug8-2}
\begin{cases}
\phi_{\tau}=\Delta_{y} \phi+g(y,\tau) &\text { in } \ B_{2R}^{+} \times(\tau_0, \infty),  \\
-\frac{d \phi}{d y_{4}}(\tilde{y}, 0, \tau)=2U(\tilde{y}, 0) \phi(\tilde{y}, 0, \tau)+h(\tilde{y},0,\tau)  &\text { in }\ (\partial B_{2R}^{+}\cap\mathbb{R}^{3}) \times (\tau_0, \infty),\\
\end{cases}
\end{equation}
where $\tau_0$ satisfies $t(\tau_0)=0$.

We now construct solutions to the initial-boundary value problems for $(\phi^{0},g^{0},h^{0})$ and $(\phi^{\perp},g^{\perp},h^{\perp})$.
\begin{equation}\label{inner4.4}
\begin{cases}
\phi^{0}_{\tau}=\Delta \phi^{0}+g^{0}(y, \tau)-c(\tau) Z_{0}(y)  & \text{in}\ B_{2 R}^{+} \times\left(\tau_{0}, \infty\right), \\
\begin{aligned}
-\frac{d \phi^{0}}{d y_{4}}(\tilde{y}, 0, \tau)=&~2U(\tilde{y}, 0) \phi^{0}(\tilde{y}, 0, \tau)+h^{0}(\tilde{y}, 0, \tau)
\end{aligned}
& \text{in}\ (\partial B_{2 R}^{+}  \cap\mathbb{R}^3) \times\left(\tau_{0}, \infty\right), \\
\phi^{0}\left(y, \tau_{0}\right)=0 & \text{in}\ B_{2 R(0)}^{+},
\end{cases}
\end{equation}

\begin{equation}\label{inner4.6}
\begin{cases}
\phi^{\perp}_{\tau}=\Delta \phi^{\perp}+g^{\perp}(y, \tau) & \text{in}\ B_{2 R}^{+} \times\left(\tau_{0}, \infty\right), \\
-\frac{d \phi^{\perp}}{d y_{4}}(\tilde{y}, 0, \tau)=2U(\tilde{y}, 0) \phi^{\perp}(\tilde{y}, 0, \tau)+h^{\perp}(\tilde{y}, 0, \tau) &  \text{in}\ (\partial B_{2 R}^{+}  \cap\mathbb{R}^3) \times\left(\tau_{0}, \infty\right), \\ \phi^{\perp}\left(y, \tau_{0}\right)=0 & \text{in}\  B_{2 R(0)}^{+}.
\end{cases}
\end{equation}
The eigenfunction $Z_0(y)$ was studied in \cite{Wei-Ye-Zhang-Zeng} which decays at infinity exponentially. The main result of this section is as follows.
\begin{prop}\label{proposition4.2}
Suppose $\sigma_1$, $\sigma_2$, $\nu\in(0,1)$, and functions $g=g^{0}+g^{\perp}$, $h=h^{0}+h^{\perp}$ satisfy
$$\|g^{0}\|_{2+\sigma_1,\nu}<+\infty, \quad \|h^{0}\|_{1+\sigma_1,\nu}<+\infty,\quad \|g^{\perp}\|_{3+\sigma_2,\nu}<+\infty,\quad\|h^{\perp}\|_{2+\sigma_2,\nu}<+\infty$$ along with the orthogonality conditions
\begin{equation}\label{orthogonality2}
 \int_{B_{2R}^{+}} g(y, \tau)Z_{i}(y) d y+\int_{B_{2R }^{+}\cap{\mathbb{R}}^3} h(\tilde{y}, 0,\tau) Z_{i}(\tilde{y},0) d \tilde{y}=0
\end{equation}
for all $ \tau\in(\tau_0, \infty)$ and $i=1,\dots,4.$ Then there exist solutions $\phi^0=\phi^0 [g^0,h^0]$, $c=c[g^0,h^0]$ to \eqref{inner4.4} and $\phi^{\perp}[g^{\perp},h^{\perp}]$ to \eqref{inner4.6}, respectively, defining linear operators of $g=g^{0}+g^{\perp}$ and $h=h^{0}+h^{\perp}$. The combined solution $\phi[g,h]=\phi^{0}[g^0,h^0]+\phi^{\perp}[g^{\perp},h^{\perp}]$ and $c=c[g^0,h^0]$ satisfy the following estimates
 \begin{equation}\label{inner4:4}
|\phi[g, h]| \lesssim \mu_0^{\nu}(t) \left[\frac{R(t)^{4-\sigma_1}\|g^{0}\|_{2+\sigma_1,\nu}}{1+|y|^4}+\frac{R(t)^{3-\sigma_1}\|h^{0}\|_{1+\sigma_1,\nu}}{1+|y|^3}+\frac{\|g^{\perp}\|_{3+\sigma_2,\nu} }{1+|y|^{1+\sigma_2}}+\frac{\|h^{\perp}\|_{2+\sigma_2,\nu} }{1+|y|^{1+\sigma_2}}\right],
\end{equation}

\begin{equation}\label{inner7:4}
 \begin{aligned}
 &\left|c(\tau)-\left(\int_{B_{2R}^{+}}g^{0}(y, \tau)Z_{0}(y)dy+\int_{\partial{B}_{2R}^{+}\cap{\mathbb{R}}^3}h^{0}(\tilde{y}, 0, \tau)\widetilde{Z}_{0}(\tilde{y})d\tilde{y}\right)\right|
 \\& \lesssim \mu_0^{\nu}(t) R(t)^{1-\sigma_1}\left(\left\|g^{0}\right\|_{2+\sigma_1,\nu}+\left\|h^{0}\right\|_{1+\sigma_1,\nu}\right)
 \\&\quad+\mu_0^{\nu}(t) R(t)^{2-\sigma_1}\left\|g^{0}-Z_{0}\left(\int_{B_{2R}^{+}}g^{0}Z_{0}dy+\int_{\partial{B}_{2R}^{+}\cap{\mathbb{R}}^3}h^{0}\widetilde{Z}_{0}d\tilde{y}\right)\right\|_{2+\sigma_1,\nu},
 \end{aligned}
\end{equation}
where $R(t)=\mu_0^{-\beta}(t)>0$, $\beta\in (0,1/2)$.
\end{prop}

The construction of the operator $\phi[g,h]$ described in the proposition will be carried out mode by mode. And for $\nu>0$ specified in the proposition, relation \eqref{Aug8-3} allows us to choose $\tilde{\nu}>0$ such that
$$\mu_0^{\nu}(t)\sim \tau^{-\tilde{\nu}}.$$

\subsection{The case of mode $j=0$}\label{4.1}
In this subsection, we construct the solution $\phi^{0}$ for problem \eqref{inner4.4}. We assume $\phi(\tilde{y},y_4,\tau)$ is radially symmetric with respect to the variable $\tilde{y}$. Correspondingly, we suppose the functions $g(\tilde{y},y_4,\tau)$ and $h(\tilde{y},\tau)$ are also radially symmetric in $\tilde{y}$. Let $g^{0}=g_{0}(r,y_4,\tau)$, $h^{0}=h_{0}(r,\tau)$, we then consider the following initial-boundary value problem.

\begin{equation}\label{inner8:2}
\begin{cases}
\phi_{\tau} = \Delta \phi + g^{0}(y, \tau) - c(\tau) Z_{0}(y) & \text{in } B_{2R}^{+} \times (\tau_{0}, \infty), \\
\begin{aligned}
-\frac{d \phi}{d y_4}(\tilde{y}, 0, \tau)
    = ~ & 2U(\tilde{y}, 0) \phi(\tilde{y}, 0, \tau)  + h^{0}(\tilde{y}, 0, \tau)
\end{aligned}   & \text{in } (\partial B_{2R}^{+} \cap \mathbb{R}^3) \times (\tau_{0}, \infty), \\
\phi(y, \tau_{0}) = 0  & \text{in } B_{2R(0)}^{+}, \\
\phi(\tilde{y}, \tau) = 0  & \text{in } \partial B_{2R}^{+} \backslash (\partial B_{2R}^{+} \cap \mathbb{R}^3) \times (\tau_{0}, \infty).
\end{cases}
\end{equation}

\begin{lemma}\label{lemma4.1}
Let $-2<\alpha<1$, $\nu>0$ and $\tilde\nu>0$ with $\mu_0^\nu(t)\sim\tau^{-\tilde{\nu}}$. Then, for $R(t)=\mu_0^{-\beta}(t)>0$, $\beta\in (0,1/2)$ and given $g^{0}$, $h^{0}$ satisfying $\|g^{0}\|_{2+\alpha,\nu}<+\infty$, $\|h^{0}\|_{1+\alpha,\nu}<+\infty$, there exist solutions $\phi=\phi (\tilde{y},y_4,\tau)$ and $c(\tau)$ to problem \eqref{inner8:2} that define linear operators of $g^{0}$ and $h^{0}$ and satisfy the estimates
\begin{equation}\label{inner9:2}
(1+|y|)|\nabla \phi(y,\tau)|+|\phi(y,\tau)|\lesssim \tau^{-\tilde{\nu}} \frac{R(t)^{2}}{1+|y|^{2}}\Theta_{R}^{\alpha}\left(\|g^{0}\|_{2+\alpha, \nu}+\|h^{0}\|_{1+\alpha,\nu}\right).
 \end{equation}
and
\begin{equation}\label{inner10:2}
\begin{aligned}
&\left|c(\tau)-\frac{\int_{B_{2R}^{+}}g^{0}(y,\tau)Z_{0}(y)dy}{\int_{B_{2R}^{+}}Z_{0}^{2}dy}\right|
\\&\lesssim\tau^{-\tilde{\nu}} \left(\left\|h^{0}\right\|_{1+\alpha,\nu}+\left\|g^{0}-Z_{0}\int_{B_{2R}^{+}}g^{0}Z_{0}dy\right\|_{2+\alpha,\nu}\right),
\end{aligned}
\end{equation}
where
\begin{align}\label{Oct3-1}
\Theta_{R}^{\alpha}:=\begin{cases}\frac{1}{1+|y|^{\alpha}} & \text{if}\  \alpha>0,
\\ \log R & \text{if}\   \alpha=0,\\
R^{-\alpha} & \text{if}\  \alpha<0.
\end{cases}
\end{align}
\end{lemma}

We will utilize a fundamental lemma concerning the quadratic form associated with the linear operator $L_{U}$:
\begin{align*}
B(\phi,\phi)&=-\int_{\mathbb{R}_{+}^4} \phi  \Delta \phi dy \\ &=\int_{\mathbb{R}_{+}^4}|\nabla \phi|^{2}dy+\int_{\mathbb{R}^3} \frac{d \phi}{d y_4}  \phi d \tilde{y} \\ &=\int_{\mathbb{R}^4_{+}}|\nabla \phi|^{2}dy-\int_{\mathbb{R}^3}2U(\tilde{y}, 0) \phi^{2}(\tilde{y}, 0, \tau)d\tilde{y}.
\end{align*}

The following lemma provides an estimate for the quadratic form $B(\phi,\phi)$ associated with the second $L^{2}$-eigenvalue in the half-ball $B_{2R}^{+} \subset \mathbb{R}^4_{+}$ under zero boundary conditions for large radius $R$.

Define the radial Sobolev space
$$H_{radial}^{1}(\mathbb{R}^{4}_+):=\left\{\phi(y)\in H^{1}(\mathbb{R}^{4}_+)\mid\phi(y)\ \text{is radially symmetric in the variable}\ \tilde y\right\}.$$
And the function space
$$
H_R := \left\{ \phi \in H_{radial}^{1}(\mathbb{R}^{4}_+) ~\Big|~
  \phi|_{\mathbb{R}^4_+ \setminus B_{2R}^+} = 0, \int_{\mathbb{R}^4_+}\phi Z_0dy = 0
\right\}.
$$

\begin{lemma}\label{lemma4.2}
There exists a constant $\gamma>0$ such that for every $\phi \in H_R$, there holds
\begin{equation}\label{inner4.13}
\frac{\gamma}{R^{2}}\int_{B_{2R}^{+}}|\phi|^{2}dy\leq B(\phi,\phi).
\end{equation}	
\end{lemma}

\begin{proof}
Define
$$ \lambda_{R}:=\inf \left\{B(\phi,\phi)~\Big|~ \phi \in H_{R},\  \int_{B_{2R}^{+}}|\phi|^{2}dy=1 \right\}.$$
By standard compactness arguments \cite[Chapter 8]{Evans1990}, $\lambda_{R}$ is achieved by a radial minimizer $\phi_{R} \in H_{R}$ with $\int_{B_{2R}^{+}} |\phi_{R}|^{2} dy= 1$, satisfying
\begin{equation}\label{inner4.14}
\begin{cases}
\Delta\phi_{R} + \lambda_R \phi_{R} = c_{R} Z_0(y)  & \text{in}\ \mathbb{R}_{+}^{4},\\
-\frac{d\phi_{R}}{d y_{4}}(\tilde{y},0) -2U(\tilde{y}, 0)\phi_{R}(\tilde{y},0) = 0 &  \text{in}\ \mathbb{R}^{3},
\end{cases}
\end{equation}
for a suitable Lagrange multiplier $c_{R}$.
Combining the $L^{2}$-orthogonality condition of $\phi_R$ with respect to $\widetilde{Z}_{0}$, it follows that
$$B(\phi,\phi)\geq 0\quad \text{for all}~ \phi\in H_{radial}^{1}(\mathbb{R}^{4}_+),$$ and thus $\lambda_{R}\geq 0$.
Indeed, if $\lambda_{R} < 0$, then the condition $\phi_{R} |_{\mathbb{R}^4_{+}\backslash B_{2R}^{+} }=0$ implies $B(\phi_{R},\phi_{R}) < 0$. However, this would force $\phi_{R}$ to satisfy
$$
\phi_R(y) = k Z_0(y) \quad \text{for some constant } k \neq 0.
$$
Combining this with the orthogonality condition $\int_{\mathbb{R}^4_+}\phi Z_0dy = 0$, we deduce $\phi_{R} = 0$, contradicting $\int_{B_{2R}^{+}}|\phi_{R}|^{2}dy=1$.

To prove \eqref{inner4.13}, assume by contradiction that
\begin{equation}\label{inner4.15}
\lambda_{R}=o(R^{-2})\quad \text{as} \ R\to \infty.
\end{equation}
Testing \eqref{inner4.14} against $Z_{0}(y)$ and integrating in $\mathbb{R}^4_{+}$, we get
$$c_{R}\int_{\mathbb{R}^4_+}Z_{0}^{2}dy=\lambda_{R}\int_{\mathbb{R}^4_{+}}\phi_{R}Z_{0}dy.$$
Since $\|\phi_{R}\|_{L^{2}(\mathbb{R}^4_+)}=1$,  this implies $|c_{R}|\lesssim \lambda_R$.

On the other hand, again since $\|\phi_{R}\|_{L^{2}(\mathbb{R}^4_+)}=1$, regularity estimates yield
$$\|\phi_{R}\|_{L^{\infty}(\mathbb{R}^4_+)}\lesssim 1.$$
Moreover, testing \eqref{inner4.14} against $Z_4$ and integrating in $\mathbb{R}^4_+$ yields $\int_{\mathbb{R}^4_+}\phi_{R}Z_4dy=0$
when $\lambda_{R}>0.$

{\bf{Claim 1:}} $\lambda_{R}\neq0$.

Observe that $\lambda_{R}$ decreases with $R$.
Suppose $\lambda_{R_1} = 0$ for some $R_1 > 0$. Then $B(\phi_{R_1}, \phi_{R_1}) = 0$. For any $R_2 > R_1$, since $\phi_{R_1} \in H_{R_2}$, we have
$$\lambda_{R_{2}}=0=B(\phi_{R_1},\phi_{R_1}).$$
Taking $R_{2}\to +\infty$ yields
$$\phi_{R_{1}}\in H_{\infty},\quad \lambda_{\infty}=0=B(\phi_{R_1},\phi_{R_1}),$$
where
$$H_{\infty}:=\left\{\phi\in H_{radial}^{1}(\mathbb{R}^4_+)~\Big|~  \int_{\mathbb{R}^4_+}\phi Z_0dy = 0\right\}.$$
However, in the whole space $\mathbb{R}^4_{+}$, $\lambda_{\infty}=0$ is achieved only by the function $cZ_4$ for some constant $c$. The condition $\phi_{R_{1}}|_{\mathbb{R}^4_{+}\backslash B_{2R}^{+} }=0$ forces $c=0$, contradicting $\int_{B_{2R}^{+}}|\phi_{R_{1}}|^{2}dy=1$.

Set $g_{R}=-\lambda_{R}\phi_{R}$ and $h_{R}=c_{R}\widetilde{Z}_{0}(\tilde{y})$. For any $0<a<1$, we aim to prove
\begin{equation}\label{inner4.16}
\|(1+|y|^{a})\phi_{R}\|_{L^{\infty}(\mathbb{R}^4_+)}\lesssim \|(1+|y|^{2+a})g_{R}\|_{L^{\infty}(\mathbb{R}^4_+)} + \|(1+|\tilde y|^{1+a})h_{R}\|_{L^{\infty}(\mathbb{R}^3)}.
\end{equation}
Suppose for contradiction that there exists a sequence $R_{m}\to +\infty$ such that
\begin{equation}\label{inner4.17:2}
\|(1+|y|^{a})\phi_{R_{m}}\|_{L^{\infty}(\mathbb{R}^4_+)}>m\left(\|(1+|y|^{2+a})g_{R_{m}}\|_{L^{\infty}(\mathbb{R}^4_+)}+ \|(1+|\tilde y|^{1+a})h_{R_{m}}\|_{L^{\infty}(\mathbb{R}^3)}\right),
\end{equation}
i.e.,
\begin{equation}\label{inner4.18:2}
\|(1+|y|^{2+a})\tilde{g}_{R_{m}}\|_{L^{\infty}(\mathbb{R}^4_+)}+ \|(1+|\tilde y|^{1+a})\tilde{h}_{R_{m}}\|_{L^{\infty}(\mathbb{R}^3)}<\frac{1}{m}
\end{equation}
holds for any $m\in \mathbb{N}^{+}$, where
$$\tilde{g}_{R_{m}}=\frac{g_{R_{m}}}{\|(1+|y|^{a})\phi_{R_{m}}\|_{L^{\infty}(\mathbb{R}^4_+)}},\quad \tilde{h}_{R_{m}}=\frac{h_{R_{m}}}{\|(1+|y|^{a})\phi_{R_{m}}\|_{L^{\infty}(\mathbb{R}^4_+)}}.$$
Let $$\tilde{\phi}_{R_{m}}=\frac{\phi_{R_{m}}}{\|(1+|y|^{a})\phi_{R_{m}}\|_{L^{\infty}(\mathbb{R}^4_+)}},$$
then $\|(1+|y|^{a})\tilde{\phi}_{R_{m}}\|_{L^{\infty}(\mathbb{R}^4_+)}=1$.

{\bf{Claim 2:}} For any compact set $K\subset\mathbb{R}^3$,
\begin{equation}\label{inner4.17}
\sup_{y\in K}(1+|y|^{a})|\tilde{\phi}_{R_{m}}(y)|\to 0\quad \text{as } m\to\infty.
\end{equation}

Suppose otherwise. Then there exists $y_{m}\in K$ satisfying
$$(1+|y_{m}|^{a})|\tilde{\phi}_{R_{m}}(y_{m})|\geq\frac{1}{2}.$$
Since
$$(1+|y_{m}|^{a})|\tilde{\phi}_{R_{m}}(y_{m})|\leq \|(1+|y|^{a})\tilde{\phi}_{R_{m}}\|_{L^{\infty}(\mathbb{R}^4_+)}=1,$$
and $\tilde{\phi}_{R_m}$ satisfies
\begin{equation*}
\begin{cases}
\Delta\tilde\phi_{R_{m}} + \lambda_{R_{m}} \tilde\phi_{R_{m}} = 0 & \text{in}\  \mathbb{R}_{+}^4,\\
-\frac{d\tilde\phi_{R_{m}}}{d y_4}(\tilde{y},0) -2U(\tilde{y}, 0)\tilde\phi_{R_{m}}(\tilde{y},0) = c_{R_{m}}\widetilde Z_0(\tilde{y}) &  \text{in}\  \mathbb{R}^3,\\
\tilde{\phi}_{R_{m}}(y)=0 & \text{in}\  \mathbb{R}^4_{+}\backslash B_{2R_{m}}^{+},
\end{cases}
\end{equation*}
with
$$\int_{\mathbb{R}^4_+}\tilde{\phi}_{R_{m}}Z_{0}dy=0,\quad \int_{\mathbb{R}_{+}^4}\tilde{\phi}_{R_{m}}Z_4dy=0,$$
we may pass to a subsequence such that $\tilde{\phi}_{R_{m}}\to \tilde{\phi}$ uniformly in any compact subset $K$ of $\mathbb{R}_+^4$ with $\tilde{\phi}\not\equiv 0$. By the arbitrariness of $K$, $\tilde{\phi}$ solves
\begin{equation*}
\begin{cases}
\Delta\tilde{\phi} = 0 & \text{in}~ \mathbb{R}_{+}^4,\\
-\frac{d\tilde{\phi}}{d y_4}(\tilde{y},0) -2U(\tilde{y}, 0)\tilde{\phi}(\tilde{y},0)=0 & \text{in}~ \mathbb{R}^3,\\
\end{cases}
\end{equation*}
and satisfies
$$\int_{\mathbb{R}^4_+}\tilde{\phi}Z_0dy=0,\quad \int_{\mathbb{R}_{+}^4}\tilde{\phi}Z_4dy=0.$$
The radial symmetry of $\tilde{\phi}$ in $\tilde{y}$ and the non-degeneracy of $L_U$ imply $\tilde{\phi} \equiv 0$, contradiction. Thus, \eqref{inner4.17} holds.

Since $\|(1+|y|^{a})\tilde{\phi}_{R_{m}}\|_{L^{\infty}(\mathbb{R}^4_{+})}=1$, there exists a sequence $\{y_m\}$ with $|y_m| \to \infty$ such that
$$|y_{m}|^{a}|\tilde{\phi}_{R_{m}}(y_{m})|\geq\frac{1}{2}.$$
Define a scaled function
$$\Phi_{R_{m}}(y)=|y_{m}|^{a}\tilde{\phi}_{R_{m}}(|y_{m}|y),$$
which satisfies
\begin{equation*}
\begin{cases}
\Delta\Phi_{R_{m}} + \lambda_{R_{m}}|y_{m}|^{2} \Phi_{R_{m}} = 0 & \text{in}~  \mathbb{R}_{+}^4,\\
-\frac{d\Phi_{R_{m}}}{d y_4}(\tilde{y},0) = H_{R_{m}} &\text{in}~\mathbb{R}^3,
\end{cases}
\end{equation*}
where
$$H_{R_{m}}(\tilde y)=|y_{m}|^{1+a}\left(h_{R_{m}}-2\widetilde{U}\tilde{\phi}_{R_{m}}\right)(|y_{m}|\tilde y).$$
By (\ref{inner4.17:2}) and (\ref{inner4.18:2}), we obtain
$$|H_{R_{m}}(\tilde y)|\lesssim o(1)|\tilde y|^{-1-a}.$$
Hence $H_{R_{m}}(\tilde y)\to 0$ uniformly in compact subsets of $\mathbb{R}^3\backslash \{0\}$. Note that
$$\Phi_{R_{m}}(\hat y_m)=|y_{m}|^{a}\left|\tilde{\phi}_{R_{m}}(y_{m})\right|\geq\frac{1}{2},\quad |\Phi_{R_{m}}(y)|\lesssim |y|^{-a},$$
where
$$\hat{y}_{m}=\frac{y_{m}}{|y_{m}|}\to \hat{e},\quad |\hat{e}|=1.$$
By compactness arguments and the assumption $\lambda_{R_{m}}|y_{m}|^{2} = o(1)$
in \eqref{inner4.15}, we conclude that (up to a subsequence) $\Phi_{R_m} \to \bar{\Phi} \not\equiv 0$ ($\bar{\Phi}(\hat e)\neq 0$) locally uniformly in $\mathbb{R}_{+}^4\backslash\{0\}$, and $\bar{\Phi}$ satisfies
\begin{equation*}
\begin{cases}
\Delta\bar{\Phi} = 0 & \text{in}~ \mathbb{R}_{+}^4\setminus\{0\},\\
-\frac{d\bar{\Phi}}{d y_4}(\tilde{y},0) = 0 & \text{in}~\mathbb{R}^3\setminus\{0\},\\
|\bar{\Phi}(\tilde y)|\leq |\tilde y|^{-a} & \text{in}~ \mathbb{R}^3\setminus\{0\},
\end{cases}
\end{equation*}
which is equivalent to $\bar{\phi}:=\bar{\Phi} |_{\mathbb{R}^3}$ satisfies
$$\begin{cases}
(-\Delta)^{\frac{1}{2}}\bar{\phi}(\tilde y)=0 & \text{in}~ \mathbb{R}^3\setminus\{0\},\\
|\bar{\phi}(\tilde y)|\leq |\tilde y|^{-a} & \text{in}~
\mathbb{R}^3\setminus\{0\}.
\end{cases}$$

Next, we prove $\bar{\phi}=0$. Without loss of generality, we consider
\begin{equation}\label{inner4.18}
\begin{cases}
(-\Delta)^{\frac{1}{2}}\bar{\phi}(\tilde y)=0 & \text{in}~ \mathbb{R}^3\setminus\{0\},\\
|\bar{\phi}(\tilde y)|\leq |\tilde y|^{-a} & \text{in}~ \mathbb{R}^3\setminus\{0\}.
\end{cases}
\end{equation}
Let $\epsilon>0$ be an arbitrary fixed constant, define $\Psi(\tilde{y})=\frac{\epsilon}{|\tilde y|^{2}}$, then $\Psi$ is a supersolution of problem \eqref{inner4.18}, implying
$$|\bar{\phi}(\tilde y)|\lesssim \frac{\epsilon}{|\tilde y|^{2}}.$$
Since $\epsilon>0$ is arbitrary, it follows that $\bar{\phi}=0$, and thus $\bar{\Phi}=0$, yielding a contradiction. Therefore, (\ref{inner4.16}) holds and we have
\begin{align*}
\left\|(1+|y|^{a})\phi_{R}\right\|_{L^{\infty}(\mathbb{R}^4_+)}
&\lesssim \left\|(1+|y|^{2+a})g_{R}\right\|_{L^{\infty}(\mathbb{R}^4_+)} + \left\|(1+|\tilde y|^{1+a})h_{R}\right\|_{L^{\infty}(\mathbb{R}^3)}\\
&\lesssim R^2\lambda_R\left\|(1+|y|^{a})\phi_{R}\right\|_{L^{\infty}(\mathbb{R}^4_+)}+|c_R|,
\end{align*}
which yields
\begin{equation}\label{inner4.19}
\|\phi_{R}\|_{L^{2}(B_{2R}^+)}\lesssim R^{2}|c_{R}|\lesssim R^{2}\lambda_{R}.
\end{equation}
By (\ref{inner4.15}), relation (\ref{inner4.19}) further implies
$$\|\phi_{R}\|_{L^{2}(B_{2R}^+)}=\|\phi_{R}\|_{L^{2}(\mathbb{R}^4_+)}\to 0 \quad \text{as} \ R\to +\infty,$$
contradicting $\int_{B_{2R}^{+}} |\phi_{R}|^{2}dy = 1$. This completes the proof of \eqref{inner4.13}.
\end{proof}

Let $\eta_{0}(s)$ be the smooth cut-off function defined in \eqref{estimate2.6}. For a large fixed $l \gg 1$ independent of $R(t)$, define $\eta_{l}(y) = \eta_{0}(|y| - l)$. Consider the following initial-boundary value problem for $-2<\alpha<1$ and general $g(y,\tau)$, $h(\tilde{y},\tau)$ with $\|g\|_{2+\alpha,\nu}<+\infty$, $\|h\|_{1+\alpha,\nu}<+\infty$,
\begin{equation}\label{inner13:2}
\begin{cases}\phi_{\tau}=\Delta \phi+g(y, \tau) & \text { in }~  B_{2 R}^{+} \times\left(\tau_{0}, \infty\right), \\
\begin{aligned}
-\frac{d \phi}{d y_4}(\tilde{y}, 0, \tau) =&~2 U(\tilde{y}, 0)(1-\eta_{l}) \phi(\tilde{y}, 0, \tau)
\\&~+h(\tilde{y}, 0, \tau)
\end{aligned}  & \text { in }~ (\partial B_{2 R}^{+}  \cap\mathbb{R}^3) \times\left(\tau_{0}, \infty\right), \\
\phi\left(y, \tau_{0}\right)=0
& \text { in }~ B_{2 R(0)}^{+},\\
\phi(y,\tau) =0 &~\text{in}~  \partial B_{2 R}^{+}\backslash (\partial B_{2 R}^{+}  \cap\mathbb{R}^3)\times(\tau_{0}, \infty).
\end{cases}
\end{equation}
By standard parabolic theory \cite{Quittner-Souplet2019}, there exists a unique solution $\phi_{*}[g,h]$ to this problem, which defines a linear operator of $g(y,\tau)$ and $h(\tilde y,0,\tau)$. Moreover, $\phi_{*}[g,h]$ is radially symmetric in $\tilde{y}$ and satisfies
\begin{equation}\label{inner14:2}
|\phi_{*}[g,h]|\lesssim \Theta_{R}^{\alpha} \tau^{-\tilde{\nu}}\left(\|g\|_{2+\alpha,\nu}+\|h\|_{1+\alpha,\nu}\right),
\end{equation}
where $\Theta_{R}^{\alpha}$ is given in \eqref{Oct3-1}.

To verify this estimate, let $p_{1}(y)$ and $p_{2}(y)$
be radial positive solutions of
\begin{equation}\label{inner15:4}
\begin{cases}
\Delta p_{1}=-\frac{1}{1+|y|^{2+\alpha}} & \text{in}~  B_{2 R}^{+},\\
-\frac{d p_{1}}{d y_4}(\tilde y,0)=0 & \text{in}~  \partial B_{2 R}^{+}  \cap\mathbb{R}^3,\\
p_1(y) =0 &~\text{in}~  \partial B_{2 R}^{+}\backslash (\partial B_{2 R}^{+}  \cap\mathbb{R}^3)
\end{cases}
\end{equation}
and
\begin{equation}\label{inner16:4}
\begin{cases}\Delta p_{2}=0 & \text{in}~  B_{2 R}^{+},\\
-\frac{d p_{2}}{d y_4}(\tilde y,0)=\frac{1}{1+|\tilde{y}|^{1+\alpha}}
& \text{in}~  \partial B_{2 R}^{+}  \cap\mathbb{R}^3,\\
p_2(y) =0 &~\text{in}~  \partial B_{2 R}^{+}\backslash (\partial B_{2 R}^{+}  \cap\mathbb{R}^3),
\end{cases}
\end{equation}
respectively. For problem \eqref{inner15:4}, the fundamental solution satisfies
$$K_{1}(y)\asymp \frac{1}{|y|^{2}},$$
so the solution can be bounded as
\begin{align*}
p_{1}(y)&\lesssim\left| \int_{B_{2 R}^{+}} \frac{1}{|y-z|^{2}}\frac{1}{1+|z|^{2+\alpha}}dz\right|
\\&\lesssim \left|\left(\int_{|z|\leq2|y|}+\int_{2|y|\leq|z|\leq2R}\right)\frac{1}{|y-z|^{2}}\frac{1}{1+|z|^{2+\alpha}}dz\right|
\\&:=I_{1}+I_{2}.
\end{align*}
For $I_{1}$, since $\frac{1}{|y-z|^{2}}$ and $\frac{1}{1+|z|^{2+\alpha}}$ decrease in $z$, the Hardy-Littlewood rearrangement inequality gives
\begin{align*}
I_{1}\lesssim \left|\int_{|z|\leq2|y|} \frac{1}{|z|^{2}}\frac{1}{1+|z|^{2+\alpha}}dz\right|\lesssim
\int_{0}^{2|y|}\frac{r}{1+r^{2+\alpha}}dr\lesssim
\begin{cases}
\log R & \text{if}~ \alpha=0,\\
|y|^{-\alpha}& \text{if}~ \alpha<0.
\end{cases}
\end{align*}
For $I_{2}$, we have
\begin{align*}
I_{2}\lesssim\left|\int_{2|y|\leq|z|\leq2R}\frac{1}{|z|^{2}}\frac{1}{1+|z|^{2+\alpha}}dz\right|\lesssim \int_{2|y|}^{2R}\frac{r}{1+r^{2+\alpha}}dr
\lesssim
\begin{cases}
\log R & \text{if}~ \alpha=0,\\
|R|^{-\alpha}-|y|^{-\alpha} & \text{if}~ \alpha<0.
\end{cases}
\end{align*}
For the case $\alpha \in (0, 2)$, it is readily verifiable that $\frac{1}{\left(|\tilde y|^2+(1+y_4)^2\right)^{\alpha /2}}$ is a supersolution to \eqref{inner15:4}. Therefore, we conclude that
$$p_{1}(y)\lesssim \Theta_{R}^{\alpha},$$
where $\Theta_{R}^{\alpha}$ is as defined in \eqref{Oct3-1}.
Similarly, we have $p_{2}(y)\lesssim \Theta_{R}^{\alpha}$.

Noting that $|U(\tilde{y}, 0)|\lesssim \frac{1}{1+|\tilde{y}|^{2}}$, for sufficiently large $l$,  there exists a constant $0<\gamma<\frac{1}{2}$ such that
$$|2U(\tilde{y}, 0)(1-\eta_{l})(p_{1}(\tilde{y},0)+p_{2}(\tilde{y},0))|< \frac{\gamma}{1+|\tilde{y}|^{1+\alpha}}.$$
Let
$$\Phi_{1}(y,\tau)=c_{1}\tau^{-\tilde{\nu}}(p_{1}(y)+p_{2}(y)),$$
where $c_{1}>\|g\|_{2+\alpha,\nu}+\|h\|_{1+\alpha,\nu}$, $\tau>\tau_0$. Then for fixed sufficiently larger $\tau_0$, $\Phi_{1}(y,\tau)$ is a positive supersolution of
\begin{align*}
\begin{cases}
\partial_{\tau}\phi=\Delta\phi+g(y, \tau) & \text { in }~  B_{2 R}^{+} \times\left(\tau_{0}, \infty\right), \\
-\frac{d \phi}{d y_4}(\tilde{y}, 0, \tau)=2U(\tilde{y}, 0)(1-\eta_{l}) \phi(\tilde{y}, 0, \tau)+h(\tilde{y}, 0, \tau) &
 \text { in }~ (\partial B_{2 R}^{+}  \cap\mathbb{R}^3) \times\left(\tau_{0}, \infty\right).
 \end{cases}
 \end{align*}
By the parabolic comparison principle \cite{Evans1990}, we further obtain
\begin{align*}
|\phi(y,\tau)|\lesssim
\begin{cases}
\tau^{-\tilde{\nu}}\frac{1}{1+|y|^{\alpha}}\left(\|g\|_{2+\alpha,\nu}+\|h\|_{1+\alpha,\nu}\right), &  \alpha>0,\\
\tau^{-\tilde{\nu}}\log R\left(\|g\|_{2+\alpha,\nu}+\|h\|_{1+\alpha,\nu}\right), &  \alpha=0,\\
\tau^{-\tilde{\nu}}R^{-\alpha}\left(\|g\|_{2+\alpha,\nu}+\|h\|_{1+\alpha,\nu}\right), & \alpha<0 .
\end{cases}
\end{align*}
Therefore, the solution $\phi_{*}[g,h]$ to problem \eqref{inner13:2} satisfies estimate \eqref{inner14:2}.

\medskip

\begin{proof}[Proof of Lemma \ref{lemma4.1}]
Consider the functions
$$\overline{g}=g^{0}-c_{0}(\tau)Z_{0}(y),\quad \overline{h}=h^{0},\quad
c_{0}(\tau)=\frac{\int_{B_{2R}^{+}} g^{0}(y, \tau) Z_{0}(y) dy}{\int_{B_{2R}^{+}}Z_{0}^{2} dy}.$$
Let $\phi_{*}[\overline{g},\overline{h}]$ be the solution to \eqref{inner13:2}. Since $g^{0}, h^{0}$ are radially symmetric in $\tilde{y}$, so is $\phi_{*}[\overline{g},\overline{h}]$.
Set
$$\phi=\phi_{*}[\overline{g},\overline{h}]+\tilde{\phi},\quad c(\tau)=c_{0}(\tau)+\tilde{c}(\tau).$$
Then \eqref{inner8:2} reduces to
\begin{equation}\label{inner15:2}
\begin{cases}
\tilde{\phi}_{\tau}=\Delta \tilde{\phi}-\tilde{c}(\tau) Z_{0}(y) & \text { in }~  B_{2 R}^{+} \times\left(\tau_{0}, \infty\right), \\
\begin{aligned}
-\frac{d \tilde{\phi}}{d y_4}(\tilde{y}, 0, \tau)=&~2U(\tilde{y}, 0) \tilde{\phi}(\tilde{y}, 0, \tau)+\tilde{h}(\tilde{y}, \tau)
\end{aligned}& \text { in }~ (\partial B_{2R}^{+}\cap\mathbb{R}^3)\times\left(\tau_{0},\infty\right),\\
\tilde{\phi}\left(y, \tau_{0}\right)=0 & \text { in }~  B_{2 R(0)}^{+},\\
\tilde{\phi}(\tilde{y},\tau)=0 &~ \text{in}~ \partial B_{2 R}^{+}\backslash (\partial B_{2 R}^{+}  \cap\mathbb{R}^3)\times\left(\tau_{0},\infty\right), \end{cases}
\end{equation}
where $\tilde{h}(\tilde{y}, \tau)=2U(\tilde{y}, 0)\eta_{l}\phi_{*}[\overline{g},\overline{h}].$
Note that $\tilde{h}(\tilde{y}, \tau)$ is radially symmetric in $\tilde{y}$ and compactly supported, with its support size controlled by $\overline{g}$ and $\overline{h}$. For any $m>0$, this yields the estimate
\begin{equation}\label{inner16:2}
\|\tilde{h}\|_{m,\nu}\lesssim \|\phi_{*}[\overline{g},\overline{h}]\|_{0,\nu}.
\end{equation}

We next solve problem \eqref{inner15:2} under the additional orthogonality condition
\begin{equation}\label{inner17:2}
\int_{B_{2 R}^{+}} \tilde{\phi}(y, \tau) Z_{0}(y) d y+\lambda_{0}\int_{\tau_{0}}^{\tau}\int_{B_{2 R}^{+}}\tilde{\phi}(\tilde{y},0, s)\widetilde {Z}_{0}(\tilde y) d \tilde{y}ds=0,
\quad  \tau\in(\tau_{0},+\infty).
\end{equation}
Problem \eqref{inner15:2} together with \eqref{inner17:2} is equivalent to solving \eqref{inner15:2}
with $\tilde{c}(\tau) := \tilde{c}[\tilde{\phi}, \tilde{h}](\tau)$ explicitly given by
\begin{equation}\label{inner18:2}
\tilde{c}(\tau)=\frac{\int_{\partial{B}_{2R}^{+}\cap{\mathbb{R}}^3}\tilde{h}(\tilde{y},\tau)\widetilde{Z}_{0}(\tilde y)d\tilde{y}}{\int_{B_{2R}^{+}} Z_{0}^{2} d y}.
\end{equation}
From (\ref{inner17:2}) and \eqref{inner18:2}, we have
$$
\int_{\partial{B}_{2R}^{+}\cap{\mathbb{R}}^3}\tilde{\phi}(\tilde{y},0, \tau)\widetilde {Z}_{0}(\tilde y) d \tilde{y}=0.
$$

Testing \eqref{inner15:2} against $\tilde{\phi}$ and integrating gives
$$\frac{1}{2}\partial_{\tau}\int_{B_{2R}^{+}}\tilde{\phi}^{2}dy+ B(\tilde{\phi},\tilde{\phi})=\int_{\partial{B}_{2R}^{+}\cap{\mathbb{R}}^3} \tilde {h}(\tilde{y},\tau)\tilde{\phi}(\tilde{y},0, \tau)d\tilde{y}-\tilde c(\tau)\int_{B_{2R}^{+}}Z_0(y)\tilde\phi(y,\tau)dy,$$
By Lemma \ref{lemma4.2}, there exists $\gamma > 0$ such that
\begin{equation}\label{inner19:2}
\frac{1}{2}\partial_{\tau}\int_{B_{2R}^{+}}|\tilde{\phi}|^{2}dy+ \frac{\gamma}{R^{2}}\int_{B_{2R}^{+}}|\tilde{\phi}|^{2}dy\lesssim \int_{\partial{B}_{2R}^{+}\cap{\mathbb{R}}^3} \tilde {h}\tilde{\phi}d\tilde{y}-\tilde c(\tau)\int_{B_{2R}^{+}}Z_0\tilde\phi dy.
\end{equation}
From \eqref{inner16:2} with $m=0$ and \eqref{inner18:2} and , we get
$$|\tilde{c}(\tau)|\leq \tau^{-\tilde{\nu}}\|\phi_{*}[\overline{g},\overline{h}]\|_{0,\nu}.$$
For sufficiently large $m$, \eqref{inner16:2} further gives
$$\int_{\partial{B}_{2R}^{+}\cap{\mathbb{R}}^3} \tilde h^{2}d\tilde{y}\lesssim \tau^{-2\tilde\nu}\|\phi_{*}[\overline{g},\overline{h}]\|_{0,\nu}^{2}.$$
Using the initial condition $\tilde{\phi}(\cdot,\tau_{0})=0$ and Gronwall’s inequality, we deduce from \eqref{inner19:2} that for all sufficiently small $\epsilon > 0$,
the following $L^{2}$-estimate holds
\begin{align*}
\|\tilde{\phi}(\cdot,\tau)\|_{L^{2}(B_{2R}^{+})}\lesssim \tau^{-\tilde{\nu}}R^{2}\left(\frac{1}{\epsilon}\left\|\phi_{*}[\overline{g},\overline{h}]\right\|_{0,\nu}+\epsilon M^{2}R^{-2}\left\|\tau^{\tilde\nu}\tilde{\phi}(\cdot,\tau)\right\|_{L^{\infty}(B_{M}^+)}\right),\quad \forall \tau>\tau_{0}.
\end{align*}
Applying standard parabolic estimates to \eqref{inner15:2} and taking $\epsilon = M^{-3}$,
we obtain that for any sufficiently large fixed $M > 0$,
\begin{align*}
\|\tilde{\phi}(\cdot,\tau)\|_{L^{\infty}(B_{M})}\lesssim \tau^{-\tilde{\nu}}R^{2}\|\phi_{*}[\overline{g},\overline{h}]\|_{0,\nu},\quad  \forall \tau>\tau_{0}.
\end{align*}
Outside $B_M^+$, the solution of \eqref{inner15:2} is bounded above by a barrier of order $\tau^{-\tilde{\nu}} R^2 |y|^{-2}$. Combining this with local parabolic gradient estimates \cite{Quittner-Souplet2019}, we conclude that
\begin{equation}\label{inner34:0}
(1+|y|)|\nabla\tilde{\phi}(y,\tau)|+|\tilde{\phi}(y,\tau)|\lesssim \tau^{-\tilde{\nu}} \frac{R^{2}}{|y|^{2}}\|\phi_{*}[\overline{g},\overline{h}]\|_{0,\nu}.
\end{equation}
By \eqref{inner16:2}, \eqref{inner34:0}, and the estimate for $|\phi_{*}[\overline{g},\overline{h}]|$, the function
\begin{align*}
\phi=\phi^{0}[g^{0},h^{0}]=\phi_{*}[\overline{g},\overline{h}]+\tilde{\phi}
\end{align*}
solves problem \eqref{inner8:2} and satisfies the desired estimate \eqref{inner9:2}, i.e.,
$$(1+|y|)|\nabla \phi(y,\tau)|+|\phi(y,\tau)|\lesssim \tau^{-\tilde{\nu}} \frac{R^{2}}{1+|y|^{2}}\Theta_{R}^{\alpha}(\|g^{0}\|_{2+\alpha, \nu}+\|h^{0}\|_{1+\alpha,\nu}).$$

Finally, from \eqref{inner18:2}, one has
\begin{align*}
c(\tau)=\frac{\int_{B_{2R}^{+}} g^{0}(y, \tau) Z_{0}(y) dy+\int_{\partial{B}_{2R}^{+}\cap{\mathbb{R}}^3}\tilde{h}(\tilde{y}, 0, \tau) \widetilde Z_{0}(\tilde{y}) d \tilde{y}}{\int_{B_{2R}^{+}}Z_{0}^{2} dy},
\end{align*}
Applying \eqref{inner14:2} and \eqref{inner16:2}, we obtain
\begin{align*}
|\tilde{c}(\tau)|
&\lesssim\tau^{-\tilde{\nu}}(\|\bar{g}\|_{2+\alpha,\nu}+\|\bar{h}\|_{1+\alpha,\nu})
\\&\lesssim\tau^{-\tilde{\nu}}\left(\|h^{0}\|_{1+\alpha,\nu}
+\left\|g^{0}-Z_{0}\int_{B_{2R}^{+}} g^{0} Z_{0} dy\right\|_{2+\alpha,\nu}\right).
\end{align*}
Thus, estimates (\ref{inner10:2}) and \eqref{inner9:2} hold.
\end{proof}
\begin{lemma}\label{lemma4.3}
Assume $0<\sigma_1<1$, $\nu>0$, $\tilde\nu>0$ with $\mu_0^\nu(t)\sim\tau^{-\tilde{\nu}}$. Assume further that $\|g^{0}\|_{2+\sigma_1,\nu}<+\infty$, $\|h^{0}\|_{1+\sigma_1,\nu}<+\infty$, and the orthogonality condition
\begin{align*}
\int_{B_{2R}^{+}} g^{0}(y, \tau)Z_4(y) d y+\int_{\partial B_{2R}^{+}\cap{\mathbb{R}}^3} h^{0}(\tilde{y},0, \tau) Z_4(\tilde{y},0) d \tilde{y}=0,\quad \forall \tau\in(\tau_{0}, \infty).
\end{align*}
Then there exist solutions $\phi=\phi^{0}[g^{0},h^{0}]$ and $c=c[g^{0},h^{0}]$ to problem \eqref{inner4.4}, which induce linear operators of $g^{0}$ and $h^{0}$. The function $\phi=\phi^{0}$ satisfies
\begin{equation}\label{inner25:2}
|\phi(y,\tau)|\lesssim \tau^{-\tilde{\nu}} \left(\frac{R(t)^{4-\sigma_1}}{1+|y|^4}\|g^{0}\|_{2+\sigma_1,\nu}+\frac{R(t)^{3-\sigma_1}}{1+|y|^3}\|h^{0}\|_{1+\sigma_1,\nu}\right)
\end{equation}
in $B_{\frac{3}{2}R}^{+}\times(\tau_{0},+\infty)$,
and $c(\tau)$ satisfies estimate (\ref{inner7:4}).
\end{lemma}

\begin{proof}
Extend $g^{0}$, $h^{0}$ to be zero outside $B_{2R}^{+}$ and $\partial B_{2R}^{+}\cap{\mathbb{R}}^3$, respectively, and still denote as $g^{0}$, $h^{0}$.
By standard elliptic estimates \cite{Evans1990}, the solution $(G_0, H_0) := L_U^{-1}[g^0, h^0]$ to
\begin{align*}
\begin{cases}
\Delta \phi=-g^{0}(y,\tau) & \text { in }~ \mathbb{R}^4_{+} ,\\
-\frac{d\phi}{d y_4}(\tilde{y}, 0, \tau)=2U(\tilde{y}, 0) \phi(\tilde{y}, 0, \tau)+h^{0}(\tilde{y}, \tau) &  \text { in}~  \mathbb{R}^3,\\
\lim_{|y|\to +\infty}\phi(y)=0
 \end{cases}
 \end{align*}
satisfies
\begin{align*}
|G_{0}(y,\tau)|\lesssim \tau^{-\tilde{\nu}} \frac{1}{1+|y|^{\sigma_1}}\left(\|g^{0}\|_{2+\sigma_1,\nu}+\|h^{0}\|_{1+\sigma_1,\nu}\right),
\end{align*}

\begin{equation}\label{inner28:6}
|H_{0}(\tilde{y},\tau)|\lesssim \tau^{-\tilde{\nu}} \frac{1}{1+|\tilde{y}|^{\sigma_1}}\left(\|g^{0}\|_{2+\sigma_1,\nu}+\|h^{0}\|_{1+\sigma_1,\nu}\right).\end{equation}

Consider the initial-boundary value problem
\begin{align*}
\begin{cases}
\Phi_{\tau}=\Delta \Phi+G_{0}(y, \tau)-c_{0}(\tau) Z_{0}(y) & \text { in }~ B_{4 R}^{+} \times\left(\tau_{0}, \infty\right),\\
\begin{aligned}
-\frac{d \Phi}{d y_4}(\tilde{y}, 0, \tau)=&~2U(\tilde{y}, 0) \Phi(\tilde{y}, 0, \tau)
+H_{0}(\tilde{y}, \tau)
\end{aligned}  & \text { in }~  (\partial B_{4R}^{+}\cap\mathbb{R}^3)\times\left(\tau_{0},\infty\right),
\\ \Phi\left(\cdot, \tau_{0}\right)=0 & \text { in }~ B_{4R(0)}^{+},\\
\Phi(\tilde y,\tau)=0 &~ \text{in}~ \partial B_{4R}^{+}\backslash (\partial B_{4R}^{+}\cap\mathbb{R}^3)\times\left(\tau_{0},\infty\right).
\end{cases}
\end{align*}
By Lemma \ref{lemma4.1} and estimate \eqref{inner14:2}, there exist solutions $\Phi_{0}[G_{0},H_{0}]$ and $c_{0}=c_{0}[G_{0},H_{0}]$ to this problem, inducing linear operators on $G_{0}$ and $H_{0}$, and satisfying
\begin{align*}
|\Phi_{0}(y,\tau)|\lesssim \tau^{-\tilde{\nu}} \frac{R(t)^{2}}{1+|y|^{2}}\left(R(t)^{2-\sigma_1}\|G_{0}\|_{\sigma_1, \nu}+R(t)^{1-\sigma_1}\|H_{0}\|_{\sigma_1,\nu}\right)
\end{align*}
and
\begin{equation}\label{inner27:4}
\begin{aligned}
&\left|c_{0}(\tau)-\frac{\int_{B_{2R}^{+}}G_{0}(\tilde{y},\tau)Z_{0}(y) dy}{\int_{B_{2R}^{+}}Z_{0}^{2}dy}\right|
 \\&\lesssim \tau^{-\tilde{\nu}}\left( R(t)^{1-\sigma_1}\left\|H_{0}\right\|_{\sigma_1,\nu}+ R(t)^{2-\sigma_1}\left\|G_{0}-Z_{0}\int_{B_{2R}^{+}}G_{0}(\tilde{y},\tau)Z_{0}(y)dy\right\|_{\sigma_1,\nu}\right).
 \end{aligned}
\end{equation}
Recall that
\begin{align*}
L_{U}[Z_{0}]=
\begin{cases}
\Delta Z_{0}=\lambda_{0}Z_{0}(y) & \text{ in }~ \mathbb{R}_{+}^4 ,\\
-\frac{dZ_{0}}{d y_4}(\tilde{y},0) -2U(\tilde{y}, 0)Z_{0}(\tilde{y},0)=0 & \text { in }~ \mathbb{R}^3.
\end{cases}
\end{align*}
Integrating by parts, we get
\begin{align*}
\lambda_{0}\int_{B_{2R}^{+}} G_{0}(y,\tau)Z_{0}(y)dy
&= \int_{B_{2R}^{+}} H_{0}(\tilde{y},\tau) L_{U}[Z_{0}]|_{B_{2R}^{+}}(\tilde{y}) \, d \tilde{y} \\
&=\int_{B_{2R}^{+}} g^{0}(y,\tau) Z_{0}(y) \, dy + \int_{\partial{B}_{2R}^{+}\cap{\mathbb{R}}^3} h^{0}(\tilde{y},0,\tau) \widetilde{Z}_{0}(\tilde{y}) \, d \tilde{y},
\end{align*}
hence
\begin{equation}\label{inner32:5}
\int_{B_{2R}^{+}} G_{0}(y,\tau)Z_{0}(y)dy = \lambda_{0}^{-1} \left[ \int_{B_{2R}^{+}} g^{0}(y,\tau) Z_{0}(y) \, dy + \int_{\partial{B}_{2R}^{+}\cap{\mathbb{R}}^3} h^{0}(\tilde{y},0,\tau) \widetilde{Z}_{0}(\tilde{y}) \, d \tilde{y} \right].
\end{equation}

From the definition of $L_{U}^{-1}$, we have
\[
\widetilde{Z}_{0}(\tilde{y}) = \lambda_{0} L_{U}^{-1}|_{B_{2R}^{+}}[\widetilde{Z}_{0}](\tilde{y}), \quad Z_{0}(y) = L_{U}^{-1}|_{\partial{B}_{2R}^{+}\cap{\mathbb{R}}^3}[0](y).
\]
Thus,
\[
\|H_{0} \|_{\sigma_1,\nu} = \left\| L_{U}^{-1}|_{\partial{B}_{2R}^{+}\cap{\mathbb{R}}^3}[g^{0}] \right\|_{\sigma_1,\nu} \lesssim\|g^{0}\|_{2+\sigma_1,\nu}+\|h^{0}\|_{1+\sigma_1,\nu},
\]
and
\begin{align*}
&\left\| G_{0}-Z_{0}\int_{B_{2R}^{+}}G_{0}(\tilde{y},\tau)Z_{0}(y)dy\right\|_{\sigma_1,\nu}
\\&= \left\| L_{U}^{-1}|_{B_{2R}^{+}} \left[ g^{0} - \lambda_{0}Z_{0} \int_{B_{2R}^{+}} G_{0}(y,\tau)Z_{0}(y)dy\right] \right\|_{\sigma_1,\nu} \\
&\lesssim \left\| g^{0} - Z_{0} \left[ \int_{B_{2R}^{+}} g^{0}(y,\tau) Z_{0}(y) \, dy + \int_{\partial{B}_{2R}^{+}\cap{\mathbb{R}}^{3}} h^{0}(\tilde{y},0,\tau) \widetilde{Z}_{0}(\tilde{y}) \, d\tilde{y} \right] \right\|_{2+\sigma_1,\nu},
\end{align*}
where the last inequality uses \eqref{inner32:5} and \eqref{inner28:6}.

Now fix a unit vector $e=(\tilde{e},e_4)$, a large number $\rho>0$ with $\rho<2R$ and $\tau_{1}\geq \tau_{0}$. Define the following scaled functions
\begin{align*}
\Phi_{\rho}(z,t)&:=\Phi(\rho e+\rho z,\tau_{1}+\rho^{2}t),
\qquad\qquad G_{\rho}(z,t):=\rho^{2}[G_{0}(\rho e+\rho z,\tau_{1}+\rho^{2}t),
\\H_{\rho}(\tilde{z},t)&:=\rho[H_{0}(\rho \tilde{e}+\rho \tilde{z},\tau_{1}+\rho^{2}t)-c_{0}(\tau_{1}+\rho^{2}t)Z_{0}(\rho \tilde{e}+\rho \tilde{z},0)],
\end{align*}
where $z=(\tilde{z},z_4)$, $\tilde{z}\in \mathbb{R}^3$, $z_4\geq 0$.
Then $\Phi_{\rho}(z,t)$ satisfies
\begin{align*}
\begin{cases}
\partial_{t}\Phi_{\rho}=\Delta_{z} \Phi_{\rho}+G_{\rho}(z, t) &  \text{in}~  B_{1}^{+}(0) \times(0,2),\\
-\frac{d \Phi_{\rho}}{d z_4}(\tilde{z}, 0, t)=A_{\rho}(\tilde{z},t)
\Phi_{\rho}(\tilde{z}, 0, t)+H_{\rho}(\tilde{z},t) &
\text{in}~  (\partial B_{1}^{+}(0)\cap\mathbb{R}^3) \times(0,2),
\end{cases}
\end{align*}
with $A_{\rho}(\tilde{z},t)=O(\rho^{-1})$ uniformly in $B_{1}^{+}(0) \times(0,+\infty)$.
By parabolic estimates \cite{Quittner-Souplet2019}, we have
\begin{align*}
\|\nabla_{z}\Phi_{\rho}\|_{L^{\infty}(B_{\frac{1}{2}}^{+}(0)\times(1,2))}
\lesssim&~ \|\Phi_{\rho}\|_{L^{\infty}(B_{1}^{+}(0)\times(0,2))}+\|G_{\rho}\|_{L^{\infty}(B_{1}^{+}(0)\times(0,2))}
\\&~+\|H_{\rho}\|_{L^{\infty}((\partial B_{1}^{+}(0)\cap \mathbb{R}^3)\times(0,2))}.
\end{align*}
Moreover,
\begin{align*}
&\|G_{\rho}\|_{L^{\infty}(B_{1}^{+}(0)\times(0,2))}\lesssim \rho^{2-\sigma_1}\tau_{1}^{-\tilde\nu}\|G_{0}\|_{\sigma_1,\nu},
\\& \|H_{\rho}\|_{L^{\infty}((\partial B_{1}^{+}(0)\cap \mathbb{R}^3)\times(0,2))}\lesssim \rho^{1-\sigma_1}\tau_{1}^{-\tilde\nu}\|H_{0}\|_{\sigma_1,\nu},
\\& \|\Phi_{\rho}\|_{L^{\infty}(B_{1}^{+}(0)\times(0,2))}\lesssim \tau_{1}^{-\tilde\nu}K(\rho),
\end{align*}
where
\begin{equation}\label{inner29:2}
K(\rho)=\frac{R(t)^{2}}{\rho^{2}}\left(R(t)^{2-\sigma_1}\|g^{0}\|_{2+\sigma_1, \nu}+R(t)^{1-\sigma_1}\|h^{0}\|_{1+\sigma_1,\nu}\right).
\end{equation}
In particular,
$$\rho|\nabla_{y}\Phi(\rho e, \tau_{1}+\rho^{2})|=|\nabla\Phi_{\rho}(0,1)|\lesssim \tau_{1}^{-\tilde\nu}K(\rho).$$
Choosing $\tau_{0}\geq R^{2}$, for any $\tau>2\tau_{0}$ and $|y|\leq 4R$, one has
$$
(1+|y|)|\nabla\Phi(y,\tau)|\lesssim \tau^{-\tilde{\nu}}K(|y|).
$$
Similar estimates hold for $\tau\leq2\tau_{0}$ by parabolic estimates up to the initial time (with zero initial data).
Since $G_{0}$ is $C^{1}$-smooth in $(\tilde{y},y_4)$ and $H_{0}$ is $C^{1}$-smooth in $\tilde{y}$ with
$$
\|\nabla_{y} G_{0}\|_{1+\sigma_1,\nu}\lesssim\|g^{0}\|_{2+\sigma_1,\nu}+\|h^{0}\|_{1+\sigma_1,\nu}, \quad \|\nabla_{\tilde{y}} H_{0}\|_{1+\sigma_1,\nu}\lesssim\|g^{0}\|_{2+\sigma_1,\nu}+\|h^{0}\|_{1+\sigma_1,\nu},
$$
it follows that for all $\tau>\tau_{0}$ and $|y|<2R$,
$$(1+|y|^{2})|D_{y}^{2}\Phi(y,\tau)|\lesssim \tau^{-\tilde{\nu}}K(|y|),$$
where $K$ is defined by \eqref{inner29:2}. Thus, in $B_{2 R}^{+}$
\begin{align*}
&(1+|y|^{2})|D_{y}^{2}\Phi(y,\tau)|+(1+|y|)|\nabla\Phi(y,\tau)|+|\Phi(y,\tau)|
\\&\lesssim\tau^{-\tilde{\nu}} \frac{R(t)^{2}}{1+|y|^{2}}\left(R(t)^{2-\sigma_1}\|g\|_{2+\sigma_1, \nu}+R(t)^{1-\sigma_1}\|h\|_{1+\sigma_1,\nu}\right).
\end{align*}
This yields
$$|L_{U}[\Phi](\cdot,\tau)|\lesssim \tau^{-\tilde{\nu}}\left(\frac{R(t)^{4-\sigma_1}}{1+|y|^4}\|g^{0}\|_{2+\sigma_1,\nu}+\frac{R(t)^{3-\sigma_1}}{1+|y|^3}\|h^{0}\|_{1+\sigma_1,\nu}\right)~ \text{in} \ B_{2 R}^{+}.$$
Define
\begin{align*}
\phi_{0}[g^{0},h^{0}]:=
\begin{cases}
L_{U}[\Phi]|_{B_{2 R}^{+}} &~ \text{in}~  B_{2 R}^{+} \times\left(\tau_{0}, \infty\right), \\
L_{U}[\Phi]|_{\partial B_{2R}^{+}\cap\mathbb{R}^3} & \text { in }~  (\partial B_{2R}^{+}\cap\mathbb{R}^3)\times\left(\tau_{0},\infty\right).
\end{cases}
\end{align*}
Then $\phi=\phi^{0}$ solves problem \eqref{inner4.4} with $c(\tau)=\lambda_{0}c_{0}(\tau)$ and satisfies
\begin{equation}\label{inner33:4}
|\phi^{0}[g^{0},h^{0}](y,\tau)|\lesssim\tau^{-\tilde{\nu}} \left(\frac{R(t)^{4-\sigma_1}}{1+|y|^4}\|g^{0}\|_{2+\sigma_1,\nu}+\frac{R(t)^{3-\sigma_1}}{1+|y|^3}\|h^{0}\|_{1+\sigma_1,\nu}\right) ~\text{in}~   B_{2 R}^{+}\times(\tau_{0},+\infty).
\end{equation}
Therefore, \eqref{inner25:2} holds for $\phi=\phi^0$, and by \eqref{inner27:4}, \eqref{inner7:4} holds for $c(\tau)=\lambda_{0}c_{0}(\tau)$.
\end{proof}

\begin{remark}
Let $0<\sigma_1<1$, $\nu>0$, define
$$\|\phi^0\|_{\sigma_1,\nu}:= \sup _{\tau>\tau_{0}}\sup_{y\in B_{2 R}^{+}} \mu_0^{-\nu}(1+|y|^{\sigma_1}) \left[|\phi(y,\tau)|+(1+|y|)|\nabla \phi(y,\tau)|\right].$$
Lemma \ref{lemma4.3} implies that
$$\|\phi^0\|_{\sigma_1,\nu}\lesssim \left(\|g^{0}\|_{2+\sigma_1, \nu}+\|h^{0}\|_{1+\sigma_1,\nu}\right).$$
\end{remark}

\subsection{The case of modes $j\geq1$}
In this subsection, we construct the solution $\phi^{\perp}$ to problem \eqref{inner4.6}. Since ODE techniques are inapplicable here, we employ a blow-up argument.

For generality, we consider a more general right-hand side than strictly required for proving Proposition \ref{proposition4.2}. Hereafter, $\phi$ denotes $\phi^{\perp}$ in the proposition, and we extend $g^{\perp}$, $h^{\perp}$ by zero outside $B_{2R}^{+}$ and $\partial{B}_{2R}^{+}\cap{\mathbb{R}}^3$, respectively, for all $\tau > \tau_0$, i.e., we solve the following problem
\begin{equation}\label{inner4}
\begin{cases}
\phi_{\tau}=\Delta\phi+g^{\perp}(y,\tau) & \text { in }~  \mathbb{R}_{+}^4\times\left(\tau_{0}, \infty\right), \\
-\frac{d \phi}{d y_4}(\tilde{y},0,\tau)=2U(\tilde{y}, 0) \phi(\tilde{y}, 0, \tau)+h^{\perp}(\tilde{y}, 0, \tau) &  \text { in }~ \mathbb{R}^3\times\left(\tau_{0}, \infty\right) , \\
\phi\left(\cdot, \tau_{0}\right)=0 &  \text { in }~ \mathbb{R}_{+}^4.
\end{cases}
\end{equation}
By construction, $\phi = \phi^{\perp}$ satisfies the orthogonality
$$\int_{\mathbb{R}^3}\phi(\tilde{y},0,\tau)\widetilde{Z}_{0}(\tilde y)d\tilde{y}=0.$$
Define the weighted norms
\begin{align*}
&\|g^{\perp}\|_{3+\sigma_2,\tau_{1}}:= \sup _{\tau \in\left(\tau_{0}, \tau_{1}\right)}\sup_{y\in \mathbb{R}^4_{+}} \tau^{\tilde\nu}(1+|y|^{3+\sigma_2}) |g^{\perp}(y,\tau)|,
\\& \|h^{\perp}\|_{2+\sigma_2,\tau_{1}}:= \sup _{\tau \in\left(\tau_{0}, \tau_{1}\right)}\sup_{\tilde{y}\in \mathbb{R}^3} \tau^{\tilde\nu}(1+|\tilde{y}|^{2+\sigma_2}) |h^{\perp}(\tilde{y},\tau)|,
\\& \|\phi\|_{1+\sigma_2,\tau_{1}}:=\sup _{\tau \in\left(\tau_{0}, \tau_{1}\right)}\sup_{y\in \mathbb{R}^4_{+}} \tau^{\tilde\nu}(1+|y|^{1+\sigma_2})|\phi(y,\tau)|.
\end{align*}
For problem \eqref{inner4}, we establish the following lemma.

\begin{lemma}\label{lemma4}
Assume $0<\sigma_2<1$, $\nu>0$, $\tilde\nu>0$ with $\mu_0^\nu(t)\sim\tau^{-\tilde{\nu}}$, $\|g^{\perp}\|_{3+\sigma_2,\nu}<+\infty$, $\|h^{\perp}\|_{2+\sigma_2,\nu}<+\infty$, and
\begin{equation*}
\int_{B_{2R}^{+}} g^{\perp}(y, \tau)Z_{i}(y) d y+\int_{\partial B_{2R}^{+}\cap{\mathbb{R}}^3} h^{\perp}(\tilde{y}, 0, \tau) Z_{i}(\tilde{y},0) d \tilde{y}=0
\end{equation*}
for all $ \tau\in(\tau_{0}, \infty)$ and $i=1,2,3.$ Then for sufficiently large $\tau_{1}>\tau_0$, there exists a solution $\phi(y,\tau)$ to problem \eqref{inner4} satisfying
\begin{equation}\label{inner6}
\\\|\phi\|_{1+\sigma_2,\tau_{1}}\lesssim \|g^{\perp}\|_{3+\sigma_2,\tau_{1}}+\|h^{\perp}\|_{2+\sigma_2,\tau_{1}}.
\end{equation}
\end{lemma}

\begin{proof}
We prove \eqref{inner6} via blow-up argument. First, we claim that for any given $\tau_{1}>\tau_0$,
\begin{equation*}
\|\phi\|_{1+\sigma_2,\tau_{1}}<+\infty.
\end{equation*}	
Indeed, standard parabolic theory yields the local boundedness in time and space of $\phi(y,\tau)$. Precisely, for any $R_{0}>0$, there exists $K=K(R_{0},\tau_{1})$ such that
$$|\phi(y,\tau)|<K   \quad\text{in} \ B^{+}_{R_{0}}(0)\times(\tau_{0},\tau_{1}].$$
For a sufficiently large fixed $R_{0}$, choosing appropriate constants $K_{1}$ and $K_{2}$ such that $\frac{K_{1}}{\rho^{1+\sigma_2}}-\frac{K_{2}y_4}{\rho^{2+\sigma_2}}$ (where $\rho=|y|$) is a positive supersolution to \eqref{inner4} for $\rho>R_{0}$. This yields the estimate
$$|\phi|\lesssim \frac{K_{1}}{\rho^{1+\sigma_2}}-\frac{K_{2}y_4}{\rho^{2+\sigma_2}},$$
and consequently,
$$\|\phi\|_{1+\sigma_2,\tau_{1}}<+\infty.$$

Next, we establish the orthogonality condition,
\begin{equation}\label{inner11}
\int_{\mathbb{R}_{+}^4} \phi(y, \tau) Z_{i}(y) d y=0, \quad \forall \tau \in\left(\tau_{0}, \tau_{1}\right),\ i=1,2,3.
\end{equation}
To verify this, test problem \eqref{inner4} against $Z_{i}(y)\eta(y)$, where
$$\eta(y)=\eta_{0}\left(\frac{|y|}{R_{1}}\right),\quad R_{1}>0,$$
and $\eta_{0}$ is the smooth cut-off function as defined in \eqref{estimate2.6}.
Integrating over $\mathbb{R}_{+}^4$ yields
\begin{align*}
&\int_{\mathbb{R}_{+}^4} \phi(y, \tau) \eta(y) Z_{i}(y) dy \\
&=\int_{\tau_0}^{\tau} \int_{\mathbb{R}_{+}^4} \left[\Delta \phi(y,s) \eta(y) Z_{i}(y) + g^{\perp}(y,s) \eta(y) Z_{i}(y)\right] dyds \\
&=\int_{\tau_0}^{\tau} \int_{\mathbb{R}_{+}^4} \left[\phi(y,s) \Delta\left(\eta Z_{i}\right)(y) + g^{\perp}(y,s) \eta(y) Z_{i}(y)\right] dyds \\
&\quad +\int_{\tau_0}^{\tau} \int_{\mathbb{R}^3} h^{\perp}(\tilde{y},0,s) \eta(\tilde{y},0) Z_{i}(\tilde{y},0) d\tilde{y}ds \\
&\quad +\int_{\tau_0}^{\tau} \int_{\mathbb{R}^3} \phi(\tilde{y},0,s) \left[2U(\tilde{y},0) \eta(\tilde{y},0) Z_{i}(\tilde{y},0) - \frac{\partial}{\partial \nu}\left(\eta Z_{i}\right)(\tilde{y},0)\right] d\tilde{y}ds.
\end{align*}
On the other hand, direct computation shows that
\begin{align*}
&\int_{\mathbb{R}_{+}^4} \left[ \phi(y,s) \Delta\left(\eta Z_{i}\right)(y) + g^{\perp}(y,s) \eta(y) Z_{i}(y) \right] dy + \int_{\mathbb{R}^3} h^{\perp}(\tilde{y},0,s) \eta(\tilde{y},0) Z_{i}(\tilde{y},0) d\tilde{y}
\\&+ \int_{\mathbb{R}^3} \phi(\tilde{y},0,s) \left[ 2U(\tilde{y},0) \eta(\tilde{y},0) Z_{i}(\tilde{y},0) - \frac{\partial}{\partial \nu}\left(\eta Z_{i}\right)(\tilde{y},0) \right] d\tilde{y}
\\&= O(R_{1}^{-a})
\end{align*}
uniformly for $s \in (\tau_{0}, \tau)$ and $\tau \in (\tau_0, \tau_1)$ with some constant $a > 0$. Letting $R_1 \to +\infty$, we conclude \eqref{inner11}.

Now we prove \eqref{inner6} by contradiction. Suppose there exist sequences $\{\tau_1^k\}_{k=1}^{\infty} \to \infty$, $\{\phi_k\}_{k=1}^{\infty}$, $\{g_k\}_{k=1}^{\infty}$, $\{h_k\}_{k=1}^{\infty}$ satisfying
\begin{align*}
\begin{cases}
\partial_{\tau}\phi_{k}=\Delta \phi_{k}+g_{k}(y, \tau) & \text { in }~ \mathbb{R}_{+}^4 \times\left(\tau_{0}, \infty\right), \\
-\frac{d \phi_{k}}{d y_4}(\tilde{y}, 0, \tau)=2U(\tilde{y}, 0) \phi_{k}(\tilde{y}, 0, \tau)+h_{k}(\tilde{y}, 0, \tau) & \text { in }~ \mathbb{R}^3 \times\left(\tau_{0}, \infty\right), \\
\int_{\mathbb{R}_{+}^4} \phi_{k}(y, \tau) Z_{i}(y) d y=0, &~ \forall \tau \in \left(\tau_{0}, \tau_{1}^{k}\right),~i=1,2,3,
\\ \phi_{k}\left(\cdot, \tau_{0}\right)=0 & \text { in }~ \mathbb{R}_{+}^4 \\  \end{cases}
\end{align*}
with
\begin{equation}\label{inner13}
\|\phi_{k}\|_{1+\sigma_2,\tau_{1}^{k}}=1,\quad\|g_{k}\|_{3+\sigma_2,\tau_{1}^{k}}\to 0,\quad\|h_{k}\|_{2+\sigma_2,\tau_{1}^{k}} \to 0.
\end{equation}
We show that for any compact subset $\Omega\subset\mathbb{R}_+^4$,
\begin{equation}\label{inner14}
\sup _{\tau_{0}<\tau<\tau_{1}^{k}}\tau^{\tilde\nu}\left|\phi_{k}(y, \tau)\right| \to 0 \quad \text{uniformly on } \Omega.
\end{equation}
Suppose not. Then there exist $M>0$, sequences $\{y_{k}\}\in \mathbb{R}^4_{+}$ with $|y_{k}|\leq M$, and $\{\tau_{2}^{k}\}\subset (\tau_0, \tau_1^k)$ such that
$$\left(\tau_{2}^{k}\right)^{\tilde\nu}\left(1+\left|y_{k}\right|^{1+\sigma_2}\right)\left|\phi_{k}\left(y_{k},\tau_{2}^{k}\right)\right| \geq \frac{1}{2}.$$
Clearly, by \eqref{inner13}, we have $\tau_2^k \to \infty$. Define scaled functions
\begin{align*}
&\bar{\phi}_{k}(y, \tau)=\left(\tau_{2}^{k}\right)^{\tilde\nu} \phi_{k}\left(y, \tau_{2}^{k}+\tau\right),\quad \bar{g}_{k}(y, \tau)=\left(\tau_{2}^{k}\right)^{\tilde\nu} g_{k}\left(y, \tau_{2}^{k}+\tau\right),
\\&\bar{h}_{k}(\tilde{y},0, \tau)=\left(\tau_{2}^{k}\right)^{\tilde\nu} h_{k}\left(\tilde{y},0,\tau_{2}^{k}+\tau\right).
\end{align*}
Then
\begin{align*}
\begin{cases}\partial_{\tau}\bar{\phi}_{k}=\Delta \bar{\phi}_{k}+\bar{g}_{k}(y, \tau) & \text { in }~ \mathbb{R}_{+}^4 \times\left(\tau_{0}-\tau_{2}^{k}, \infty\right), \\
-\frac{d \bar{\phi}_{k}}{d y_4}(\tilde{y}, 0, \tau)=2 U(\tilde{y}, 0) \bar{\phi}_{k}(\tilde{y}, 0, \tau)+\bar{h}_{k}(\tilde{y}, 0, \tau) & \text { in }~ \mathbb{R}^3 \times\left(\tau_{0}-\tau_{2}^{k}, \infty\right),
\end{cases}
\end{align*}
with $\bar{g}_{k}\to 0$, $\bar{h}_{k}\to 0$ uniformly on compact subsets of $\mathbb{R}_{+}^4 \times(-\infty, 0)$ and $\mathbb{R}^3 \times(-\infty, 0)$, respectively. Additionally,
$$|\bar{\phi}_{k}(y,\tau)|\leq\frac{1}{1+|y|^{1+\sigma_2}}\qquad\text{in}~ \mathbb{R}_{+}^4 \times\left(\tau_{0}-\tau_{2}^{k}, 0\right].$$
By parabolic regularity \cite{Quittner-Souplet2019}, up to a subsequence (still denoted as $\bar{\phi}_{k}$), $\bar{\phi}_{k}\to\bar{\phi}\not\equiv 0$ uniformly on compact subsets $\mathbb{R}_{+}^4\times(-\infty,0]$, and satisfy
\begin{equation}\label{inner4:61}
\begin{cases}
\partial_{\tau}\bar{\phi}=\Delta \bar{\phi} & \text { in }~ \mathbb{R}_{+}^4 \times(-\infty, 0] ,\\
-\frac{d \bar{\phi}}{d y_4}(\tilde{y}, 0, \tau)=2 U(\tilde{y}, 0) \bar{\phi}(\tilde{y}, 0, \tau) & \text { in }~ \mathbb{R}^3 \times(-\infty, 0], \\
\int_{\mathbb{R}_{+}^4} \bar{\phi}(y, \tau) Z_{i}(y) d y=0, &~ \forall \tau \in(-\infty,0 ],\ i=1,2,3,\\
|\bar{\phi}(y,\tau)|\leq\frac{1}{1+|y|^{1+\sigma_2}} &~ \text{in}~ \mathbb{R}_{+}^4 \times(-\infty, 0].
 \end{cases}
 \end{equation}
We now derive a contradiction by showing $\bar\phi\equiv0$.
Indeed, $\bar\phi(y,\tau)$ is smooth. A scaling argument gives the estimate
\begin{equation}\label{inner15}
(1+|y|)^{-1}|D_{y}\bar{\phi}|+|\bar{\phi}_{\tau}|+|D_{y}^{2}\bar{\phi}|\lesssim (1+|y|)^{-3-\sigma_2}\quad\text{in}\ \mathbb{R}_{+}^4 \times(-\infty, 0].
\end{equation}
Differentiating equation \eqref{inner4:61} with respect to $\tau$, we obtain
\begin{equation}\label{inner:60:0}
\begin{cases}
\partial_{\tau}\bar{\phi}_{\tau}=\Delta \bar{\phi}_{\tau} & \text { in }~ \mathbb{R}_{+}^4 \times(-\infty, 0] ,\\
-\frac{d \bar{\phi}_{\tau}}{d y_4}(\tilde{y}, 0, \tau)=2U(\tilde{y}, 0) \bar{\phi}_{\tau}(\tilde{y}, 0, \tau) & \text { in }~ \mathbb{R}^3 \times(-\infty, 0],
\end{cases}
\end{equation}
and then get the improved estimate
\begin{equation}\label{inner16}
(1+|y|)^{-1}|D_{y}\bar{\phi}_{\tau}|+|\bar{\phi}_{\tau\tau}|+|D_{y}^{2}\bar{\phi}_{\tau}|\lesssim (1+|y|)^{-5-\sigma_2}\quad\text{in}\ \mathbb{R}_{+}^4 \times(-\infty, 0].
\end{equation}
Testing the first equation in \eqref{inner:60:0} by $\bar{\phi}_{\tau}$ and integrating over $\mathbb{R}_{+}^4$, we find
$$\int_{\mathbb{R}_{+}^4} (\partial_{\tau} \bar{\phi}_{\tau} ) \bar{\phi}_{\tau}dy=\int_{\mathbb{R}_{+}^4} \tilde{\phi}_{\tau}  \Delta \tilde{\phi}_{\tau}dy=-B\left(\bar{\phi}_{\tau}, \bar{\phi}_{\tau}\right),$$
where
$$B(\bar{\phi}, \bar{\phi}) =\int_{\mathbb{R}_{+}^4}|\nabla \bar{\phi}|^{2} dy-\int_{\mathbb{R}^3} 2U(\tilde{y}, 0) \bar{\phi}^{2}(\tilde{y}, 0, \tau) d \tilde{y} .$$
From \eqref{inner15}, \eqref{inner16}, and the parabolic regularity theory,
$$|B(\bar{\phi}, \bar{\phi})|<+\infty,\qquad|B(\bar{\phi}_{\tau}, \bar{\phi}_{\tau})|<+\infty.$$
It then follows that
\begin{equation}\label{inner17}
\frac{1}{2}\partial_{\tau}\int_{\mathbb{R}_{+}^4}|\bar{\phi}_{\tau}|^{2}dy+B\left(\bar{\phi}_{\tau}, \bar{\phi}_{\tau}\right)=0.
\end{equation}
Recall that
$$\int_{\mathbb{R}^3}\bar{\phi}(\tilde{y},0,\tau)\widetilde{Z}_{0}(\tilde y)d\tilde{y}=0,\quad \forall \tau \in(-\infty,0],$$
implying
$$B(\bar{\phi}, \bar{\phi})\geq 0,\quad B(\bar{\phi}_{\tau}, \bar{\phi}_{\tau})\geq 0.$$
From \eqref{inner17}, we get
$$\frac{1}{2}\partial_{\tau}\int_{\mathbb{R}_{+}^4}|\bar{\phi}_{\tau}|^{2}dy\leq 0.$$
On the other hand, multiplying $\bar{\phi}_{\tau}=\Delta \bar{\phi}$ by $\bar{\phi}_{\tau}$ and integrating over $\mathbb{R}_{+}^4$ gives
$$\int_{\mathbb{R}_{+}^4}|\bar{\phi}_{\tau}|^{2}dy=-\frac{1}{2}\partial_\tau B(\bar{\phi}, \bar{\phi}).$$
Thus,
$$\int_{-\infty}^{0}\int_{\mathbb{R}_{+}^4}|\bar{\phi}_{\tau}|^{2}dyd\tau<+\infty,$$
which implies
$$\bar{\phi}_{\tau}=0.$$
Hence, $\bar{\phi}$ is independent of $\tau$ and solves
\begin{align*}
\begin{cases}
\Delta \bar{\phi}=0 & \text { in }~ \mathbb{R}_{+}^4 \times(-\infty, 0], \\
-\frac{d \bar{\phi}}{d y_4}(\tilde{y}, 0, \tau)=2U(\tilde{y}, 0) \bar{\phi}(\tilde{y}, 0, \tau) & \text { in }~\mathbb{R}^3 \times(-\infty, 0].
\end{cases}
\end{align*}
Since $\bar{\phi}$ is bounded, the nondegeneracy of the linearized operator $L_{U}$ \cite{Davila-del-Pino-Sire2013} implies $\bar{\phi}$ is a linear combination of $Z_{i}$, $i=1,\dots,4.$ Then by
$$\int_{\mathbb{R}^4_{+}}\bar{\phi}(y, \tau)Z_{i}(y)dy=0,\ i=1,\dots,4,$$
we conclude that $\bar{\phi}=0$, a contradiction. Therefore, \eqref{inner14} holds.

From \eqref{inner13}, there exist sequences $\{y_k\}$ with $|y_k| \to \infty$ and $\{\tau_2^k\} > 0$ satisfying
$$\left(\tau_{2}^{k}\right)^{\tilde\nu}\left|y_{k}\right|^{1+\sigma_2}\left|\phi_{k}\left(y_{k},\tau_{2}^{k}\right)\right| \geq \frac{1}{2}.$$
Define
$$\tilde{\phi}_{k}(z,\tau):=\left(\tau_{2}^{k}\right)^{\tilde\nu}|y_{k}|^{1+\sigma_2} \phi_{k}\left(y_{k}+|y_{k}|z,|y_{k}|^{2}\tau+\tau_{2}^{k}\right),$$
which solves
\begin{equation}\label{inner62}
\begin{cases}
\partial_{\tau}\tilde{\phi}_{k}=\Delta_{z} \tilde{\phi}_{k}+\tilde{g}_{k}(z,\tau) & \text{in} ~ \mathbb{R}^4_{+}\times\left(\frac{\tau_{0}-\tau_{2}^{k}}{|y_{k}|^{2}},\infty \right), \\
-\frac{d \tilde{\phi}_{k}}{d z_4}(\tilde{z},0,\tau)=a_{k}\tilde{\phi}_{k}(\tilde{z},0,\tau)+\tilde{h}_{k}(\tilde{z},\tau) & \text{in}~ \mathbb{R}^3\times\left(\frac{\tau_{0}-\tau_{2}^{k}}{|y_{k}|^{2}},\infty \right),\\
\tilde{\phi}_{k}(\cdot,\frac{\tau_{0}-\tau_{2}^{k}}{|y_{k}|^{2}})=0 & \text{in}~ \mathbb{R}^4_{+},
\end{cases}
\end{equation}
where
\begin{align*}
&\tilde{g}_{k}(z,\tau)=\left(\tau_{2}^{k}\right)^{\tilde\nu}|y_{k}|^{3+\sigma_2}g_{k}\left(y_{k}+|y_{k}|z,|y_{k}|^{2}\tau+\tau_{2}^{k}\right),\quad
a_{k}=|y_{k}|2U\left(\tilde{y}_{k}+|y_{k}|\tilde{z},0\right),
\\&\tilde{h}_{k}(\tilde{z},\tau)=\left(\tau_{2}^{k}\right)^{\tilde\nu}|y_{k}|^{2+\sigma_2}h_{k}\left(\tilde{y}_{k}+|y_{k}|\tilde{z},0,|y_{k}|^{2}\tau+\tau_{2}^{k}\right).
\end{align*}
By \eqref{inner13}, \eqref{inner14}, and the asymptotic behavior $\widetilde{Z}_{0}(\tilde{y})\sim |\tilde{y}|^{-4}$ as $|\tilde{y}|\to +\infty$, combined with the estimate
$$\tilde{\phi}_{k}(z,\tau)\leq\left(|y_{k}|^{-1}+|y_{k}|y_k|^{-1}+z|\right)^{-1-\sigma_2}\left[1+(\tau_{2}^{k})^{-1}|y_{k}|^{2}\tau\right]^{-\tilde\nu},$$
where $\hat{y}_{k}=\frac{\tilde{y}_{k}}{|y_{k}|}$, we obtain for
$(z,\tau)\in \mathbb{R}^4_{+}\times\left(\frac{\tau_{0}-\tau_{2}^{k}}{|y_{k}|^{2}},\frac{\tau_{1}^{k}-\tau_{2}^{k}}{|y_{k}|^{2}}\right)$
and
$(\tilde{z},\tau)\in \mathbb{R}^3\times\left(\frac{\tau_{0}-\tau_{2}^{k}}{|y_{k}|^{2}},\frac{\tau_{1}^{k}-\tau_{2}^{k}}{|y_{k}|^{2}}\right)$ that
\begin{align*}
|\tilde{g}_{k}(z,\tau)|
&\lesssim o(1)|y_{k}|^{3+\sigma_2}\left(1+|y_{k}+|y_{k}|z|\right)^{-3-\sigma_2}\left[1+(\tau_{2}^{k})^{-1}|y_{k}|^{2}\tau\right]^{-\tilde\nu}
\\&\lesssim o(1)\left(|y_{k}|^{-1}+|y_{k}|y_k|^{-1}+z|\right)^{-3-\sigma_2}\left[1+(\tau_{2}^{k})^{-1}|y_{k}|^{2}\tau\right]^{-\tilde\nu},
\\|a_{k}\tilde{\phi}_{k}(\tilde{z},0,\tau)|
&\lesssim |y_{k}|^{-1}\left(|y_{k}|^{-1}+|\hat{y}_{k}+\tilde{z}|\right)^{-3-\sigma_2}\left[1+(\tau_{2}^{k})^{-1}|y_{k}|^{2}\tau\right]^{-\tilde\nu},
\\|\tilde{h}_{k}(\tilde{z},\tau)|
&\lesssim o(1)\left(|y_{k}|^{-1}+|\hat{y}_{k}+\tilde{z}|\right)^{-2-\sigma_2}\left[1+(\tau_{2}^{k})^{-1}|y_{k}|^{2}\tau\right]^{-\tilde\nu}.
\end{align*}
Note that by construction
\begin{equation}\label{inner63}
|\tilde{\phi}_{k}(0,0)|=\left(\tau_{2}^{k}\right)^{\tilde\nu}|y_{k}|^{1+\sigma_2} \phi_{k}(y_{k}, \tau_{2}^{k})\geq\frac{1}{2}.
\end{equation}
Let $\tau_{3}^{k}=\frac{\tau_{0}-\tau_{2}^{k}}{|y_{k}|^{2}}$. The solution to \eqref{inner62} is given by
\begin{align*}
\tilde{\phi}_{k}(z,\tau)=&~\int_{\tau_{3}^{k}}^{\tau}\int_{\mathbb{R}^4_{+}}K_{\tau-r}(z,w)\tilde{g}_{k}(w,r)dwdr
\\&~+\int_{\tau_{3}^{k}}^{\tau}\int_{\mathbb{R}^3}K(\tilde{z}-\widetilde{w},z_4,\tau-r)\left[a_{k}\tilde{\phi}_{k}(\widetilde{w},0,r)+\tilde{h}_{k}(\widetilde{w},r)\right]d\widetilde{w}dr,
\end{align*}
where $K_t(x,y)$ and $K(\tilde{x},x_4,t)$ are defined in \eqref{outer20} and \eqref{outer3}, respectively. Thus
\begin{align*}
&|\tilde{\phi}_{k}(0,0)| \\&\lesssim\int_{\tau_{3}^{k}}^{0}\int_{\mathbb{R}^4_{+}}o(1)\frac{1}{(-r)^{2}}e^{\frac{|w|^{2}}{4r}}\left(|y_{k}|^{-1}+|y_{k}|y_k|^{-1}+w|\right)^{-3-\sigma_2}\left[1+(\tau_{2}^{k})^{-1}|y_{k}|^{2}r\right]^{-\tilde\nu}dwdr
\\&\quad+\int_{\tau_{3}^{k}}^{0}\int_{\mathbb{R}^3}\frac{|y_{k}|^{-1}}{(-r)^{2}}e^{\frac{|\widetilde{w}|^{2}}{4r}}\left(|y_{k}|^{-1}+|\hat{y}_{k}+\widetilde{w}|\right)^{-3-\sigma_2}\left[1+(\tau_{2}^{k})^{-1}|y_{k}|^{2}r\right]^{-\tilde\nu}d\widetilde{w}dr
\\&\quad+\int_{\tau_{3}^{k}}^{0}\int_{\mathbb{R}^3}o(1)\frac{1}{(-r)^{2}}e^{\frac{|\widetilde{w}|^{2}}{4r}}\left(|y_{k}|^{-1}+|\hat{y}_{k}+\widetilde{w}|\right)^{-2-\sigma_2}\left[1+(\tau_{2}^{k})^{-1}|y_{k}|^{2}r\right]^{-\tilde\nu}d\widetilde{w}dr.
\end{align*}
From standard hear kernel estimates,
suppose $\tau_2^k > 2\tau_0$, $|y_k| \geq 4$, $2<m<3$, then
\begin{equation}\label{inner64}
\begin{aligned}
&\int_{\tau_{3}^{k}}^{0}\int_{\mathbb{R}^4_{+}}\frac{1}{(-r)^{2}}e^{\frac{|w|^{2}}{4r}}\left(|y_{k}|^{-1}+|y_{k}|y_k|^{-1}+w|\right)^{-m}\left[1+(\tau_{2}^{k})^{-1}|y_{k}|^{2}r\right]^{-\tilde\nu}dwdr
\\&\lesssim\int_{\tau_{3}^{k}}^{0}\int_{\mathbb{R}^3}\frac{1}{(-r)^{\frac{3}{2}}}e^{\frac{|\widetilde{w}|^{2}}{4r}}\left(|y_{k}|^{-1}+|\hat{y}_{k}+\widetilde{w}|\right)^{-m}\left[1+(\tau_{2}^{k})^{-1}|y_{k}|^{2}r\right]^{-\tilde\nu}d\widetilde{w}dr
\\&\lesssim 1,
\end{aligned}
\end{equation}
and
\begin{equation}\label{inner65}
\int_{\tau_{3}^{k}}^{0}\int_{\mathbb{R}^3}\frac{1}{(-r)^{2}}e^{\frac{|\widetilde{w}|^{2}}{4r}}\left(|y_{k}|^{-1}+|\hat{y}_{k}+\widetilde{w}|\right)^{-m}\left[1+(\tau_{2}^{k})^{-1}|y_{k}|^{2}r\right]^{-\tilde\nu}d\widetilde{w}dr
\lesssim 1.
\end{equation}
From (\ref{inner64}) and (\ref{inner65}), for $0 < \sigma_2 < 1$ we have
$$|\tilde{\phi}_{k}(0,0)|\lesssim o(1)+|y_{k}|^{-1},$$
contradicting \eqref{inner63}. Then we complete the proof of Lemma \ref{lemma4}.
\end{proof}

\begin{proof}[Proof of Proposition \ref{proposition4.2}]
For modes $j\geq1$, let $\phi(y,\tau)$ solve the initial value problem \eqref{inner4}. By Lemma \ref{lemma4}, for any $\tau_1 > \tau_0$, there holds
$$|\phi(y,\tau)|\lesssim \tau^{-\tilde{\nu}}(1+|y|)^{-1-\sigma_2} \left(\|g^{\perp}\|_{3+\sigma_2,\tau_{1}}+\|h^{\perp}\|_{2+\sigma_2,\tau_{1}}\right),\  \forall\tau\in(\tau_{0},\tau_{1}),\ y\in\mathbb{R}_{+}^4.
$$
Since
$$\|g^{\perp}\|_{3+\sigma_2,\nu}<+\infty,\quad \|h^{\perp}\|_{2+\sigma_2,\nu}<+\infty$$
and
$$\|g^{\perp}\|_{3+\sigma_2,\tau_{1}}\leq\|g^{\perp}\|_{3+\sigma_2,\nu}, \quad\|h^{\perp}\|_{2+\sigma_2,\tau_{1}}\leq \|h^{\perp}\|_{2+\sigma_2,\nu}$$
hold for arbitrary $\tau_{1}$, it follows that
$$|\phi(y,\tau)|\lesssim \tau^{-\tilde{\nu}}(1+|y|)^{-1-\sigma_2} \left(\|g^{\perp}\|_{3+\sigma_2,\nu}+\|h^{\perp}\|_{2+\sigma_2,\nu}\right),\ \forall  \tau\in(\tau_{0},\tau_{1}),\ y\in\mathbb{R}_{+}^4.$$
By the arbitrariness of $\tau_{1}$, we have
\begin{equation}\label{inner50:4}
|\phi(y,\tau)|\lesssim \tau^{-\tilde{\nu}}(1+|y|)^{-1-\sigma_2} \left(\|g^{\perp}\|_{3+\sigma_2,\nu}+\|h^{\perp}\|_{2+\sigma_2,\nu}\right),\ \forall \tau\in(\tau_{0},+\infty),\ y\in\mathbb{R}_{+}^4.
\end{equation}
Thus, by parabolic regularity theory and a scaling argument, we establish that Proposition \ref{proposition4.2} holds for $\phi=\phi^{\perp}$.

Let $\phi[g,h]:=\phi^{0}[g^{0},h^{0}]+\phi^{\perp}[g^{\perp},h^{\perp}]$, where $\phi^{0}[g^{0},h^{0}]$ and $\phi^{\perp}[g^{\perp},h^{\perp}]$ are constructed in Lemmas \ref{lemma4.3} and \ref{lemma4}, respectively. Then $\phi[g,h]$ solves
\begin{equation}\label{inner6:2}
\begin{cases}
\phi_{\tau}=\Delta \phi+g(y, \tau)-c(\tau)Z_{0}(y) & \text { in }~ B_{2 R}^{+} \times\left(\tau_{0}, \infty\right), \\
\begin{aligned}
-\frac{d \phi}{d y_4}(\tilde{y}, 0, \tau)=&~2U(\tilde{y}, 0) \phi(\tilde{y}, 0, \tau)
+h(\tilde{y}, 0, \tau)
\end{aligned}  & \text { in }~ (\partial B_{2 R}^{+}  \cap\mathbb{R}^3) \times\left(\tau_{0}, \infty\right), \\
\phi\left(y, \tau_{0}\right)=0 & \text { in }~  B_{2 R(0)}^{+},
\end{cases}
\end{equation}
where $c(\tau)=\lambda_{0}c_{0}(\tau)$.  Estimates \eqref{inner33:4} and \eqref{inner50:4} imply that $\phi[g,h]$ satisfies \eqref{inner4:4}. This completes the proof of Proposition \ref{proposition4.2}.
\end{proof}

\medskip

\section{Solving the inner-outer gluing system}\label{6}
In this section, we construct a solution to the inner-outer gluing system \eqref{Jul-inner-problem} and \eqref{Jul-outer-problem} by formulating it as a fixed-point problem for the tuple $\vec{p}=(\phi,\psi,\mu,\xi)$ within a carefully chosen function space. The construction relies on the linear theories developed in Sections \ref{5} and \ref{linear-inner}, and is completed via the Schauder fixed-point theorem.

\subsection{Function spaces and parameter decomposition}\label{Sep26-space}

To set up the fixed-point argument, we first define the function spaces for the inner and outer corrections, as well as the modulation parameters. The choice of these spaces and their associated weighted norms is crucial, they are meticulously designed to:
\begin{enumerate}
\item capture the expected spatial decay and temporal blow-up rates of the solutions,
\item accommodate the singular behavior near the blow-up point,
\item be compatible with the linear theories established in Sections \ref{5} and \ref{linear-inner}, ensuring that the solution operators are bounded.
\end{enumerate}
We therefore introduce the following decompositions and Banach spaces.
The scaling and translation parameters are decomposed as
$$\mu(t)=\mu_{0}(t)+\mu_{1}(t),\quad \xi(t)=\xi_{0}(t)+\xi^{1}(t),\quad t\in[0,T),$$
where $\mu_{0}(t)$ and $\xi_{0}(t)$ are defined in \eqref{Aug1-1} and \eqref{solution2}, respectively, with $\mu_{1}(T)=0$ and $\xi^{1}(T)=0$. Furthermore, we assume that the time derivatives satisfy the comparative bound
$$
c_1|\dot{\mu}_0(t)|\leq |\dot{\mu}(t)|\leq c_2|\dot{\mu}_0(t)|\quad \text{for all}~t\in(0,T),
$$
for some constants $c_1, c_2>0$.

The inner correction $\phi$ and the outer correction $\psi$ are decomposed according to their spectral properties and roles in the boundary conditions:
$$
\phi=\phi^0+\phi^{\perp},\qquad \psi=\psi_1+\psi_2.$$
Here, $\phi^0$ corresponds to the mode associated with $Z_4$ (the scaling instability), while $\phi^\perp$ spans the higher modes related to translations ($Z_\ell$, $\ell=1,2,3$). The outer correction $\psi_1$ solves a homogeneous Neumann problem, and $\psi_2$ handles the inhomogeneous boundary data.

The solution will be sought in the following Banach spaces, defined by their respective weighted norms.
\begin{align*}
X_{\phi^{0}} &:= \left\{ \phi^{0} \in L^{\infty}(B_{2R}^{+} \times(0, T)) \mid\nabla_{y}\phi^{0} \in L^{\infty}(B_{2R}^{+} \times(0, T)),\ \|\phi^{0}\|_{\sigma_1,\nu} < +\infty \right\},  \\
X_{\phi^{\perp}} &:= \left\{ \phi^{\perp} \in L^{\infty}(B_{2R}^{+} \times(0, T)) \mid\nabla_{y}\phi^{\perp} \in L^{\infty}(B_{2R}^{+} \times(0, T)),\ \|\phi^{\perp}\|_{1+\sigma_2,\nu} < +\infty \right\}, \\
X_{\psi_1} &:= \left\{ \psi_1 \in L^{\infty}(\mathbb{R}_{+}^{4} \times(0, T)) \mid\|\psi_1\|_{*}^{(1)} < +\infty \right\}, \\
X_{\psi_2} &:= \left\{ \psi_2 \in L^{\infty}(\mathbb{R}_{+}^{4} \times(0, T)) \mid \|\psi_2\|_{*}^{(2)}< +\infty \right\}.
\end{align*}
The norms referenced here are defined in \eqref{Aug17-5}, \eqref{Aug5-3} and \eqref{outer5}.

The modulation parameter $\mu(t)$ will be determined by a non-local integral equation arising from the orthogonality condition $c^0 = 0$ (to be detailed in Subsection \ref{6.3.2} ). The analysis of this equation, following the methodology of \cite{Davila-del-Pino-Wei2020}, dictates a specific decomposition of the solution. We therefore define the parameter space accordingly.

We decompose $\mu(t)$ as
$$\mu(t)=\mu_{0,\kappa}(t)+\mu_1(t),$$
where the leading-order term is defined for some $\kappa\in\mathbb{R}$ by
\begin{align}\label{Sep27-1}
\mu_{0,\kappa} := \kappa |\log T| \int_{t}^{T} \frac{1}{|\log(T-s)|^{2}}  ds, \quad t \leq T,
\end{align}
and the perturbation $\mu_1$ is assumed to lie in the space
$$
X_{\mu}:= \left\{ \mu_1 \in C^{1}([-T,T]),0 < \iota < 1 \mid \mu_1(T) = 0,\ \|\mu_1\|_{*,3-\iota} < +\infty\right\}.
$$
Here, the norm for the perturbation is given by
$$
\|f\|_{*,k} := \sup_{t \in [-T,T]} |\log(T-t)|^{k} |\dot{f}(t)|.
$$
The full parameter $\mu$ is then equipped with the norm
\begin{align}\label{Aug18-7}
\|\mu\|_F = |\kappa| + \|\mu_1\|_{*,3-\iota}.
\end{align}
The motivation for this specific decomposition and the associated norms will become clear when we analyze the governing equation for $\mu(t)$ in Subsection \ref{6.3.2}.

Writing $\xi(t)=\hat{q}+\xi^{1}(t)$, we define the space for $\xi^1$ as
$$
X_{\xi}=\left\{\xi\in C^{1}((0,T);\mathbb{R}^{4}_+)\mid \dot{\xi}(T)=0,\ \|\xi\|_{G}< +\infty\right\}
$$
where the norm is given by
\begin{align}\label{Aug18-8}
\|\xi\|_{G}=\|\xi\|_{L^{\infty}(0,T)}+\sup_{t\in(0,T)}\mu_0^{-\nu_1}(t) |\dot{\xi}(t)|
\end{align}
for some fixed $\nu_1\in(0,\nu)$.

The full solution space is the product space
$$
\mathcal{X}=X_{\phi^{0}}\times  X_{\phi^{\perp}}\times X_{\psi_1}\times X_{\psi_2}\times\mathbb{R}\times X_{\mu}\times X_{\xi}.
$$
where the $\mathbb{R}$ corresponds to the parameter $\kappa$. We will seek a fixed point within a closed ball $\mathcal{B}\subset \mathcal{X}$ defined by the bounds in \eqref{Aug17-7}.

Next, we will estimate the nonlinear terms using the appropriate norms. This ensures that the linear theories can be applied and enables the definition of a compact fixed-point map.

\subsection{Estimates of the nonlinear terms}\label{Sep27-nonlinear}
This subsection is devoted to estimating the nonlinear terms $\mathcal{G}_1$, $\mathcal{G}_2$ and $\mathcal{H}_1$, $\mathcal{H}_2$ in the appropriate dual norms. These estimates are crucial for ensuring the boundedness of the fixed-point map to be defined later. We proceed by examining each term systematically, using the support properties of the cut-off function $\eta_R$ and the asymptotic behavior of the bubble $U(y)$.

\subsubsection{Estimates for the outer problem: $\mathcal{G}_1$ and $\mathcal{G}_2$}
We begin by analyzing the terms $\mathcal{G}_1[\phi,\psi,\mu,\xi]$ and $\mathcal{G}_2[\phi,\psi,\mu,\xi]$ in the norms $\|\cdot\|_{**}^{(1)}$ and $\|\cdot\|_{**}^{(2)}$, defined in \eqref{Aug5-2} and \eqref{outer6}, respectively.

Recall that
\begin{align*}
\mathcal{G}_1\left[\phi,\psi,\mu,\xi\right](y,t)=\mu^{-3}\left[\phi\Delta_y\eta_R+2\nabla_y\phi\nabla_y\eta_R-\mu^{2}\phi\partial_t\eta_R\right]+(1-\eta_R)\mathcal{K}[\mu,\xi],
\end{align*}
where $\mathcal{K}[\mu,\xi]$ is defined in \eqref{Jul26-8}.

We proceed term by term. Due to the cut-off function $\eta_R$, the term $$\mu^{-3}\left[\phi\Delta_y\eta_R+2\nabla_y\phi\nabla_y\eta_R-\mu^{2}\phi\partial_t\eta_R\right]
$$ is supported in the region
$$\left\{(x,t)\in \mathbb{R}_{+}^{4} \times(0, T)\mid \mu_0R\leq|x-\xi(t)|\leq 2\mu_0R \right\}.$$
It follows directly that
\begin{align}\label{Aug4-1}
&\mu^{-3}\left[\phi\Delta_y\eta_R+2\nabla_y\phi\nabla_y\eta_R-\mu^{2}\phi\partial_t\eta_R\right]\notag
\\&\lesssim \mu_0^{\nu-3}R^{-2-\alpha}\chi_{\{|x-\xi(t)|\sim\mu_0R\}}\left(R^{\alpha-\sigma_1}\|\phi^0\|_{\sigma_1,\nu}+ R^{\alpha-1-\sigma_2} \|\phi^\perp\|_{1+\sigma_2,\nu}\right) \notag
\\&\quad\times \left(1+\mu_0^{\nu_1+1}R\|\xi\|_{G}+\mu_0R^{2}\|\dot{\mu}\|_{\infty}+\mu_0^{2}R\dot{R}\right)\notag
\\&\lesssim \rho_1\left(R^{\alpha-\sigma_1}\|\phi^0\|_{\sigma_1,\nu}+ R^{\alpha-1-\sigma_2}\|\phi^\perp\|_{1+\sigma_2,\nu}\right),
\end{align}
where we used the relation $R(t)=\mu_0^{-\beta}(t)$ with $\beta\in(0,1/2)$ and the norm $\|\xi\|_{G}$ is defined in \eqref{Aug18-8}.

Next, we estimate the term $(1-\eta_R)\mathcal{K}[\mu,\xi]$. Recalling from \eqref{Jul26-8} that
\begin{equation*}
\mathcal{K}[\mu,\xi]= \frac{\mu^{-2}(t)\dot{\mu}(t)\alpha_{0}(2+2y_4)}{\left(|\tilde y|^2+(1+y_4)^2\right)^2}+ \mu^{-2}(t)\nabla U(y) \cdot \dot{\xi}(t) - \mathcal{R}[\mu].
\end{equation*}
Then
\begin{align}\label{Aug4-2}
&(1-\eta_R)\left[\frac{\mu^{-2}(t)\dot{\mu}(t)\alpha_{0}(2+2y_4)}{\left(|\tilde y|^2+(1+y_4)^2\right)^2}+ \mu^{-2}(t)\nabla U(y) \cdot \dot{\xi}(t)\right]\notag
\\&\lesssim R^{-1}\mu_0^{-\nu_2}\rho_2\|\dot{\mu}\|_{\infty}+ R^{-1}\mu_0^{\nu_1-\nu_2}\rho_2\|\xi\|_{G}.
\end{align}
Recall from \eqref{Jul26-7} that
\begin{align*}
\mathcal{R}[\mu]
&= \partial_{t}\Psi_{0} - \Delta\Psi_{0} - \mathcal{E}_{0}
\\&=\alpha_{0}\frac{y \cdot \dot{\xi} - \dot{\mu}(t)}{(1+|y|^{2})^{\frac{1}{2}}}
\int_{-T}^{t} \dot{\mu}(s)k_{\zeta}(\zeta,t-s)  ds
\\&\quad + \frac{\alpha_{0}}{\mu(t)(1+|y|^{2})^{3/2}}
\int_{-T}^{t} \dot{\mu}(s) \left[-\zeta k_{\zeta\zeta}(\zeta,t-s) + k_{\zeta}(\zeta,t-s)\right] ds.
\end{align*}
Following computations similar to those in \cite[Section 6.1]{delPino-Musso-Wei-Zhou2020}, we decompose the integral over $s$ into two intervals
\begin{align*}
\int_{-T}^{t} \dot{\mu}(s)k_{\zeta}(\zeta,t-s) ds=  \left(\int_{-T}^{t-\frac{\zeta^{2}}{4}} + \int_{t-\frac{\zeta^{2}}{4}}^{t}\right)\dot{\mu}(s)k_{\zeta}(\zeta,t-s) ds.
\end{align*}
We estimate the first integral as follows.
\begin{itemize}
\item For $T-t>\frac{\zeta^{2}}{4}$, we further split this interval into
\begin{align*}
\int_{-T}^{t-\frac{\zeta^{2}}{4}}\dot{\mu}(s)k_{\zeta}(\zeta,t-s)ds
= \left(\int_{-T}^{t-(T-t)} + \int_{t-(T-t)}^{t-\frac{\zeta^{2}}{4}}\right)
\dot{\mu}(s)k_{\zeta}(\zeta,t-s)ds.
\end{align*}
Using the expression
$$
k_{\zeta}(\zeta,t-s)=-\frac{2}{\zeta^{3}} + \frac{e^{-\frac{\zeta^{2}}{4(t-s)}}}{2\zeta(t-s)} + \frac{2e^{-\frac{\zeta^{2}}{4(t-s)}}}{\zeta^{3}},
$$
and noting that $\frac{\zeta^{2}}{4(t-s)}<1$ and $T-s<2(t-s)$ for $s\in (-T, t-(T-t))$, we deduce
\begin{align}\label{Aug3-1}
&\int_{-T}^{t-(T-t)}\dot{\mu}(s)\left[-\frac{2}{\zeta^{3}} + \frac{e^{-\frac{\zeta^{2}}{4(t-s)}}}{2\zeta(t-s)} + \frac{2e^{-\frac{\zeta^{2}}{4(t-s)}}}{\zeta^{3}}\right]ds \notag
\\& \lesssim \int_{-T}^{t-(T-t)}\frac{|\dot{\mu}(s)|\zeta}{(t-s)^2}ds  \lesssim \|\dot{\mu}\|_{\infty}\int_{-T}^{t-(T-t)}\frac{1}{(T-s)^{\frac{3}{2}}}ds \lesssim \|\dot{\mu}\|_{\infty}\frac{1}{\left(T-t\right)^{\frac{1}{2}}}.
\end{align}

Similarly, for the second integral
\begin{align}\label{Aug3-2}
\int_{t-(T-t)}^{t-\frac{\zeta^{2}}{4}}\dot{\mu}(s)\left[-\frac{2}{\zeta^{3}} + \frac{e^{-\frac{\zeta^{2}}{4(t-s)}}}{2\zeta(t-s)} + \frac{2e^{-\frac{\zeta^{2}}{4(t-s)}}}{\zeta^{3}}\right]ds
&\lesssim \|\dot{\mu}\|_{\infty}\int_{t-(T-t)}^{t-\frac{\zeta^{2}}{4}}\frac{1}{(t-s)^{\frac{3}{2}}}ds \notag
\\&\lesssim  \|\dot{\mu}\|_{\infty}\left|\frac{1}{\zeta}-\frac{1}{\left(T-t\right)^{\frac{1}{2}}}\right|
\end{align}
\end{itemize}

\begin{itemize}
\item For $T-t < \frac{\zeta^{2}}{4}$, since $s < t - \frac{\zeta^{2}}{4} < t - (T - t)$, a direct estimation gives
\begin{align}\label{Aug3-3}
\int_{-T}^{t-\frac{\zeta^{2}}{4}}\frac{|\dot{\mu}(s)|\zeta}{(t-s)^2}ds
\lesssim \int_{-T}^{t-(T-t)} \frac{|\dot{\mu}(s)|\zeta}{(t-s)^2}ds ds \lesssim \|\dot{\mu}\|_{\infty}\frac{1}{\left(T-t\right)^{\frac{1}{2}}}
\end{align}
\end{itemize}

Next, we estimate the integral near $s=t$,
\begin{align}\label{Aug3-4}
\int_{t-\frac{\zeta^{2}}{4}}^{t} \dot{\mu}(s)k_{\zeta}(\zeta,t-s) ds
\lesssim \frac{1}{\zeta^{3}} \int_{t-\frac{\zeta^{2}}{4}}^{t} |\dot{\mu}(s)| ds \lesssim \|\dot{\mu}\|_{\infty} \frac{1}{\zeta}.
\end{align}
Combining estimates \eqref{Aug3-1}-\eqref{Aug3-4} and recalling that $\zeta=\mu(t)\sqrt{1+|y|^2}$, we conclude that
\begin{align}\label{Aug3-5}
\left|\alpha_{0}\frac{y \cdot \dot{\xi} - \dot{\mu}(t)}{(1+|y|^{2})^{\frac{1}{2}}}\int_{-T}^{t} \dot{\mu}(s)k_{\zeta}(\zeta,t-s)  ds\right| \lesssim \|\dot{\mu}\|_{\infty} \left[ \frac{|y\cdot \dot{\xi} + \dot{\mu}|}{\left(T-t\right)^{\frac{1}{2}}(1+|y|^{2})^{\frac{1}{2}}} + \frac{|y\cdot \dot{\xi} + \dot{\mu}|}{\mu(1+|y|^2)} \right].
\end{align}

For the term involving $-\zeta k_{\zeta\zeta}(\zeta,t-s) + k_{\zeta}(\zeta,t-s)$, note the asymptotic behaviors:
$$\left|-\zeta k_{\zeta\zeta}(\zeta,t-s) + k_{\zeta}(\zeta,t-s)\right|\sim \frac{\zeta}{(t-s)^2}\quad \text{for}~~\frac{\zeta^{2}}{4(t-s)}<1$$
and
$$\left|-\zeta k_{\zeta\zeta}(\zeta,t-s) + k_{\zeta}(\zeta,t-s)\right|\sim \frac{1}{\zeta^3}~\quad \text{for}~~\frac{\zeta^{2}}{4(t-s)}>1.$$
By analogous estimates to those above, we obtain
\begin{align}\label{Aug3-6}
 &\frac{\alpha_{0}}{\mu(t)(1+|y|^{2})^{3/2}}
\int_{-T}^{t} \dot{\mu}(s) \left[-\zeta k_{\zeta\zeta}(\zeta,t-s) + k_{\zeta}(\zeta,t-s)\right] ds \notag
\\& \lesssim \|\dot{\mu}\|_{\infty} \left[ \frac{1}{\left(T-t\right)^{\frac{1}{2}}\mu(t)(1+|y|^2)^{\frac{3}{2}}} + \frac{1}{\mu^2 (1+|y|^2)^2} \right].
\end{align}
Combining all contributions, we summarize
\begin{align}\label{Aug1-3}
\mathcal{R}[\mu] \lesssim \|\dot{\mu}\|_{\infty}
&\left[  \frac{|y\cdot \dot{\xi} + \dot{\mu}|}{\left(T-t\right)^{\frac{1}{2}}(1+|y|^{2})^{\frac{1}{2}}} + \frac{|y\cdot \dot{\xi} + \dot{\mu}|}{\mu(1+|y|^2)}\right.\notag
\\&\left.\quad+\frac{1}{\left(T-t\right)^{\frac{1}{2}}\mu(t)(1+|y|^2)^{\frac{3}{2}}} + \frac{1}{\mu^2 (1+|y|^2)^2}\right].
\end{align}
Consequently, the contribution of $(1-\eta_R)\mathcal{R}[\mu]$ is bounded by
\begin{align}\label{Aug4-3}
  (1-\eta_R)\mathcal{R}[\mu]\lesssim  \|\dot{\mu}\|_{\infty}
  &\left[\mu_0^{\nu_1-\frac{1}{2}}\|\xi\|_{G}+\mu_0^{\frac{1}{2}-\nu_2}\rho_2+\mu_0^{\frac{1}{2}}\rho_3+\mu_0^{\nu_1-1}R^{-1}\|\xi\|_{G}\right.\notag
  \\&\left.\quad+\mu_0^{1-\nu_2}\rho_2+\mu_0^{\frac{1}{2}-\nu_2}R^{-1}\rho_2+ R^{-2}\mu_0^{-\nu_2}\rho_2\right].
\end{align}

Combining the estimates from \eqref{Aug4-1}, \eqref{Aug4-2}, and \eqref{Aug4-3}, we conclude that
\begin{align}\label{Aug4-5}
 \left\|\mathcal{G}_1\right\|_{**}^{(1)}\lesssim T^{\varepsilon_0}\left(\|\phi^0\|_{\sigma,\nu}+ \|\phi^\perp\|_{1+\sigma,\nu}+\|\dot{\mu}\|_{\infty}+\|\xi\|_{G}+1\right),
\end{align}
provided the following parameter conditions hold:
\begin{align}\label{Aug18-1}
\begin{cases}
     \sigma_1-\alpha>0,\\
    1+\sigma_2-\alpha>0,\\
    \beta-\nu_2>0,\\
    \nu_1+\beta-\nu_2>0,\\
    \nu_1-\frac{1}{2}>0,\\
    \frac{1}{2}-\nu_2>0,\\
     \nu_1-1+\beta>0,\\
     2\beta-\nu_2>0,
\end{cases}
\end{align}
where $\varepsilon_0$ is a small positive number.

Next, we proceed to estimate $\mathcal{G}_2\left[\phi,\psi,\mu,\xi\right]$. Recall that
\begin{align*}
\mathcal{G}_2\left[\phi,\psi,\mu,\xi\right](\tilde{y},0,t)=&~2\mu^{-1}(1-\eta_R)U(\tilde y,0)\left(\Psi_0+\psi+Z^*\right)(\tilde x,0,t)
\\&~+\mu^{-2}\phi(\tilde{y},0,t)\frac{d\eta_{R}}{d y_{4}}(\tilde{y},0,t)+f_{1}(\tilde{x},t),
\end{align*}
where
$$f_{1}(\tilde{x},t)=-\mu^{-2}(t)\eta_{R}(y)f(\tilde{y},t).$$

By direct computations analogous to those employed in estimating $\mathcal{R}[\mu]$, see also \cite[Section 6.1]{delPino-Musso-Wei-Zhou2020}, we establish the bound
\begin{align}\label{Aug1-2}
|\Psi_0|\lesssim |\dot{\mu}|\left[\log\left(\rho^2+\mu^2\right)+1\right].
\end{align}
Substituting this into the first term of $\mathcal{G}_2$, we obtain
\begin{align*}
&2\mu^{-1}(1-\eta_R)U(\tilde y,0)\left(\Psi_0+\psi+Z^*\right)(\tilde x,0,t)
\\&\lesssim \frac{\mu_0^{\nu_2}(t)}{|\tilde{x}-\tilde{\xi}(t)|}\chi_{\{|\tilde{x}-\tilde{\xi}(t)| \geq \mu_0R\}}R^{-1}(t)\mu_0^{-\nu_2}(t)
\Big[|\log(T-t)|\|\dot{\mu}\|_{\infty}+\|Z^{*}\|_{*}
\\&\qquad+\left(\mu_0^{\nu-1}(0) R^{-\alpha}(0)+T^{\nu_2}|\log T|^{1-\nu_2}+T\right)\left(\|\psi\|_{*}^{(1)}+\|\psi\|_{*}^{(2)}\right)\Big].
\end{align*}

Moreover,
\begin{align*}
\left|\mu^{-2}\phi(\tilde{y}, 0, t) \frac{d \eta_{R}}{d y_{4}}(\tilde{y}, 0, t)\right|=0.
\end{align*}
Recall that $f(\tilde{y},t)$ corresponds to $-c(\tau)\widetilde{Z}_{0}(\tilde{y})$ under a suitable change of variables $\tau$. Combining this with the fact that $\widetilde{Z}_{0}(\tilde{y})\sim |\tilde{y}|^{-4}$, we get
$$\left|f_{1}(\tilde{x},t)\right|=\left|-\mu^{-2}(t)\eta_{R}(y)f(\tilde{y},t)\right|\lesssim \mu_0^{\nu-2}(t) R^{-1-\alpha}(t) \chi_{\{|\tilde{x} -\tilde{\xi}(t)| \leq 2\mu_0R\}}R^{\alpha-3}(t).$$

Combining these estimates, we conclude the bound for $\mathcal{G}_{2}$ in the norm $\|\cdot\|_{**}^{(2)}$, defined in \eqref{outer6},
\begin{align}\label{Aug4-5-2}
\|\mathcal{G}_{2}\|_{**}^{(2)}\lesssim T^{\varepsilon_0}\left(\|\dot{\mu}\|_{\infty}+\|\psi\|_{\infty}+\|Z^{*}\|_{\infty}+1\right)
\end{align}
provided
\begin{align}\label{Aug18-2}
\begin{cases}
     \beta-\nu_2>0,\\
     \nu-1-\nu_2+\beta(1+\alpha)>0,\\
    \alpha-3<0.
\end{cases}
\end{align}

\subsubsection{Estimates for the inner problem: $\mathcal{H}_1$ and $\mathcal{H}_2$}
We now consider the inner problem \eqref{Jul-inner-problem}. The linear theory developed in Section \ref{linear-inner} guarantees that for $\mathcal{H}_1 = \mathcal{H}_1^{0} + \mathcal{H}_1^{\perp}$ and $\mathcal{H}_2 = \mathcal{H}_2^{0} + \mathcal{H}_2^{\perp}$ satisfying
\begin{align}\label{Aug12-3}
\|\mathcal{H}_1^{0}\|_{2+\sigma_1,\nu},\ \|\mathcal{H}_1^{\perp}\|_{3+\sigma_2,\nu},~\|\mathcal{H}_2^{0}\|_{1+\sigma_1,\nu},\ \|\mathcal{H}_2^{\perp}\|_{2+\sigma_2,\nu} < +\infty,
\end{align}
there exists a solution $(\phi^{0},\phi^{\perp},c^{0},c^{\ell})$ ($\ell=1,2,3$) to the projected inner problems
\begin{equation}\label{Aug12-1}
\begin{cases}
\mu^{2} \phi^{0}_{t}=\Delta_{y} \phi^{0}+\mathcal{H}^{0}_{1}[\phi,\psi,\mu,\xi] & \text { in } ~  B_{2R}^{+} \times(0, T),  \\
-\frac{d \phi^{0}}{d y_{4}}(\tilde{y}, 0, t)=2U(\tilde{y}, 0) \phi^{0}(\tilde{y}, 0, t)+\mathcal{H}^{0}_{2}[\phi,\psi,\mu,\xi]\\
 \qquad\qquad\qquad\quad+c^{0}Z_4(\tilde{y},0)
 & \text { in }~  (\partial B_{2R}^{+}\cap\mathbb{R}^{3})\times(0, T),\\
\phi^{0}(\cdot,0) = 0 & \ \text{in }~ B_{2R}^+,
\end{cases}
\end{equation}
and
\begin{equation}\label{Aug12-2}
\begin{cases}
\mu^{2} \phi^{\perp}_{t}=\Delta_{y} \phi^{\perp}+\mathcal{H}^{\perp}_{1}[\phi,\psi,\mu,\xi] & \text { in } ~  B_{2R}^{+} \times(0, T),  \\
-\frac{d \phi^{\perp}}{d y_{4}}(\tilde{y}, 0, t)=2U(\tilde{y}, 0) \phi^{\perp}(\tilde{y}, 0, t)+\mathcal{H}^{\perp}_{2}[\phi,\psi,\mu,\xi]\\
\qquad\qquad\qquad\quad + \sum\limits_{\ell=1}^{3}c^{\ell}Z_{\ell}(\tilde{y}, 0) & \text { in }~  (\partial B_{2R}^{+}\cap\mathbb{R}^{3})\times(0, T),\\
\phi^{\perp}(\cdot,0) = 0 & \ \text{in } ~B_{2R}^+.
\end{cases}
\end{equation}
Under these conditions, the inner solution $\phi[\mathcal{H}_1,\mathcal{H}_2] = \phi^{0}[\mathcal{H}_1^{0},\mathcal{H}_2^{0}] + \phi^{\perp}[\mathcal{H}_1^{\perp},\mathcal{H}_2^{\perp}]$ with appropriate space-time decay can then be constructed to carry out the inner-outer gluing procedure. Our first step is to ensure that the boundedness conditions \eqref{Aug12-3} hold by appropriately selecting constants.

Recall that
\begin{align*}
\mathcal{H}_1[\phi,\psi,\mu,\xi](y, t)&=\mu\left[\dot{\mu}\left(\nabla_y\phi\cdot y+\phi\right)+\nabla_y\phi\cdot \dot{\xi}\right]+\mu^3\mathcal{K}[\mu,\xi],
\\\mathcal{H}_2[\phi,\psi,\mu,\xi](\tilde y,0,t)&=2\mu U(\tilde y,0)\left(\Psi_0+\psi+Z^*\right)(\tilde x,0,t).
\end{align*}
Direct computation gives
\begin{align}\label{Aug2-1}
&\left|\mu\left[\dot{\mu}\left(\nabla_y\phi\cdot y+\phi\right)+\nabla_y\phi\cdot \dot{\xi}\right]\right|\notag
\\&\lesssim \mu_0 |\dot{\mu}_0|\left[\frac{\mu_0^{\nu} }{1+|y|^{\sigma_1}}\|\phi^0\|_{\sigma_1,\nu}+\frac{\mu_0^{\nu}}{1+|y|^{1+\sigma_2}}\|\phi^{\perp}\|_{1+\sigma_2,\nu}\right] \notag
\\&\quad+\mu_0|\dot{\xi}|\left[\frac{\mu_0^{\nu}}{1+|y|^{1+\sigma_1}}\|\phi^0\|_{\sigma_1,\nu}+\frac{\mu_0^{\nu} }{1+|y|^{2+\sigma_2}}\|\phi^{\perp}\|_{1+\sigma_2,\nu}\right].
\end{align}
Using the bound for $\mathcal{R}[\mu]$ from \eqref{Aug1-3}, we estimate the second term in $\mathcal{H}_1$
\begin{align}\label{Aug2-2}
\left|\mu^3\mathcal{K}[\mu,\xi]\right|
\lesssim &~\frac{\mu_0\left|\dot{\mu}_0\right|+\mu_0|\dot{\xi}|}{1+|y|^3}
+\|\dot{\mu}\|_{\infty}\left[ \frac{\mu_0^3|y\cdot \dot{\xi} + \dot{\mu}|}{\left(T-t\right)^{\frac{1}{2}}(1+|y|^{2})^{\frac{1}{2}}} + \frac{\mu_0^2|y\cdot \dot{\xi} + \dot{\mu}|}{1+|y|^2} \right] \notag
\\&~+\|\dot{\mu}\|_{\infty}\left[ \frac{\mu_0^2}{\left(T-t\right)^{\frac{1}{2}}(1+|y|^2)^{\frac{3}{2}}} + \frac{\mu_0}{ (1+|y|^2)^2}\right].
\end{align}
Combining \eqref{Aug2-1}, \eqref{Aug2-2} and applying the relevant norms, we obtain the bound for $\mathcal{H}_1^0$
\begin{align}\label{Aug17-8}
&\|\mathcal{H}_1^0\|_{2+\sigma_1,\nu}\notag
\\&\lesssim\mu_0\|\dot{\mu}\|_{\infty}\left(R^{2}\|\phi^0\|_{\sigma_1,\nu}+R^{1+\sigma_1-\sigma_2}\|\phi^{\perp}\|_{1+\sigma_2,\nu}\right)    \notag
\\&\quad+ \mu_0^{1+\nu_1}\|\xi\|_{G} \left(R\|\phi^0\|_{\sigma_1,\nu}+R^{\sigma_1-\sigma_2}\|\phi^{\perp}\|_{1+\sigma_2,\nu}\right)
+\left(\mu_0^{1-\nu}\|\dot{\mu}\|_{\infty}+\mu_0^{1+\nu_1-\nu} \|\xi\|_{G}\right)R^{-1+\sigma_1}          \notag
\\&\quad+\|\dot{\mu}\|_{\infty}\left[ \frac{\mu_0^{3-\nu}R^{1+\sigma_1}|R\cdot \dot{\xi} + \dot{\mu}|+\mu_0^{2-\nu}R^{-1+\sigma_1}}{\left(T-t\right)^{\frac{1}{2}}} + \mu_0^{2-\nu}R^{\sigma_1}|R\cdot \dot{\xi} + \dot{\mu}|+\mu_0^{1-\nu}R^{-2+\sigma_1} \right]\notag
\\&\lesssim T^{\varepsilon_0}\left(\|\phi^0\|_{\sigma_1,\nu}+\|\phi^{\perp}\|_{1+\sigma_2,\nu}+\|\dot{\mu}\|_{\infty}+\|\xi\|_{G}+1\right),
\end{align}
provided that the following parameter conditions are satisfied
\begin{align}\label{Aug18-3}
\begin{cases}
    1-2\beta>0,\\
    1-\beta(1+\sigma_1-\sigma_2)>0,\\
    1-\nu+\beta(1-\sigma_1)>0,\\
    3+\nu_1-\nu-\beta(2+\sigma_1)-\frac{1}{2}>0,\\
    3-\nu-\beta(1+\sigma_1)-\frac{1}{2}>0,\\
    2-\nu+\beta(1-\sigma_1)-\frac{1}{2}>0,\\
    2+\nu_1-\nu-\beta(1+\sigma_1)>0,\\
    2-\nu-\beta\sigma_1>0,\\
    1-\nu+\beta(2-\sigma_1)>0.\\
\end{cases}
\end{align}

Similarly, for the $\mathcal{H}_1^{\perp}$ component
\begin{align}\label{Aug17-9}
&\|\mathcal{H}_1^{\perp}\|_{3+\sigma_2,\nu}\notag
\\&\lesssim\mu_0\|\dot{\mu}\|_{\infty}\left(R^{3+\sigma_2-\sigma_1}\|\phi^0\|_{\sigma_1,\nu}+R^2\|\phi^{\perp}\|_{1+\sigma_2,\nu}\right) \notag
\\&\quad + \mu_0^{1+\nu_1}\|\xi\|_{G} \left(R^{2+\sigma_2-\sigma_1}\|\phi^0\|_{\sigma_1,\nu}+R\|\phi^{\perp}\|_{1+\sigma_2,\nu}\right) +\left(\mu_0^{1-\nu}\|\dot{\mu}\|_{\infty}+\mu_0^{1+\nu_1-\nu}\|\xi\|_{G}\right)R^{\sigma_2} \notag
\\&\quad+\|\dot{\mu}\|_{\infty}\left[ \frac{\mu_0^{3-\nu}R^{2+\sigma_2}|R\cdot \dot{\xi} + \dot{\mu}|+\mu_0^{2-\nu}R^{\sigma_2}}{\left(T-t\right)^{\frac{1}{2}}} + \mu_0^{2-\nu}R^{1+\sigma_2}|R\cdot \dot{\xi} + \dot{\mu}|+\mu_0^{1-\nu}R^{-1+\sigma_2} \right] \notag
\\&\lesssim T^{\varepsilon_0}\left(\|\phi^0\|_{\sigma_1,\nu}+\|\phi^{\perp}\|_{1+\sigma_2,\nu}+\|\dot{\mu}\|_{\infty}+\|\xi\|_{G}+1\right),
\end{align}
provided
\begin{align}\label{Aug18-4}
\begin{cases}
    1-\beta(3+\sigma_2-\sigma_1)>0,\\
    1-2\beta>0,\\
    1-\nu-\beta\sigma_2>0,\\
    3+\nu_1-\nu-\beta(3+\sigma_2)-\frac{1}{2}>0,\\
    3-\nu-\beta(2+\sigma_2)-\frac{1}{2}>0,\\
    2-\nu-\beta\sigma_2-\frac{1}{2}>0,\\
    2+\nu_1-\nu-\beta(2+\sigma_2)>0,\\
    2-\nu-\beta(1+\sigma_2)>0,\\
    1-\nu+\beta(1-\sigma_2)>0.\\
\end{cases}
\end{align}

For $\mathcal{H}_2$, using estimate \eqref{Aug1-2} for $\Psi_0$, we get
\begin{align}\label{Aug2-3}
&\left|2\mu U(\tilde y,0)\left(\Psi_0+\psi+Z^*\right)(\tilde x,0,t)\right|\notag
\\&\lesssim \frac{\mu_0(t)}{1+|\tilde{y}|^2}\Big[|\dot{\mu}_0|\left(\left|\log \mu_0\right|+\log(1+|\tilde{y}|\right)   \notag
\\&\qquad\qquad\quad +\left(\mu_0^{\nu-1}(0) R^{-\alpha}(0)+T^{\nu_2}|\log T|^{1-\nu_2}+T\right)\left(\|\psi\|_{*}^{(1)}+\|\psi\|_{*}^{(2)}\right)+\|Z^{*}\|_*\Big],
\end{align}
where
\begin{align}\label{Sep22-1}
\|Z^{*}\|_*=\|Z^*\|_{L^\infty(\mathbb{R}_+^4\times(0,T)}+\|\nabla Z^*\|_{L^\infty(\mathbb{R}_+^4\times(0,T)}.
\end{align}
Additionally, from the definition of $f(\tilde y,t)$, there holds
$$\left|f(\tilde y,t)\right|\lesssim \frac{\mu_0^{\nu}(t)}{1+|\tilde{y}|^4}.$$

Consequently, we derive the estimate for $\mathcal{H}_2^0$
\begin{align}\label{Aug17-10}
\|\mathcal{H}_2^0\|_{1+\sigma_1,\nu}
&\lesssim \mu_0^{1-\nu}\|\dot{\mu}\|_{\infty}R^{-1+\sigma_1}|\log(T-t)|   \notag
\\&\quad +\mu_0^{1-\nu}R^{-1+\sigma_1}\left(\mu_0^{\nu-1}(0) R^{-\alpha}(0)+T^{\nu_2}|\log T|^{1-\nu_2}+T\right)\left(\|\psi\|_{*}^{(1)}+\|\psi\|_{*}^{(2)}\right)    \notag
\\&\quad +\mu_0^{1-\nu}R^{-1+\sigma_1}\|Z^{*}\|_*+R^{-3+\sigma_1}    \notag
\\&\lesssim T^{\varepsilon_0}\left(\|\dot{\mu}\|_{\infty}+\|\psi\|_{*}^{(1)}+\|\psi\|_{*}^{(2)}+\|Z^{*}\|_*+1\right),
\end{align}
provided the conditions
\begin{align}\label{Aug18-5}
\begin{cases}
1-\nu+\beta(1-\sigma_1)>0, \\
\beta(3-\sigma_1)>0.
\end{cases}
\end{align}

Similarly, for $\mathcal{H}_2^{\perp}$
\begin{align}\label{Aug17-11}
\|\mathcal{H}_2^{\perp}\|_{2+\sigma_2,\nu}
&\lesssim \mu_0^{1-\nu}|\dot{\mu}_0|R^{\sigma_2}|\log(T-t)|   \notag
\\&\quad +\mu_0^{1-\nu}R^{\sigma_2} \left(\mu_0^{\nu-1}(0) R^{-\alpha}(0)+T^{\nu_2}|\log T|^{1-\nu_2}+T\right)\left(\|\psi\|_{*}^{(1)}+\|\psi\|_{*}^{(2)}\right)   \notag
\\&\quad+\mu_0^{1-\nu}R^{\sigma_2}\|Z^{*}\|_*+R^{-2+\sigma_2}    \notag
\\&\lesssim T^{\varepsilon_0}\left(\|\dot{\mu}\|_{\infty}+\|\psi\|_{*}^{(1)}+\|\psi\|_{*}^{(2)}+\|Z^{*}\|_*+1\right),
\end{align}
provided
\begin{align}\label{Aug18-6}
\begin{cases}
1-\nu-\beta\sigma_2>0, \\
\beta(2-\sigma_2)>0.
\end{cases}
\end{align}

\subsection{The parameter problems}\label{Sep26-parameters}

The parameters $\mu(t)$ and $\xi(t)$ are not prescribed in advance but are determined as part of the solution. Their evolution is governed by the conditions $c^0 = 0$ and $c^\ell = 0$. Specifically, to complete the system defined by \eqref{Aug12-1} and \eqref{Aug12-2}, we must adjust the parameters $\mu(t)$ and $\xi(t)$ such that
\begin{align}\label{Oct10-1}
c^0[\mu,\xi,\psi,Z^*]=0,\quad c^{\ell}[\mu,\xi,\psi,Z^*]=0,\quad \ell=1,2,3,\ \forall ~t\in(0,T),
\end{align}
where the coefficients are given by
\begin{align}\label{Aug13-1}
  c^0[\mu,\xi,\psi,Z^*]= \frac{\int_{B_{2R}^{+}}\mathcal{H}_1^0(y,t)Z_{4}(y) dy+\int_{\partial{B}_{2R}^{+}\cap{\mathbb{R}}^{3}}\mathcal{H}_2^0(\tilde{y},t)\widetilde{Z}_{4}(\tilde{y}) d\tilde{y}}{\int_{\partial{B}_{2R}^{+}\cap{\mathbb{R}}^{3}} |\widetilde{Z}_{4}|^{2}d\tilde{y}},
\end{align}
and
\begin{align}\label{Aug13-2}
  c^{\ell}[\mu,\xi,\psi,Z^*]= \frac{\int_{B_{2R}^{+}}\mathcal{H}_1^{\perp}(y,t)Z_{\ell}(y) dy+\int_{\partial{B}_{2R}^{+}\cap{\mathbb{R}}^{3}}\mathcal{H}_2^{\perp}(\tilde{y},t)\widetilde{Z}_{\ell}(\tilde{y}) d\tilde{y}}{\int_{\partial{B}_{2R}^{+}\cap{\mathbb{R}}^{3}} |\widetilde{Z}_{\ell}|^{2}d\tilde{y}} \quad \text{for}~\ell=1,2,3.
\end{align}

The equation for $\xi(t)$ controls the location of the blow-up point by counteracting translation instabilities, while the non-local integro-differential equation for $\mu(t)$ determines the blow-up rate by balancing the dominant scaling instability. The subsequent analysis translates conditions \eqref{Oct10-1} into solvable equations for $\dot{\xi}$ and $\dot{\mu}$.

\subsubsection{The reduced equation for $\xi(t)$}

We begin by considering the reduced equation for the parameter $\xi(t)$. Observing that condition \eqref{Aug13-2} is equivalent to
$$\int_{B_{2R}^{+}}\mathcal{H}_1^{\perp}(y,t)Z_{\ell}(y) dy+\int_{\partial{B}_{2R}^{+}\cap{\mathbb{R}}^{3}}\mathcal{H}_2^{\perp}(\tilde{y},t)\widetilde{Z}_{\ell}(\tilde{y}) d\tilde{y}=0,\quad \ell=1,2,3,$$
which implies
\begin{equation}\label{Aug15-1}
\dot{\xi}_{\ell} = b_{\ell}[\mu,\xi,\phi,\psi,Z^{*}],
\end{equation}
where
\begin{align*}
b_{\ell}[\mu,\xi,\phi,\psi,Z^{*}] = \frac{\int_{B_{2R}^{+}}\left[\mathcal{H}_1^{\perp}(y,t)-\mu\partial_{y_l}U(y)\dot{\xi}_l\right]Z_{\ell}(y) dy+\int_{\partial{B}_{2R}^{+}\cap{\mathbb{R}}^{3}}\mathcal{H}_2^{\perp}(\tilde{y},t)\widetilde{Z}_{\ell}(\tilde{y}) d\tilde{y}}{\mu(t)\int_{B_{2R}^{+}}\left|Z_{\ell}(y)\right|^2 dy}.
\end{align*}

A direct estimate yields
\begin{align*}
&\left| \int_{B_{2R}^{+}}\mu\left[\dot{\mu}\left(\nabla_y\phi\cdot y+\phi\right)+\nabla_y\phi\cdot \dot{\xi}\right]Z_{\ell}(y) dy\right|
\\&\lesssim \mu_0^{1+\nu}|\dot{\mu}_0|\left(R^{1-\sigma_1}\|\phi^0\|_{\sigma_1,\nu}+\|\phi^{\perp}\|_{1+\sigma_2,\nu}\right)+\mu_0^{1+\nu}|\dot{\xi}|\left(\|\phi^0\|_{\sigma_1,\nu}+\|\phi^{\perp}\|_{1+\sigma_2,\nu}\right).
\end{align*}
Symmetry considerations play a crucial role. From \eqref{Jul26-7} and \eqref{Jul26-8}, the term $\mu^3 \mathcal{K}[\mu,\xi] - \mu \partial_{y_\ell} U(y) \dot{\xi}_\ell$ is even in $y_\ell$. Since $Z_\ell(y)$ is odd in $y_\ell$, it follows that
\begin{align*}
\int_{B_{2R}^{+}}\left[\mu^3 \mathcal{K}[\mu,\xi] - \mu \partial_{y_l} U(y) \dot{\xi}_l\right]Z_{\ell}(y) dy=0.
\end{align*}
Similarly, by \eqref{Jul26-5}, $\Psi_0$ is even in $y_\ell$. Furthermore, recalling that $f(\tilde{y},t)$ corresponds to $-c(\tau)\widetilde{Z}_{0}(\tilde{y})$ under a suitable change of variables $\tau$, and noting that $\widetilde{Z}_{0}(\tilde{y})$ is even in $\tilde{y}$, we obtain
\begin{align*}
 \int_{\partial{B}_{2R}^{+}\cap{\mathbb{R}}^{3}}\left[2\mu U(\tilde y,0)\Psi_0(\tilde x,0,t)+f(\tilde y,t)\right]\widetilde{Z}_{\ell}(\tilde{y}) d\tilde{y}=0.
\end{align*}
For the remaining boundary term involving $\psi$ and $Z^*$, we expand
\begin{align*}
 &\mu \int_{\partial{B}_{2R}^{+}\cap{\mathbb{R}}^{3}}U(\tilde y,0)\left(\psi+Z^*\right)(\tilde x,0,t)\widetilde{Z}_{\ell}(\tilde{y}) d\tilde{y}
 \\&=\mu \int_{\partial{B}_{2R}^{+}\cap{\mathbb{R}}^{3}}U(\tilde y,0)\left[\psi(\tilde \xi,0,t)+Z^*(\tilde \xi,0,t)+\mu y_l\left(\frac{\pp \psi}{\pp x_l}+\frac{\pp Z^*}{\pp x_l}\right)(\bar \xi,0,t)\right]\widetilde{Z}_{\ell}(\tilde{y}) d\tilde{y},
\end{align*}
where $\bar\xi=\tilde{\xi}+\theta(\tilde{x}-\tilde{\xi})$ for some $\theta \in [0,1]$. Using the evenness of $U(\tilde y,0)$ in $\tilde{y}$ and the oddness of $\widetilde{Z}_{\ell}(\tilde{y})$ in $\tilde{y}$ for $\ell=1,2,3$, we estimate
\begin{align*}
 &\left|2\mu \int_{\partial{B}_{2R}^{+}\cap{\mathbb{R}}^{3}}U(\tilde y,0)\left(\psi+Z^*\right)(\tilde x,0,t)\widetilde{Z}_{\ell}(\tilde{y}) d\tilde{y}\right|
 \\&\lesssim \mu_0^{2}(t)\left[\left(\mu_0^{\nu-1}(0) R^{-\alpha}(0)+T^{\nu_2}|\log T|^{1-\nu_2}+T\right)\left(\|\psi\|_{*}^{(1)}+\|\psi\|_{*}^{(2)}\right)+\|Z^{*}\|_*\right].
\end{align*}

Thus, the size of $b_{\ell}[\mu,\xi,\phi,\psi,Z^{*}]$ is controlled by
\begin{align}\label{Aug15-2-2}
&|b_{\ell}[\mu,\xi,\phi,\psi,Z^{*}]|     \notag
\\&\lesssim \mu_0^{\nu}|\dot{\mu}_0|\left(R^{1-\sigma_1}\|\phi^0\|_{\sigma_1,\nu}+\|\phi^{\perp}\|_{1+\sigma_2,\nu}\right)+\mu_0^{\nu}|\dot{\xi}|\left(\|\phi^0\|_{\sigma_1,\nu}+\|\phi^{\perp}\|_{1+\sigma_2,\nu}\right)  \notag
\\&\quad+\mu_0(t)\left[\left(\mu_0^{\nu-1}(0) R^{-\alpha}(0)+T^{\nu_2}|\log T|^{1-\nu_2}+T\right)\left(\|\psi\|_{*}^{(1)}+\|\psi\|_{*}^{(2)}\right)+\|Z^{*}\|_*\right].
\end{align}

We now analyze the reduced problem \eqref{Aug15-1}, which defines solution operators $\Xi_{\ell}$ ($\ell=1,2,3$) for the components $\xi_{\ell}$. Writing
\begin{equation}\label{Aug15-2}
\Xi=(\Xi_{1},\Xi_{2},\Xi_{3}) ,
\end{equation}
we express $\xi(t) = \hat{q} + \xi^{1}(t)$, where $\hat{q} = (q_1, q_2, q_3, 0)$ is a prescribed point in $\mathbb{R}_+^4$. We seek solutions for $\xi^{1}(t)$ within the function space equipped with the norm
$$
\|\xi\|_{G}=\|\xi\|_{L^{\infty}(0,T)}+\sup_{t\in(0,T)}\mu_0^{-\nu_1}(t)|\dot{\xi}(t)|
$$
for some fixed $0<\nu_1<\nu<1$. Integrating \eqref{Aug15-1} gives the bound
$$
|\xi_{\ell}(t)|\leq|q_{\ell}|+\|b_{\ell}[\lambda,\xi,\phi,\psi,Z^{*}]\|_{L^{\infty}(0,T)}\,(T-t).
$$
Consequently, we have
\begin{equation}\label{Aug15-3}
\|\Xi_{\ell}\|_{G}\leq|q_{\ell}|+(T-t)^{-\nu_1}\|b_{\ell}[\mu,\xi,\phi,\psi,Z^{*}]\|_{L^{\infty}(0,T)}.
\end{equation}

Combining \eqref{Aug15-2-2} and \eqref{Aug15-3}, we conclude that for some constant $C>0$,
\begin{align}\label{Aug17-12}
\|\Xi_{\ell}\|_{G}\leq&~|q_{\ell}|+C(T-t)^{-\nu_1}\mu_0^{\nu}|\dot{\mu}_0|\left(R^{1-\sigma_1}\|\phi^0\|_{\sigma_1,\nu}+\|\phi^{\perp}\|_{1+\sigma_2,\nu}\right)  \notag
\\&+C(T-t)^{-\nu_1}\left[\mu_0^{\nu}|\dot{\xi}|\left(\|\phi^0\|_{\sigma_1,\nu}+\|\phi^{\perp}\|_{1+\sigma_2,\nu}\right)+\mu_0\|Z^{*}\|_*\right]  \notag
\\&+C(T-t)^{-\nu_1}\mu_0\left(\mu_0^{\nu-1}(0) R^{-\alpha}(0)+T^{\nu_2}|\log T|^{1-\nu_2}+T\right)\left(\|\psi\|_{*}^{(1)}+\|\psi\|_{*}^{(2)}\right).
\end{align}

\subsubsection{The integro-differential equation for $\mu(t)$}\label{6.3.2}

The reduced problem for the scaling parameter $\mu(t)$ is fundamentally analogous to that treated in \cite{Davila-del-Pino-Wei2020}. Therefore, we adopt the strategy and logical framework developed therein.

Direct computations show that condition \eqref{Aug13-1} yields a non-local integro-differential equation
\begin{equation}\label{Aug15-4}
\int_{-T}^{t}\frac{\dot{\mu}(s)}{t-s}\Gamma\left(\frac{\mu^{2}(t)}{t-s}\right)ds + \mathbf{c}_{0}\dot{\mu} = a[\mu,\xi,\psi,Z^{*}](t) + \mathbf{a}_{r}[\mu,\xi,\phi,\psi,Z^{*}](t),
\end{equation}
where $\mathbf{c}_{0} = \alpha_0\int_{\mathbb{R}^4_+}\frac{2+2y_4}{\left(1+|y|^2\right)^2} Z_{4}(y) d y$,
\begin{equation}\label{Aug15-5}
a[\mu,\xi,\psi,Z^{*}] = -\int_{\partial{B}_{2R}^{+}\cap{\mathbb{R}}^{3}} 2 U(\tilde y,0)\left(\psi+Z^*\right)(\tilde x,0,t)Z_4(\tilde{y},0) d\tilde y,
\end{equation}
and the remainder term $\mathbf{a}_{r}[\mu,\xi,\phi,\psi,Z^{*}](t)$ is of lower order and satisfies the bound
\begin{align*}
|\mathbf{a}_{r}[\mu,\xi,\phi,\psi,Z^{*}](t)|
&=\Bigg|-2\alpha_{0}\mu(t)\dot{\mu}(t)\int_{B_{2R}^+}\frac{Z_{4}(y)}{1+|y|^{2}}
\left(\int_{-T}^{t}\frac{\dot{\mu}(s)}{t-s}\Upsilon K_{\Upsilon}(\Upsilon)ds\right)dy
\\&\qquad-\int_{B_{2R}^+}\left[\dot{\mu}\left(\nabla_y\phi\cdot y+\phi\right)+\nabla_y\phi\cdot \dot{\xi}\right]Z_{4}(y)dy+\cdots\Bigg|
\\&\lesssim \mu_0|\dot{\mu}_0|+ \mu_0^{\nu}|\dot{\mu}_0|\left(R^{2-\sigma_1}\|\phi^0\|_{\sigma_1,\nu}+R^{1-\sigma_2}\|\phi^{\perp}\|_{1+\sigma_2,\nu}\right)
\\&\quad+\mu_0^{\nu}|\dot{\xi}|\left(R^{1-\sigma_1}\|\phi^0\|_{\sigma_1,\nu}+\|\phi^{\perp}\|_{1+\sigma_2,\nu}\right)
+|\dot{\mu}_0|R^{-1}.
\end{align*}

To solve $\mu(t)$, we introduce the following functional norms.
\begin{itemize}
    \item $\|\cdot\|_{\Theta,l}$-norm: for functions $f \in C([-T,T];\mathbb{R})$ with $f(T) = 0$,
    $$
    \|f\|_{\Theta,l} := \sup_{t \in [0,T]} \frac{|\log(T-t)|^{l}}{(T-t)^{\Theta}} |f(t)|,
   $$
    where $\Theta \in (0,1)$ and $l \in \mathbb{R}$.

    \item $[\cdot]_{\gamma,m,l}$-seminorm: for $g \in C([-T,T];\mathbb{R})$ with $g(T) = 0$,
    $$
    [g]_{\gamma,m,l} := \sup_{I_{T}} \frac{|\log(T-t)|^{l}}{(T-t)^{m}(t-s)^{\gamma}} |g(t) - g(s)|,
   $$
    where $I_{T} = \left\{ -T \leq s \leq t \leq T \mid t - s \leq \frac{1}{10}(T - t) \right\}$, $0 < \gamma < 1$, $m > 0$ and $l \in \mathbb{R}$.
\end{itemize}

We define the principal operator as
\begin{align}\label{Aug17-3}
\mathcal{B}_0[\mu](t) := \int_{-T}^{t} \frac{\dot{\mu}(s)}{t-s} \Gamma \left( \frac{\mu^{2}(t)}{t-s} \right) ds + \mathbf{c}_{0} \dot{\mu},
\end{align}
and express the orthogonality coefficient as
\begin{align*}
c^{0}[\mathcal{H}^0_1,\mathcal{H}^0_2] = \frac{\mathcal{B}_0[\mu] - \left( a[\mu,\xi,\psi, Z^{*}] + \mathbf{a}_{r}[\mu,\xi,\phi,\psi, Z^{*}] \right)}{\int_{\partial{B}_{2R}^{+}\cap{\mathbb{R}}^{3}} \left|\widetilde{Z}_{4}\right|^{2}d\tilde{y}}.
\end{align*}

The solvability of $\mu(t)$ relies on the following key proposition from \cite{Davila-del-Pino-Wei2020}.
\begin{prop}\label{prop7.1}
Let $\omega,\Theta \in (0,\frac{1}{2})$, $\gamma \in (0,1)$, $m \leq \Theta - \gamma$ and $l \in \mathbb{R}$. Suppose $a(t)$ satisfies $a(T) < 0$ with $1/C \leq |a(T)| \leq C$ for some constant $C > 1$, and
\begin{align}\label{Aug17-1}
T^{\Theta}|\log T|^{1+c-l}\|a(\cdot)-a(T)\|_{\Theta,l-1} + [a]_{\gamma,m,l-1} \leq C_{1} \end{align}
for some $c > 0$. Then there exist operators $\mathcal{P}$ and $\mathcal{R}_0$ such that $\mu = \mathcal{P}[a] : [-T,T] \to \mathbb{R}$ satisfies
\begin{align}\label{Aug17-2}
\mathcal{B}_0[\mu](t) = a(t) + \mathcal{R}_0[a](t)
\end{align}
with the remainder estimate
$$
|\mathcal{R}_0[a](t)| \lesssim \left( T^{\frac{1}{2}+c} + T^{\Theta}\frac{\log|\log T|}{|\log T|}\|a(\cdot)-a(T)\|_{\Theta,l-1} + [a]_{\gamma,m,l-1} \right) \frac{(T-t)^{m+(1+\omega)\gamma}}{|\log(T-t)|^{l}}.
$$
\end{prop}

Recall from \eqref{Aug17-3} that the operator $\mathcal{B}_0$ satisfies the asymptotic expansion
$$
\mathcal{B}_0[\mu] = \int_{-T}^{t - \mu^{2}_{0}(t)} \frac{\dot{\mu}(s)}{t-s}  ds + O(\|\dot{\mu}\|_{\infty}).
$$
Proposition \ref{prop7.1} establishes an approximate inverse operator $\mathcal{P}$ for $\mathcal{B}_0$. Specifically, if the source term $a$ satisfies \eqref{Aug17-1}, then the function $\mu := \mathcal{P}[a]$ satisfies
$$
\mathcal{B}_0[\mu] = a + \mathcal{R}_0[a] \quad \text{in }~ [-T,T],
$$
where $\mathcal{R}_0[a]$ is a negligible remainder. Moreover, the analysis in \cite{Davila-del-Pino-Wei2020} yields the decomposition
\begin{align}\label{Aug17-4}
\mathcal{P}[a] = \mu_{0,\kappa} + \mathcal{P}_1[a],
\end{align}
in which $\mu_{0,\kappa}$ is given by \eqref{Sep27-1} with $\kappa = \kappa[a] \in \mathbb{R}$,
and the perturbation $\mu_1=\mathcal{P}_1[a]$ satisfies the norm estimate
\begin{align}\label{Aug17-13}
\|\mu_1\|_{*,3-\iota} \lesssim |\log T|^{1-\iota} \log^{2}(|\log T|),\quad 0 < \iota < 1.
\end{align}

This analysis thereby clarifies the motivation behind the functional framework introduced in Subsection \ref{Sep26-space} for the parameter $\mu$.

\subsection{Proof of Theorem \ref{theorem1}: the fixed-point argument}

The proof of Theorem \ref{theorem1}, building on the preparatory work of Subsections \ref{Sep26-space}, \ref{Sep27-nonlinear} and \ref{Sep26-parameters}, rests on a fixed-point formulation of the inner-outer gluing system. The first step is to summarize the coupled system to be solved, along with the necessary constraints on the parameters.

\subsubsection{The inner-outer gluing system and parameter constraints}

Following the analysis of the modulation parameters in Subsection \ref{Sep26-parameters}, we transform the inner-outer problems \eqref{Jul-inner-problem} and \eqref{Jul-outer-problem} into finding a tuple $(\psi,\phi^{0},\phi^{\perp},\mu,\xi)$ that satisfy the following \textit{inner-outer gluing system}

\begin{align}\label{Aug16-3}
\begin{cases}
\partial_t\psi_{1}=\Delta_{x} \psi_1+\mathcal{G}_{1}[\phi,\psi,\mu,\xi] & \text { in }~ \mathbb{R}_{+}^{4} \times(0, T), \\
 -\frac{d \psi_1}{d x_{4}}\left(\tilde{x}, 0, t\right)=0 & \text { in }~ \mathbb{R}^{3} \times(0, T),
\end{cases}
\end{align}

\begin{align}\label{Aug16-3-2}
\begin{cases}
\partial_t\psi_{2}=\Delta_{x} \psi_2 & \text { in }~ \mathbb{R}_{+}^{4} \times(0, T), \\
 -\frac{d \psi_2}{d x_{4}}\left(\tilde{x}, 0, t\right)=\mathcal{G}_{2}[\phi,\psi,\mu,\xi] & \text { in }~ \mathbb{R}^{3} \times(0, T),
\end{cases}
\end{align}

\begin{align}\label{Aug16-4}
\begin{cases}
\mu^{2} \phi^{0}_{t}=\Delta_{y} \phi^{0}+\mathcal{H}^{0}_{1}[\phi,\psi,\mu,\xi] & \text { in } ~  B_{2R}^{+} \times(0, T),  \\
-\frac{d \phi^{0}}{d y_{4}}(\tilde{y}, 0, t)=2U(\tilde{y}, 0) \phi^{0}(\tilde{y}, 0, t)+\mathcal{H}^{0}_{2}[\phi,\psi,\mu,\xi]\\
 \qquad\qquad\qquad\quad+c^{0}[\mathcal{H}^0_1,\mathcal{H}^0_2]Z_4(\tilde{y},0)
 & \text { in }~  (\partial B_{2R}^{+}\cap\mathbb{R}^{3})\times(0, T),\\
\phi^{0}(\cdot,0) = 0 & \ \text{in }~ B_{2R(0)}^+,
\end{cases}
\end{align}

\begin{align}\label{Aug16-5}
\begin{cases}
\mu^{2} \phi^{\perp}_{t}=\Delta_{y} \phi^{\perp}+\mathcal{H}^{\perp}_{1}[\phi,\psi,\mu,\xi] & \text { in } ~  B_{2R}^{+} \times(0, T),  \\
-\frac{d \phi^{\perp}}{d y_{4}}(\tilde{y}, 0, t)=2U(\tilde{y}, 0) \phi^{\perp}(\tilde{y}, 0, t)+\mathcal{H}^{\perp}_{2}[\phi,\psi,\mu,\xi]\\
\qquad\qquad\qquad\quad + \sum\limits_{\ell=1}^{3}c^{\ell}[\mathcal{H}^{\perp}_1,\mathcal{H}^{\perp}_2]Z_{\ell}(\tilde{y}, 0) & \text { in }~  (\partial B_{2R}^{+}\cap\mathbb{R}^{3})\times(0, T),\\
\phi^{\perp}(\cdot,0) = 0 & \ \text{in } ~B_{2R(0)}^+,
\end{cases}
\end{align}

\begin{align}\label{Aug16-6}
    c^{0}[\mathcal{H}^0_1,\mathcal{H}^0_2](t) = 0 \quad \text{for all} \ t \in (0,T),
\end{align}
\begin{align}\label{Aug16-7}
   c^{\ell}[\mathcal{H}^{\perp}_1,\mathcal{H}^{\perp}_2](t) = 0 \quad \text{for all} \ t \in (0,T),~\ell=1,2,3.
\end{align}
Here, $\mathcal{G}_1$ and $\mathcal{G}_2$ are defined in Section \ref{25-inner-outer}, while $\mathcal{H}_i^{0}$ and $\mathcal{H}_i^{\perp}$ ($i=1,2$) denote the projections of $\mathcal{H}_i$ onto different modes. It follows directly that if $(\psi,\phi^{0},\phi^{\perp},\mu,\xi)$ satisfies the system \eqref{Aug16-3}-\eqref{Aug16-7}, then setting $\psi = \psi_1 + \psi_2$ and $\phi = \phi^{0} + \phi^{\perp}$ yields solutions to the original inner-outer problems \eqref{Jul-inner-problem} and \eqref{Jul-outer-problem}, thereby constructing the desired blow-up solution.

We now reformulate this system as a fixed-point problem in the product space $$
\mathcal{X}=X_{\phi^{0}}\times  X_{\phi^{\perp}}\times X_{\psi_1}\times X_{\psi_2}\times\mathbb{R}\times X_{\mu}\times X_{\xi},$$
which is defined in Subsection \ref{Sep26-space}.

Before constructing the fixed-point operator, we summarize the key parameters and function space norms that have been introduced throughout our analysis.
\begin{itemize}
    \item The inner domain radius is $R(t)=\mu_0^{-\beta}(t)$ with $\beta\in(0,1/2)$.

    \item For the inner problem, the solution $\phi^{0}$ (mode $0$) is measured in the norm $\|\cdot\|_{\sigma_1,\nu}$, while $\phi^{\perp}$ (higher modes) is measured in $\|\cdot\|_{1+\sigma_2,\nu}$ as defined in \eqref{Aug17-5}, with $\sigma_1,\sigma_2,\nu\in(0,1)$.

    \item For the outer problem, the solution $\psi_1$ is measured in $\|\cdot\|_{*}^{(1)}$ (see \eqref{Aug5-3}) and its right-hand side in $\|\cdot\|_{**}^{(1)}$ (see \eqref{Aug5-2}). Similarly, $\psi_2$ and its right-hand side are measured in $\|\cdot\|_{*}^{(2)}$ and $\|\cdot\|_{**}^{(2)}$ (see \eqref{outer5}, \eqref{outer6}), with parameters $\nu,\alpha,\nu_{2},\gamma\in(0,1)$.

    \item For the parameters in Proposition \ref{prop7.1}, to apply Proposition \ref{prop7.1} for the scaling parameter $\mu(t)$, we set
    $$
    \Theta = \nu - 1 + \alpha\beta, \quad
    m = \nu - 2 - \gamma + \beta(2 + \alpha), \quad
    l < 1 + 2m,
    $$
    and require $\beta>\frac{1-\omega}{2}$ to ensure $m+(1+\omega)\gamma>\Theta$, where  $\omega\in(0,1/2)$ describes the remainder $\mathcal{R}_{\omega}$.
\end{itemize}

To establish the desired estimates for the inner-outer system \eqref{Aug16-3}-\eqref{Aug16-5}, we rely on the computations in Subsection \ref{Sep27-nonlinear}, which require the constraints \eqref{Aug18-1}, \eqref{Aug18-2}, \eqref{Aug18-3}, \eqref{Aug18-4} and \eqref{Aug18-5}.
These constraints can be satisfied, as exemplified by the choice
$$
\beta =0.26, \quad
\alpha=0.97,\quad  \sigma_1= 0.98, \quad  \nu_1=0.8,\quad
\nu=0.98, \quad \sigma_2=
\nu_{2}=0.005.
$$

This establishes the self-consistency of the framework and the validity of the estimates for a suitable range of parameters.

\subsubsection{Construction of the fixed-point operator}

We aim to solve the inner-outer gluing system within a closed ball $\mathcal{B}\subset \mathcal{X}$, where elements $(\phi^{0},\phi^{\perp},\psi_1,\psi_2,\kappa,\mu_{1},\xi^{1})$ satisfy
\begin{align}\label{Aug17-7}
\begin{cases}
\|\phi^{0}\|_{\sigma_1,\nu}+\|\phi^{\perp}\|_{1+\sigma_2,\nu}\leq 1, \\
\|\psi_1\|_{*}^{(1)}+\|\psi_2\|_{*}^{(2)}\leq 1, \\
|\kappa-\kappa_{0}|\leq|\log T|^{-1/2}, \\
\|\mu_{1}\|_{\ast,3-\iota}\leq C|\log T|^{1-\iota}\log^{2}(|\log T|), \\
\|\xi\|_{G}\leq 1,
\end{cases}
\end{align}
with $\kappa_{0}=Z^{*}_{0}(0)$ and $C > 0$ is a large and fixed constant.

We now construct the fixed-point operator $\mathcal{F}: \mathcal{B} \subset \mathcal{X} \to \mathcal{X}$ for the entire inner-outer gluing system. The main strategy is to apply the linear solution operators from the previous sections to define a map whose fixed point satisfies the nonlinear system \eqref{Aug16-3}--\eqref{Aug16-7}. Specifically, for a given tuple of approximate solutions and parameters, we
\begin{enumerate}
    \item solve the inner problems \eqref{Aug16-4} and \eqref{Aug16-5} for $\phi^0$ and $\phi^\perp$ using the linear theory in Proposition \ref{proposition4.2},
    \item solve the outer problems \eqref{Aug16-3} and \eqref{Aug16-3-2} for $\psi_1$ and $\psi_2$ using the linear theories in Propositions \ref{25Aug-prop5.1} and \ref{25-prop5.5},
    \item update the parameters $\mu$ and $\xi$ by solving the reduced equations \eqref{Aug15-4} and  \eqref{Aug15-1} derived from the orthogonality conditions.
\end{enumerate}
This procedure defines the components of our fixed-point operator $\mathcal{F}:\mathcal{B} \rightarrow \mathcal{X}$ as follows:
\begin{align}\label{Aug17-6}
 \mathcal{F}(v)=(\mathcal{F}_{\phi^{0}}(v),\mathcal{F}_{\phi^{\perp}}(v),\mathcal{F}_{\psi_1}(v),\mathcal{F}_{\psi_2}(v),\mathcal{F}_{\kappa}(v),\mathcal{F}_{\mu_{1}}(v),\mathcal{F}_{\xi}(v)),
\end{align}
where
\begin{align*}
&\mathcal{F}_{\phi^{0}}(\phi^{0},\phi^{\perp},\psi_1,\psi_2,\kappa,\mu_{1},\xi^{1})
\\&= \mathcal{T}_{0}(\mathcal{H}^{0}_1[\mu,\xi,\psi_1,\psi_2,Z^{*}],\mathcal{H}^{0}_2[\mu,\xi,\psi_1,\psi_2,Z^{*}]+c^{0}[\mathcal{H}^0_1,\mathcal{H}^0_2]Z_4 ),
\\&\mathcal{F}_{\phi^{\perp}}(\phi^{0},\phi^{\perp},\psi_1,\psi_2,\kappa,\mu_{1},\xi^{1})
\\&= \mathcal{T}_{\perp}\left(\mathcal{H}_1^{\perp}[\mu,\xi,\psi_1,\psi_2,Z^{*}],\mathcal{H}_2^{\perp}[\mu,\xi,\psi_1,\psi_2,Z^{*}]+\sum\limits_{\ell=1}^{3}c^{\ell}[\mathcal{H}^{\perp}_1,\mathcal{H}^{\perp}_2]Z_{\ell}\right),
\end{align*}
\begin{align*}
\mathcal{F}_{\psi_1}(\phi^{0},\phi^{\perp},\psi_1,\psi_2,\kappa,\mu_{1},\xi^{1})& = \mathcal{T}_{\psi_1}\left(\mathcal{G}_1(\phi^{0}+\phi^{\perp},\psi_1,\psi_2,Z^{*},\mu,\xi)\right),
\\
\mathcal{F}_{\psi_2}(\phi^{0},\phi^{\perp},\psi_1,\psi_2,\kappa,\mu_{1},\xi^{1}) &= \mathcal{T}_{\psi_2}\left(\mathcal{G}_2(\phi^{0}+\phi^{\perp},\psi_1,\psi_2,Z^{*},\mu,\xi)\right) ,
\\
\mathcal{F}_{\kappa}(\phi^{0},\phi^{\perp},\psi_1,\psi_2,\kappa,\mu_{1},\xi^{1}) &= \kappa\left[a^{0}[\mu,\xi,\psi_1,\psi_2,Z^{*}]\right] ,
\\
\mathcal{F}_{\mu_{1}}(\phi^{0},\phi^{\perp},\psi_1,\psi_2,\kappa,\mu_{1},\xi^{1})& = \mathcal{P}_{1}\left[a^{0}[\mu,\xi,\psi_1,\psi_2,Z^{*}]\right] ,
\\
\mathcal{F}_{\xi}(\phi^{0},\phi^{\perp},\psi_1,\psi_2,\kappa,\mu_{1},\xi^{1}) &= \Xi(\phi^{0},\phi^{\perp},\psi_1,\psi_2,\mu,\xi).
\end{align*}
Here, $\mathcal{T}_{0}$ and $\mathcal{T}_{\perp}$ are the solution operators for the inner problems (Proposition \ref{proposition4.2}), $\mathcal{T}_{\psi_1}$ and $\mathcal{T}_{\psi_2}$ solve the outer problems (Propositions \ref{25Aug-prop5.1} and \ref{25-prop5.5}), while $\kappa[a]$, $\mathcal{P}_{1}$, and $\Xi$ are the parameter-updating operators (which determine the new values of $\mu$ and $\xi$ from the orthogonality conditions) defined in Proposition \ref{prop7.1}, \eqref{Aug17-4}, and \eqref{Aug15-2}, respectively.

A fixed point of $\mathcal{F}$, i.e., a point $v$ such that $v=\mathcal{F}(v)$, corresponds to a solution of the coupled system \eqref{Aug16-3}-\eqref{Aug16-7}.
We now verify that the operator $\mathcal{F}$ defined in \eqref{Aug17-6} with components
\begin{align}\label{Aug17-14}
\mathcal{F} = \left( \mathcal{F}_{\phi^{0}}, \mathcal{F}_{\phi^{\perp}}, \mathcal{F}_{\psi_1}, \mathcal{F}_{\psi_2}, \mathcal{F}_{\kappa}, \mathcal{F}_{\mu_{1}}, \mathcal{F}_{\xi} \right)
\end{align}
satisfies the conditions of Schauder's fixed-point theorem in the closed ball $\mathcal{B}$ defined by \eqref{Aug17-7}.

Combining the estimates \eqref{Aug4-5}, \eqref{Aug4-5-2}, \eqref{Aug17-8}, \eqref{Aug17-9}, \eqref{Aug17-10}, \eqref{Aug17-11}, \eqref{Aug17-12} and \eqref{Aug17-13} with Propositions \ref{proposition4.2}, \ref{25Aug-prop5.1}, \ref{25-prop5.5}, and \ref{prop7.1}, we obtain for any $(\phi^{0},\phi^{\perp},\psi_1,\psi_2,\kappa,\mu_{1},\xi^{1}) \in \mathcal{B}$,
\begin{align}\label{Aug17-15}
\begin{cases}
\|\mathcal{F}_{\phi^{0}}(\phi^{0},\phi^{\perp},\psi_1,\psi_2,\kappa,\mu_{1},\xi^{1})\|_{\sigma_1,\nu}\leq CT^{\epsilon},\\
\|\mathcal{F}_{\phi^{\perp}}(\phi^{0},\phi^{\perp},\psi_1,\psi_2,\kappa,\mu_{1},\xi^{1})\|_{1+\sigma_2,\nu}\leq CT^{\epsilon},\\
\|\mathcal{F}_{\psi_1}(\phi^{0},\phi^{\perp},\psi_1,\psi_2,\kappa,\mu_{1},\xi^{1})\|_{*}^{(1)}\leq CT^{\epsilon},\\
\|\mathcal{F}_{\psi_2}(\phi^{0},\phi^{\perp},\psi_1,\psi_2,\kappa,\mu_{1},\xi^{1})\|_{*}^{(2)}\leq CT^{\epsilon},\\
|\mathcal{F}_{\kappa}(\phi^{0},\phi^{\perp},\psi_1,\psi_2,\kappa,\mu_{1},\xi^{1})-\kappa_{0}|\leq C|\log T|^{-1},\\
\|\mathcal{F}_{\mu_{1}}(\phi^{0},\phi^{\perp},\psi_1,\psi_2,\kappa,\mu_{1},\xi^{1})\|_{\ast,3-\epsilon}\leq C|\log T|^{1-\epsilon}\log^{2}(|\log T|),\\
\|\mathcal{F}_{\xi}(\phi^{0},\phi^{\perp},\psi_1,\psi_2,\kappa,\mu_{1},\xi^{1})\|_{G}\leq CT^{\epsilon},
\end{cases}
\end{align}
where $C>0$ is independent of $T$, and $\epsilon>0$ is sufficiently small. This establishes that $\mathcal{F}$ maps $\mathcal{B}$ into itself.

The compactness of $\mathcal{F}$ follows from continuity and the Arzel$\grave{a}$-Ascoli theorem. Indeed, for parameters $\beta,\alpha,\sigma_1,\sigma_2,\nu,\nu_1,\nu_2$ slightly perturbed while maintaining \eqref{Aug18-1}, \eqref{Aug18-2}, \eqref{Aug18-3}, \eqref{Aug18-4} and \eqref{Aug18-5}, the estimates \eqref{Aug17-15} persist with norms adapted to the new parameters, while $\mathcal{B}$ remains unchanged.  More precisely, for fixed $\nu^{\prime},\sigma_1^{\prime}$ close to $\nu,\sigma_1$, one may show that if $(\phi^{0},\phi^{\perp},\psi_1,\psi_2,\kappa,\mu_{1},\xi^{1})\in\mathcal{B}$, then
$$
\|\mathcal{F}_{\phi^{0}}(\phi^{0},\phi^{\perp},\psi_1,\psi_2,\kappa,\mu_{1},\xi^{1})\|_{\sigma_1^{\prime},\nu^{\prime}}\leq CT^{\epsilon^{\prime}}.
$$

Moreover, for $\nu^{\prime}>\nu$ satisfying
$$
\nu^{\prime} - \beta\left(2 - \frac{\sigma_1^{\prime}}{2}\right) > \nu - \beta\left(2 - \frac{\sigma_1}{2}\right),
$$
the embedding $\|\cdot\|_{\sigma_1^{\prime},\nu^{\prime}} \hookrightarrow \|\cdot\|_{\sigma_1,\nu}$ is compact. Since any bounded sequence in the $\|\cdot\|_{\sigma_1^{\prime},\nu^{\prime}}$-norm admits a subsequence converging in the $\|\cdot\|_{\sigma_1,\nu}$-norm, a standard diagonal argument using Arzelà-Ascoli theorem establishes the compactness of $\mathcal{F}$. Similar reasoning applies to the remaining components.

\subsubsection{Completion of the proof}
Since $\mathcal{B}$ is a nonempty, closed, bounded, and convex subset of $\mathcal{X}$, and $\mathcal{F}: \mathcal{B} \to \mathcal{B}$ is compact, Schauder's fixed-point theorem guarantees the existence of a fixed point $v^* = \mathcal{F}(v^*)$. This corresponds to a solution of the system \eqref{Aug16-3}-\eqref{Aug16-7}, which yields the desired blow-up solution for the case $k = 1$ via the decompositions $\psi = \psi_1 + \psi_2$ and $\phi = \phi^0 + \phi^\perp$.

The multi-bubble case follows an analogous construction. We consider the modified ansatz
$$
u(x,t) = u^{*}(x,t) + \sum_{j=1}^{k} \mu_{j}^{-1}(t) \eta_{R(t)}(y_{j}) \phi_{j}(y_{j},t) + Z^{*}(x,t) + \psi(x,t),~y_{j} = \frac{x - \xi_{j}(t)}{\mu_{j}(t)},
$$
where the leading profile is given by
$$
u^{*}(x,t) = \sum_{j=1}^{k} U_{\mu_{j}(t),\xi_{j}(t)} + \Psi_{0j}(x,t)
$$
with $\Psi_{0j}$ defined as in \eqref{Jul26-5}  under the replacements $\mu$ to $\mu_{j}$, $\xi$ to $\xi_{j}$.
This leads to one outer problem and $k$ inner problems with estimates identical to the single-bubble case. The corresponding system of fixed-point problems can be solved in the same manner, and we omit the details.

\section*{Acknowledgements}
J.C. Wei is supported by National Key R\&D Program of China 2022YFA1005602, and Hong Kong General Research Fund “New frontiers in singularity formations of nonlinear partial differential equations”. Y. Zheng is supported by NSF of China (No. 12171355), he is very grateful to Dr. Qidi Zhang for useful discussions.

\bigskip

%\bibliography{RefDatabase}{}

\bibliographystyle{plain}

\end{document}